\documentclass[12pt]{amsart}
\usepackage{latexsym, amsmath, amssymb, amsthm, amsfonts, amstext, amscd, amsxtra, amsbsy, mathrsfs, stmaryrd, mathdots, bbold}
\usepackage[all]{xy}
\usepackage{tikz}
\usepackage[makeroom]{cancel}
\usepackage{bbm}
\usepackage{color}
\usepackage{hyperref}
\usepackage{mathtools}

\allowdisplaybreaks

\newtheorem{theorem}{Theorem}[section]
\newtheorem{lemma}[theorem]{Lemma}
\newtheorem{prop}[theorem]{Proposition}

\theoremstyle{definition}

\newtheorem{claim}[theorem]{Claim}

\newtheorem{remark}[theorem]{Remark}

\numberwithin{equation}{section}

\newcommand{\nc}{\newcommand}
\nc{\rnc}{\renewcommand}
\nc{\bb}[1]{{\mathbb #1}}
\nc{\bbA}{\bb{A}}\nc{\bbB}{\bb{B}}\nc{\bbC}{\bb{C}}\nc{\bbD}{\bb{D}}
\nc{\bbE}{\bb{E}}\nc{\bbF}{\bb{F}}\nc{\bbG}{\bb{G}}\nc{\bbH}{\bb{H}}
\nc{\bbI}{\bb{I}}\nc{\bbJ}{\bb{J}}\nc{\bbK}{\bb{K}}\nc{\bbL}{\bb{L}}
\nc{\bbM}{\bb{M}}\nc{\bbN}{\bb{N}}\nc{\bbO}{\bb{O}}\nc{\bbP}{\bb{P}}
\nc{\bbQ}{\bb{Q}}\nc{\bbR}{\bb{R}}\nc{\bbS}{\bb{S}}\nc{\bbT}{\bb{T}}
\nc{\bbU}{\bb{U}}\nc{\bbV}{\bb{V}}\nc{\bbW}{\bb{W}}\nc{\bbX}{\bb{X}}
\nc{\bbY}{\bb{Y}}\nc{\bbZ}{\bb{Z}}
\nc{\mbf}[1]{{\mathbf #1}}
\nc{\bfA}{\mbf{A}}\nc{\bfB}{\mbf{B}}\nc{\bfC}{\mbf{C}}\nc{\bfD}{\mbf{D}}
\nc{\bfE}{\mbf{E}}\nc{\bfF}{\mbf{F}}\nc{\bfG}{\mbf{G}}\nc{\bfH}{\mbf{H}}
\nc{\bfI}{\mbf{I}}\nc{\bfJ}{\mbf{J}}\nc{\bfK}{\mbf{K}}\nc{\bfL}{\mbf{L}}
\nc{\bfM}{\mbf{M}}\nc{\bfN}{\mbf{N}}\nc{\bfO}{\mbf{O}}\nc{\bfP}{\mbf{P}}
\nc{\bfQ}{\mbf{Q}}\nc{\bfR}{\mbf{R}}\nc{\bfS}{\mbf{S}}\nc{\bfT}{\mbf{T}}
\nc{\bfU}{\mbf{U}}\nc{\bfV}{\mbf{V}}\nc{\bfW}{\mbf{W}}\nc{\bfX}{\mbf{X}}
\nc{\bfY}{\mbf{Y}}\nc{\bfZ}{\mbf{Z}}
\nc{\bfa}{\mbf{a}}\nc{\bfb}{\mbf{b}}\nc{\bfc}{\mbf{c}}\nc{\bfd}{\mbf{d}}
\nc{\bfe}{\mbf{e}}\nc{\bff}{\mbf{f}}\nc{\bfg}{\mbf{g}}\nc{\bfh}{\mbf{h}}
\nc{\bfi}{\mbf{i}}\nc{\bfj}{\mbf{j}}\nc{\bfk}{\mbf{k}}\nc{\bfl}{\mbf{l}}
\nc{\bfm}{\mbf{m}}\nc{\bfn}{\mbf{n}}\nc{\bfo}{\mbf{o}}\nc{\bfp}{\mbf{p}}
\nc{\bfq}{\mbf{q}}\nc{\bfr}{\mbf{r}}\nc{\bfs}{\mbf{s}}\nc{\bft}{\mbf{t}}
\nc{\bfu}{\mbf{u}}\nc{\bfv}{\mbf{v}}\nc{\bfw}{\mbf{w}}\nc{\bfx}{\mbf{x}}
\nc{\bfy}{\mbf{y}}\nc{\bfz}{\mbf{z}}

\nc{\mcal}[1]{{\mathcal #1}}
\nc{\calA}{\mcal{A}}\nc{\calB}{\mcal{B}}\nc{\calC}{\mcal{C}}\nc{\calD}{\mcal{D}}
\nc{\calE}{\mcal{E}} \nc{\calF}{\mcal{F}}\nc{\calG}{\mcal{G}}\nc{\calH}{\mcal{H}}
\nc{\calI}{\mcal{I}}\nc{\calJ}{\mcal{J}}\nc{\calK}{\mcal{K}}\nc{\calL}{\mcal{L}}
\nc{\calM}{\mcal{M}}\nc{\calN}{\mcal{N}}\nc{\calO}{\mcal{O}}\nc{\calP}{\mcal{P}}
\nc{\calQ}{\mcal{Q}}\nc{\calR}{\mcal{R}}\nc{\calS}{\mcal{S}}\nc{\calT}{\mcal{T}}
\nc{\calU}{\mcal{U}}\nc{\calV}{\mcal{V}}\nc{\calW}{\mcal{W}}\nc{\calX}{\mcal{X}}
\nc{\calY}{\mcal{Y}}\nc{\calZ}{\mcal{Z}}
\nc{\fA}{\frak{A}}\nc{\fB}{\frak{B}}\nc{\fC}{\frak{C}} \nc{\fD}{\frak{D}}
\nc{\fE}{\frak{E}}\nc{\fF}{\frak{F}}\nc{\fG}{\frak{G}}\nc{\fH}{\frak{H}}
\nc{\fI}{\frak{I}}\nc{\fJ}{\frak{J}}\nc{\fK}{\frak{K}}\nc{\fL}{\frak{L}}
\nc{\fM}{\frak{M}}\nc{\fN}{\frak{N}}\nc{\fO}{\frak{O}}\nc{\fP}{\frak{P}}
\nc{\fQ}{\frak{Q}}\nc{\fR}{\frak{R}}\nc{\fS}{\frak{S}}\nc{\fT}{\frak{T}}
\nc{\fU}{\frak{U}}\nc{\fV}{\frak{V}}\nc{\fW}{\frak{W}}\nc{\fX}{\frak{X}}
\nc{\fY}{\frak{Y}}\nc{\fZ}{\frak{Z}}
\nc{\fa}{\frak{a}}\nc{\fb}{\frak{b}}\nc{\fc}{\frak{c}} \nc{\fd}{\frak{d}}
\nc{\fe}{\frak{e}}\nc{\fFf}{\frak{f}}\nc{\fg}{\frak{g}}\nc{\fh}{\frak{h}}
\nc{\fri}{\frak{i}}\nc{\fj}{\frak{j}}\nc{\fk}{\frak{k}}\nc{\fl}{\frak{l}}
\nc{\fm}{\frak{m}}\nc{\fn}{\frak{n}}\nc{\fo}{\frak{o}}\nc{\fp}{\frak{p}}
\nc{\fq}{\frak{q}}\nc{\fr}{\frak{r}}\nc{\fs}{\frak{s}}\nc{\ft}{\frak{t}}
\nc{\fu}{\frak{u}}\nc{\fv}{\frak{v}}\nc{\fw}{\frak{w}}\nc{\fx}{\frak{x}}
\nc{\fy}{\frak{y}}\nc{\fz}{\frak{z}}

\newcommand{\Hom}{\mathrm{Hom}}
\newcommand{\mbb}{\mathbb}
\newcommand{\mscr}{\mathscr}
\newcommand{\mrm}{\mathrm}

\def\ro{\text{ro}}
\def\co{\text{co}}

\def\diag{\text{diag}}

\newcommand{\Gr}{\mrm Gr}

\newcommand{\Sym}{\operatorname{Sym}\nolimits}
\newcommand{\Res}{\operatorname{Res}\nolimits}

\newcommand{\End}{\mrm{End}}

\def \C{{\mathbb C}}

\newcommand{\qbinom}[2]{\begin{bmatrix} #1\\#2 \end{bmatrix} }

\newcommand{\sll}{\mathfrak{sl}}

\newcommand{\K}{\mathbb K}

\newcommand{\arxiv}[1]{\href{http://arxiv.org/abs/#1}{\tt arXiv:\nolinkurl{#1}}}

\newcommand{\Z}{\mathbb Z}

\newcommand{\bB}{{\mathbf B }} 
\newcommand{\bDel}{\boldsymbol{\Delta}}
\newcommand{\bTh}{\boldsymbol{\Theta}}
\newcommand{\bH}{\mathbf H}

\def \fg{\mathfrak{g}}

\def \I{\mathbb{I}}

\newcommand{\tUi}{\widetilde{{\mathbf U}}^\imath}

\def \fn{\mathfrak{n}}
\def \fh{\mathfrak{h}}
\def \fu{\mathfrak{u}}
\def \fv{\mathfrak{v}}
\def \fa{\mathfrak{a}}
\def \fk{\mathfrak{k}}

\DeclareMathOperator{\IC}{IC}
\DeclareMathOperator{\pt}{pt}
\DeclareMathOperator{\loc}{loc}

\DeclareMathOperator{\Span}{Span}
\DeclareMathOperator{\Sp}{Sp}
\DeclareMathOperator{\AIII}{AIII}
\DeclareMathOperator{\Ext}{Ext}

\title[Affine $\mathrm{i}$quantum groups and  Steinberg varieties of type C]{Affine $\mathrm{i}$quantum groups and Steinberg varieties \\ of type C}
\author{Changjian Su}
 \address{Yau Mathematical Sciences Center, Tsinghua University, Beijing, 100084, China}\email{changjiansu@mail.tsinghua.edu.cn}

\author{Weiqiang Wang}
 \address{Department of Mathematics, University of Virginia,
Charlottesville, VA 22903, USA}\email{ww9c@virginia.edu}

\date{} 

\begin{document}


\begin{abstract}
    We provide a geometric realization of the quasi-split affine iquantum group of type AIII$_{2n-1}^{(\tau)}$ in terms of equivariant K-groups of non-connected Steinberg varieties of type C. This uses a new Drinfeld type presentation of this affine iquantum group which admits very nontrivial Serre relations. We then construct \`a la Springer a family of finite-dimensional standard modules and irreducible modules of this iquantum group, and provide a composition multiplicity formula of the standard modules. 
\end{abstract}
\maketitle
 \setcounter{tocdepth}{1}
\tableofcontents

\section{Introduction}

\subsection{}

A powerful geometric approach to realize quantum algebras is via an equivariant K-theory construction, where the quantum parameter $q$ is realized through a $\C^*$-action \cite{Lus85}. On the other hand, one could add that algebras which can be realized geometrically are often basic and important. For example, it is by now well known that the equivariant K-group of the Steinberg variety provides a geometric construction of affine Hecke algebras and a classification of their irreducible representations when $q$ is not a root of $1$ \cite{KL87} (also see \cite{CG97, X07}). 

This idea has been adapted to realize Drinfeld-Jimbo affine quantum groups since then. The affine quantum group for $\sll_N$ or $\mathfrak{gl}_N$ was realized via the equivariant K-group of the Steinberg variety associated to the $N$-step flag variety (of type $A$) by Ginzburg and Vasserot \cite{GV93, V98}, see also \cite{G91, V93} for an earlier cohomological 
version. Subsequently this was generalized by Nakajima \cite{Na01} to realize affine quantum groups of type ADE (and more general quantum loop algebras) in the setting of quiver varieties. In these constructions, it is essential to use Drinfeld's current presentation of affine quantum groups (see \cite{Dr87, Beck, Da12}).

In recent years, it has been fruitful to take the viewpoint that iquantum groups (sometimes written as $\imath$quantum groups) arising from quantum symmetric pairs are a natural vast generalization of Drinfeld-Jimbo quantum groups; see the survey \cite{W23} for a long list of of basic constructions on quantum groups which have been (partially) generalized in this direction. In contrast to the uniqueness of the rank one quantum group (i.e. quantum $\sll_2$), there exist different rank one iquantum groups which led to rich and complicated higher rank cases. In \cite{BKLW18}, the iquantum group of quasi-split type AIII was realized by working with $N$-step isotropic flags of type $B$ over finite fields, generalizing the geometric realization of the quantum group of type $A$ by Beilinson-Lusztig-MacPherson \cite{BLM}. This has been generalized in \cite{FLW+} in the affine flag variety of type $C$ setting to realize the affine iquantum group of quasi-split type AIII (cf. \cite{Lu99} for affine type $A$).

\subsection{}

In this paper, we study the equivariant K-group of the Steinberg variety 
\[
\calZ =T^*\calF \times_{\calN} T^*\calF
\]
associated with the $N$-step isotropic flag variety $\calF$ for $G= \text{Sp}(2d)$ with $N=2n$ even. We establish an algebra homomorphism 
\[
\Psi:\tUi\longrightarrow \underline{K}^{G\times\mathbb{C}^{*}}(\calZ)
\]
from the affine iquantum group $\tUi$ of type AIII to the equivariant K-group $\underline{K}^{G\times\mathbb{C}^{*}}(\calZ)$. Here and below, the underline notation denotes a localized version. Then we show that the specialization of $\Psi$ at $a=(s,t)$ in \eqref{eq:Psia}, $\Psi_a:\tUi_t \rightarrow K^{G\times\mathbb{C}^{*}}(\calZ)_a$, is surjective for $q=t$ not a root of unity. 

The well-known convolution algebra formalism further allows us to construct a family of finite-dimensional standard modules and describe the composition multiplicities of the standard modules in terms of dimensions of intersection cohomology groups. (Almost nothing was known before about the finite-dimensional representation theory of $\tUi$ for lack of a triangular decomposition of $\tUi$.)

\subsection{}

The construction of the homomorphism $\Psi$ uses the Drinfeld type current presentation of an affine iquantum group which exhibits the twisted loop algebra structure. The Drinfeld presentation for affine iquantum groups of quasi-split types including type AIII$_{2n-1}^{(\tau)}$ used here has only become available recently in \cite{LWZ24} (extending the earlier Drinfeld presentations of affine iquantum groups of split types \cite{LW21, Z22}). 

Our construction and computation have been greatly facilitated by and in turn extend the earlier type $A$ work of Vasserot \cite{V98}. To achieve our goal, we need to deal with several complications which did not arise in the type $A$ setting though, as discussed in some detail below. 

There is a diagonal $G$-orbit on $\calF \times \calF$ which is not closed. At various points of the paper, we need to apply localization to deal with this non-closed orbit. 

In contrast to the Drinfeld presentation of affine quantum groups, the Drinfeld presentation of $\tUi$ is complicated since it contains different (affine) rank one and rank two relations and some current Serre relations admit quite nontrivial lower order terms. All the current relations for $\tUi$ can be formulated in generating function form, and the formulation of the lower order part of a current Serre relation is by no means unique; cf. \cite{LW21, Z22}. 

To verify that $\Psi$ is a homomorphism, we identify suitable $G\times \C^*$-equivariant sheaves over $\calZ$ with the current generators of $\tUi$ and then we must verify all the corresponding current relations of $\tUi$ for these sheaf classes. The verification of several relations requires lengthy computations; the proof of one particular Serre relation turns out to be especially difficult, and we have to formulate a new form of the current Serre relation in order to match with the geometric computation. We remark that a homomorphism from the {\em finite type} AIII iquantum group to $\underline{K}^{G\times\mathbb{C}^{*}}(\calZ)$ was constructed in \cite{FMX22}, where they did not need to deal algebraically and geometrically with the complexity of the current generators and their relations. 

We show that the homomorphism $\Psi$ by specializing $q$ to a non-root of unity is surjective onto the corresponding Borel-Moore homology group; see \eqref{UtoH}. This allows us to apply the convolution algebra yoga \cite[Chapter 8]{CG97} to construct geometrically a family of finite-dimensional standard and simple $\tUi$-modules. We further provide a composition multiplicity formula for these standard modules in term of intersection cohomology groups.

\subsection{}

It is a natural open question to work out the equivariant K-group of the Steinberg variety $\calZ$ associated with the $N$-step isotropic flag variety $\calF$ of type $C$ with $N=2n+1$ odd. The relevant affine iquantum group is expected to be of quasi-split type AIII$_{2n}^{(\tau)}$, whose Drinfeld presentation is given in \cite{LWZ23} for $n=1$ and very recently in \cite{LPWZ25} for general $n$. 

One expects that the equivariant K-theoretic realization of affine iquantum groups goes through if one uses $N$-step flag variety of type $B$ (instead of type $C$ here). We expect that again the affine iquantum group of type AIII will arise this way, possibly with different parameters. One can also try using the type $D$ flags with the orthogonal group action.  These are achieved in \cite{LSX}. 

If one applies equivariant cohomology instead of equivariant K-group to the Steinberg variety $\calZ$, one is expected to realize the twisted Yangians of quasi-split type AIII in its Drinfeld presentation. The Drinfeld presentation for this class of twisted Yangians has been obtained in \cite{LZ25} (building on \cite{LWZ25degen}). It will be interesting to compare with the twisted Yangians constructed in \cite{Li19} using the stable envelope \`a la Maulik-Okounkov \cite{MO19}. See \cite{nakajima2025instantons, su2026twisted} for recent developments in this direction.

It remains to be seen if one can extend the equivariant K-theoretic construction to ``iquiver varieties" (see \cite{Li19} for an approach to such varieties) to realize general quasi-split affine iquantum groups of ADE types \cite{LWZ24}; note that even the Borel-Moore homology of these varieties has not been studied in depth (see \cite{Li21}). 

\subsection{}

The paper is organized as follows. 

In the preliminary Section~\ref{sec:Stein}, we describe the diagonal $G$-orbits on $\calF \times \calF$ following \cite{BKLW18} in terms of matrix data. We set up some basics on the equivariant K-theory of the Steinberg variety $\calZ$ and its convolution product. We establish a generating set for the convolution algebra $\underline{K}^{G\times\mathbb{C}^{*}}(\calZ)$ in Section~ \ref{sec:generators}; see Theorem~\ref{thm:generators}.

In Section~\ref{sec:iquantum}, we review the Drinfeld presentation of the affine iquantum group $\tUi$ from \cite{LWZ24}, and formulate new variants of several relations including a Serre relations and hence a new variant of Drinfeld presentation of $\tUi$. We then formulate (see Theorem~\ref{thm:polyrepeven}) a polynomial representation of $\tUi$ on \begin{align}
 \label{eq:MP}
\underline{K}^{G\times \bbC^*}(T^*\calF)\simeq \underline{\bfP}; 
\end{align}
see \eqref{eq:P} for $\bfP$, which is a direct sum of Laurent polynomial algebras. The verification of various relations for Theorem~\ref{thm:polyrepeven} is tedious, and this will be carried out in the subsequent Section~\ref{sec:proof} and Appendix~\ref{App:A}. 

In Section~\ref{sec:K}, we construct classes of $G\times \bbC^*$-equivariant sheaves on $\calZ$,  denoted by $\hat \bTh_{i,k}, \mathscr{E}_{i,k}, \mathscr{B}_{n,k}$, and $\mathscr{F}_{i,k}$. We verify that when acting on \eqref{eq:MP} these classes are realized by the operators 
$\hat\bTh_{i,k}, \hat E_{i,k}, \hat B_{n,k}, \hat F_{i,k}$ on the polynomial representation $\underline{\bfP}$ in Section~\ref{sec:iquantum}. In this way, we have obtained a $\bbC(q)$-algebra homomorphism
$\Psi:\tUi\rightarrow \underline{K}^{G\times\mathbb{C}^{*}}(\calZ).$

Finally, we apply the convolution algebra formalism to our setting in Section~\ref{sec:modules}. We construct a family of finite-dimensional standard modules and irreducible modules of $\tUi$ this way and give the composition multiplicities of a standard module in terms of dimensions of certain IC cohomology groups.

\vspace{2mm}
\noindent {\bf Acknowledgement.} 
We thank Weinan Zhang for very helpful discussions and suggestions regarding the Serre relations in the Drinfeld presentation. We also thank Li Luo, Zheming Xu, and Yang Yang for useful discussions. Finally, we thank an anonymous referee for meticulous reading and helpful suggestions. 
CS is supported by National Key R\&D Program of China  (No. 2024YFA1014700). WW is partially supported by DMS--2401351. 

\section{Convolution algebra of the Steinberg variety}
\label{sec:Stein}

In this section, we review the basics on convolution product in equivariant K-theory  and apply to the non-connected Steinberg variety $\calZ$ of type $C$ (cf. \cite{CG97, V98}). We recall the classification of the $G$-diagonal orbits on $\calF \times \calF$ from \cite{BKLW18}. The new main result of this section is a generating set for the convolution algebra $\underline{K}^{G\times\bbC^*}(\calZ)$. 

\subsection{Convolution in equivariant K-theory}
\label{subsec:convolution}

For a connected complex reductive algebraic group $G$ and a quasi-projective $G$-variety $X$, let $K^G(X)$ denote the complexified $G$-equivariant K-group of $X$, see \cite{CG97}. 
If $X= \{\pt\}$, $K^G(\pt) = R(G)$, the complexified representation ring of $G$.

Given three smooth $G$-varieties $M_1,~M_2,~M_3$, let
$$p_{ij}:M_1 \times M_2 \times M_3 \rightarrow M_i\times M_j$$
be the obvious projection maps.
Let $Z_{12}\subseteq M_1 \times M_2 $ and $Z_{23}\subseteq M_2 \times M_3 $ be $G$-stable closed subvarieties.
We denote
 \begin{equation*}
   Z_{12}\circ Z_{23}= p_{13}(p_{12}^{-1}(Z_{12})\cap p_{23}^{-1}(Z_{23})).
 \end{equation*}
If the restriction of $p_{13}$ to $p_{12}^{-1}(Z_{12})\cap p_{23}^{-1}(Z_{23})$ is a proper map,
then we define the convolution product as follows:
\begin{equation*}
\begin{split}
\star: \ \ K^{G}(Z_{12})\otimes K^{G}(Z_{23})&\longrightarrow K^{G}(Z_{12} \circ Z_{23}), \\
   \mscr{F}_{1} \otimes \mscr{F}_{2} &\mapsto p_{{13}{\ast}}(p_{12}^{\ast}\mscr{F}_{1} \otimes p_{23}^{\ast}\mscr{F}_{2}),
\end{split}
\end{equation*}
where all the functors here and below are understood to be derived.

Let $F_i$ ($i=1,2$) be smooth $G$-varieties, $M_i=T^*F_i$, and $\pi_i$ denote the projections $M_i\rightarrow F_i$. The torus $\bbC^*$ acts on $M_i$ by $z\cdot (x,\xi)=(x,z^{-2}\xi)$, where $x\in F_i$ and $\xi\in T_x^*F_i$. By definition, $K^{\bbC^*}(\pt)=\bbC[q,q^{-1}]$, where $q$ corresponds to the standard representation of $\bbC^*$. Let $\calO\subset F_1\times F_2$ be a smooth $G$-variety, and $Z_\calO$ denote the conormal bundle $T_\calO^*(F_1\times F_2)\subset M_1\times M_2$. Suppose the projection $Z_\calO\rightarrow M_1$ is proper and the projections $p_{i,\calO}:\calO\rightarrow F_i$ are smooth fibrations with $p_{1,\calO}$ being proper. By the Thom isomorphism,  $K^{G\times\bbC^*}(Z_\calO)\simeq K^{G\times\bbC^*}(\calO)$. Therefore, any $\mscr{K}\in K^{G\times\bbC^*}(\calO)$ defines an $R(G\times\bbC^*)$-modules homomorphism $\rho_{\mscr K}:K^{G\times\bbC^*}(M_2)\rightarrow K^{G\times \bbC^*}(M_1)$ by convolution. We have the following useful formula. 

\begin{lemma}\cite[Corollary 4]{V98}\label{lem:convolution}
For any $\mscr{K}\in K^{G\times\bbC^*}(\calO)$ and $\mscr F\in K^{G\times \bbC^*}(F_2)$,
\[\rho_\mscr K(\pi_2^*\mscr F)=\pi_1^* p_{1,\calO*}\bigg(\bigwedge\nolimits_{q^2}T_{p_{1,\calO}}\otimes p_{2,\calO}^*\mscr F\otimes \mscr K\bigg),\]
where $T_{p_{1,\calO}}$ is the relative tangent sheaf along the fibers of $p_{1,\calO}$, and $\bigwedge_{q^2}T_{p_{1,\calO}}=\sum_i (-q^2)^i \bigwedge^i T_{p_{1,\calO}}$. 
\end{lemma}

For the computations, we will use frequently the localization formula in equivariant K-theory. Let $T\subset G$ be a maximal torus, and let $X$ be a smooth projective variety such that the torus fixed point set $X^T$ is finite. First of all, we have $K^G(X)\simeq K^T(X)^W$, where $W$ is the Weyl group. Let $\pi:X\rightarrow \pt$ be the structure morphism. Then for any $\mscr F\in K^T(X)$, we have the following localization formula \cite{CG97}
\begin{equation}\label{equ:localization}
    \pi_*(\mscr F)=\sum_{x\in X^T}\frac{\mscr F|_x}{\bigwedge^\bullet (T_x^*X)}\in K^T(\pt),
\end{equation}
where $\mscr F|_x\in K^T(\pt)$ is the pullback of $\mscr F$ to the fixed point $x\in X^T$, and $\bigwedge^\bullet T_x^*X=\sum_i(-1)^i\wedge^i(T^*_xX)=\prod_{\mu_i}(1-e^{\mu_i})\in K^T(\pt)$ with the product over all the torus weights $\{\mu_i\}$ in the $T$-vector space $T_x^*X$. 

Using the localization formula, we can define the pushforward morphism for non-proper maps as follows. Let $p:X\rightarrow Y$ be a morphism between smooth $G$-varieties such that $p|_{X^T}:X^T\rightarrow Y^T$ is a proper morphism. Let $S$ be the multiplicative subset in $K^T(\pt)$ generated by $\{1-e^\mu\}$ where $\mu$ runs through the torus weights in the normal bundle of $X^T$ inside $X$. Let $K^T(X)_{\loc}$ be the localization of $K^T(X)$ at $S$. Then we can define $p_*:K^T(X)_{\loc}\rightarrow K^T(Y)_{\loc}$ 
by the above localization formula, see \cite[Section 2.3.10]{O17}. Finally, to get the formula for the $G$-equivariant K-theory, we just take the $W$-invariants.

\subsection{Partial flag varieties of type C}
\label{subsec:flag}

Let $V :=\bbC^{2d}$ with a non-degenerate skew-symmetric bilinear form $(-,-)$ given by the matrix $\begin{pmatrix}
    0&I_d\\-I_d&0
\end{pmatrix}.$ 
Throughout the paper, we set
\[
G = \Sp(V) \quad\text{ and } \quad N=2n,
\]
for a fixed positive integer $n$. Let 
\begin{align}  \label{eq:Lamdacd}
\Lambda_{\fc,d}= \big\{ \bfv=(v_i) \in \mbb N^{N} \mid   v_i = v_{N+1-i},\quad \textstyle \sum_{i=1}^{n} v_i = d \big\}.
\end{align}
For any subspace $W\subseteq V$, let $W^{\perp} = \{x \in V\mid (x,y)=0, \ \forall y\in W\}$.
For any $\mbf v \in \Lambda_{\fc,d}$, define 
$$
\calF_{\bfv}=\{ F=( 0=V_0 \subset V_1 \subset \cdots \subset V_{N}=V)\ \mid \  V_i=V_{N-i}^{\perp} ,\ \text{dim}( V_{i}/V_{i-1})= v_i,\  \forall i \}. 
$$
The natural $G$-action on $V$ induces a natural transitive action of $G$ on $\calF_{\mbf v}$, and thus 
\begin{align}
 \label{eq:F}
\calF = \bigsqcup _{\bfv \in \Lambda_{\fc,d}} \calF_{\bfv} 
\end{align}
is a $G$-variety called the {\em $N$-step partial flag variety}. Let $P_{\bfv} $ be the stabilizer of $F \in \calF_{\bfv}$ in $G$.
We have $G/P_{\bfv}\simeq \calF_{\mbf v}$.

Let $W_{\fc} = \mathbb{Z}_2^{d}\rtimes S_{d}$ be the Weyl group of type $C_d$, which
has a natural action on the set $\{1,2, \ldots,2d\}$. For  $\bfv =(v_1,\ldots, v_N) \in \Lambda_{\fc,d}$ and $0\le i\le N$, we set $\bar{v}_{i}=\sum_{r=1}^{i} v_r $ with $\bar{v}_{0}=0$. Denote the intervals, for $1\le i \le N$,
\begin{align}  
   \label{eq:vi}
[\bfv]_i =[1+\bar{v}_{i-1},\bar{v}_{i}]\subseteq \mbb N. 
\end{align}
Then 
\[
[\bfv] := \big([\bfv]_1,[\bfv]_2,\ldots,[\bfv]_N \big)
\]
forms a set partition of $\{1,2,\ldots,2d\}$, and
\begin{align}
[\mbf v]^{\fc} := \big([\bfv]_1,[\bfv]_2, \ldots,[\mbf v]_n \big)
\end{align}
forms a set partition of $\{1,2,\ldots,d\}$.
Define a subgroup $W_{[\mbf v]^{\fc}}$ of $W_{\fc}$ by
\begin{align*}
W_{[\mbf v]^{\fc}} :=
S_{[\mbf v]_1}\times S_{[\mbf v]_2} \times \cdots \times S_{[\mbf v]_n},
\end{align*}
where  $S_{[\mbf v]_i}$ is the subgroup of $S_d$ consisting of all permutations which preserve $[\mbf v]_i$.

Denote
\begin{align}
\Xi_d=\Big\{A=(a_{ij}) \in {\rm Mat}_{N\times N}(\mbb N)\ | \ \sum_{i,j}a_{ij}=2d,\ a_{ij}=a_{N+1-i,N+1-j},\ \forall \  i, j \Big\}.
\end{align}
In other words, any matrix $A \in \Xi_d$ is fixed by the rotation of 180 degrees.
To any matrix $A \in \Xi_d$, we associate the following set partition of $\{1,2,\ldots,2d\}$: 
    \begin{equation*}
      [A]=\big([A]_{11}, \ldots, [A]_{N1}, [A]_{12},\ldots, [A]_{NN}\big),
    \end{equation*}
where $[A]_{ij} = \Big[\sum\limits_{(h,k)< (i,j)}a_{hk}+1 , \sum\limits_{(h,k)<(i,j)}a_{hk}+a_{ij} \Big]\subseteq \mbb N$,
and $<$ is the lexicographical order defined by
\begin{equation}\label{equ:rightorder}
    (h,k) < (i,j) \Leftrightarrow k<j  ~\text{or}~ (k = j ~\text{and}~  h<i).
\end{equation} 
Define $[A]^{\fc}$ to be the following partition of the set $\{1,2,\cdots,d\}$,
\[[A]^{\fc} =
\big([A]_{11}, \ldots, [A]_{N1}, [A]_{12},\ldots, [A]_{N2}, \ldots, [A]_{Nn}\big),
\]
and define a type A parabolic subgroup of the Weyl group $W_\fc$ by
\begin{align}
W_{[A]^{\fc}}=
S_{[A]_{11}}\times\dotsb \times S_{[A]_{N1}}\times S_{[A]_{12}} \times\dotsb\times S_{[A]_{N2}} \times \dotsb \times S_{[A]_{Nn}}.
\end{align}

\subsection{The $G$-orbits in $\calF \times \calF$}

For a matrix $A \in \Xi_d$, we define its row vector and column vector
\begin{equation*}
  \ro(A) = \Big(\sum_j a_{ij}\Big)_{i=1, 2,\ldots, N}\in\Lambda_{\fc,d},\quad\ {\rm and}\  \co(A) = \Big(\sum_i a_{ij}\Big)_{j=1,2, \ldots, N}\in\Lambda_{\fc,d}.
\end{equation*}
For any $\bfv, \mathbf{w} \in \Lambda_{\fc,d}$, let
\begin{align}
\Xi_d(\bfv,\mathbf{w}) = \{ A \in \Xi_d \mid \ro(A) = \mbf v,\  \co(A) = \mbf w \}.
\end{align}

For a pair of flags $(F,F')\in \calF_{\bfv}\times \calF_{\mathbf{w}}$, define an $N\times N$ matrix $A=(a_{ij})$ with
\begin{equation}\label{equ:a}
	a_{ij}=\dim \frac{V_i\cap V_j'}{V_{i-1}\cap V_j'+V_{i}\cap V_{j-1}'}.
\end{equation}
It has been shown in \cite[Section 6]{BKLW18} (as a generalization of \cite{BLM} in type A) that sending the $G$-orbit of $(F,F')$ to the matrix $A$ gives a bijection 
\begin{align}
    \label{bij}
    \{G\text{-orbits in } \calF_{\bfv}\times \calF_{\mathbf{w}} \} \stackrel{1:1}{\longleftrightarrow} \Xi_d(\bfv,\mathbf{w}).
\end{align} 
For any $A\in \Xi_d$, let $\calO_A$ denote the corresponding $G$-orbit on $\calF\times \calF$. If $A=\diag(v_i)$ is a diagonal matrix, then the orbit $\calO_A$ is just the diagonal copy of $\calF_\bfv$ inside $\calF_\bfv\times\calF_\bfv$. We can define an order $\preceq$ on $\Xi_d$ as follows.
For any $A=(a_{ij}), B=(b_{ij})\in \Xi_d$, $A \preceq B $ if and only if
\begin{align}\label{equ:order}
\ro(A) = \ro(B),\ \co(A)=\co(B),\ {\rm and}\ \sum_{r\leq i; s\geq j} a_{rs} \leq \sum_{r\leq i; s\geq j} b_{rs}, \  \forall  i<j.
\end{align}
Then $\calO_A\subset \overline{\calO_B}$ if and only if $A \preceq B$.

Let $E_{ij}$ be the standard $N\times N$ matrix unit with 1 at $(i,j)$th entry. For $a\ge 0$, define 
\begin{align} \label{eq:Eijv}
E_{ij}^{\theta} :=  E_{ij} + E_{N+1-i,N+1-j}, \qquad
E_{ij}^{\theta}(\bfv,a):= \diag(\bfv) + a E_{ij}^{\theta}, 
\end{align}
where  $\bfv = (v_1, \ldots, v_N)$ such that $v_i = v_{N+1 -i}$ and $\sum\limits_{i=1}^N v_i= 2d-2a$. Let $\bfe_i$ ($1\leq i\leq N$) be the standard basis for $\C^N$ (viewed as row vectors). For $1\leq i\leq n-1$, the $G$-orbits $\calO_{E^\theta_{i,i+1}(\bfv,a)}$ and $\calO_{E^\theta_{i+1,i}(\bfv,a)}$ on $\calF \times \calF$ are closed, and they are given by
\begin{align*}
\mathcal{O}_{E_{i,i+1}^{\theta}(\bfv,a)}&= \bigg\{(F,F')\mid 
\substack{F=(V_k)_{0\le k \le N} \in \mcal{F}_{\bfv + a\mathbf{e}_i  + a\mathbf{e}_{N + 1-i} }\\ F'=(V'_k)_{0\le k \le N} \in \mcal{F}_{\bfv + a\mathbf{e}_{i+1}  + a\mathbf{e}_{N -i}}}, V'_i\stackrel{a}{\subset} V_i, V_k=V'_k \textit{ if } k\neq i,N-i\bigg\},
\end{align*} 
and 
\begin{align*}
\mathcal{O}_{E_{i+1,i}^{\theta}(\bfv,a)}&=\bigg\{(F,F')\mid 
\substack{F=(V_k)_{0\le k \le N} 
\in \mcal{F}_{\bfv + a\mathbf{e}_{i+1}  + a\mathbf{e}_{N-i} }\\ F'=(V'_k)_{0\le k \le N}
\in \mcal{F}_{\bfv + a\mathbf{e}_i  + a\mathbf{e}_{N+1-i}}}, V_i\stackrel{a}{\subset} V'_i, V_k=V'_k \textit{ if } k\neq i,N-i\bigg\}.
\end{align*}
The notation $V'_i\stackrel{a}{\subset} V_i$ above means the inclusion $V'_i\subset V_i$ of codimension $a$. However, the orbit $\mathcal{O}_{E_{n,n+1}^{\theta}(\bfv,a)}$ is not closed for $a\geq 1$. We have
\begin{align} \label{equ:orbitn}
	\mathcal{O}_{E_{n,n+1}^{\theta}(\bfv,1)}&= \bigg\{(F,F')\mid 
 \substack{F=(V_k)_{0\le k \le N}\in \mcal{F}_{\bfv + \mathbf{e}_n  + \mathbf{e}_{n+1} }\\ F'=(V'_k)_{0\le k \le N}\in \mcal{F}_{\bfv + \mathbf{e}_{n}  + \mathbf{e}_{n+1}}},
V_n\cap V'_n \stackrel{1}{\subset}V_n, V_k=V'_k \textit{ if } k\neq n\bigg\}, 
\end{align}
whose closure contains a diagonal copy of $\calF_{\bfv + \bfe_n  + \bfe_{n+1}}$.

For $1\leq i<j\leq 3$, let $p_{ij}:\calF^3\rightarrow \calF^2$ be the projection to the $i,j$ factors. For any $A,B \in \Xi_d$ with $\co(A)=\ro(B)$, the set $p_{13}\big(p_{12}^{-1}(\calO_A)\cap p_{23}^{-1}(\calO_B)\big)$ is stable with respect to the diagonal $G$-action. Hence, we can decompose it into $G$-orbits
\[p_{13}\big(p_{12}^{-1}(\calO_A)\cap p_{23}^{-1}(\calO_B)\big)=\bigcup_{C\in \Xi_d(A,B)}\calO_C\]
for some set $\Xi_d(A,B)$. Then the same argument as in \cite[Proposition 5(a),(c)]{V98} gives the following proposition.
\begin{prop}\label{prop:orbit}
Let $A,B \in \Xi_d$.
\begin{enumerate}
    \item $\mathcal{O}_A \subseteq  \overline{\mathcal{O}}_B$ if and only if $A \preceq B$.
   
    \item There exists a unique matrix $A\circ B$ in $\Xi_d(A,B)$ maximal with respect to $\preceq$. 
\end{enumerate}
\end{prop}

\begin{remark}
    \cite[Proposition 5(b),(d)]{V98} are incorrect, which are used in the proof for \cite[Proposition 10]{V98}. Here are counterexamples. Consider $n=d=2$, and $A=B=s=:\begin{pmatrix}
        0 & 1\\
        1 & 0
    \end{pmatrix}$. Then $p_{13}(p_{12}^{-1}(\calO_A)\cap p_{12}^{-1}(\calO_B))$ is the disjoint union of $\calO_{I_2}$ and $\calO_s$, where $I_2$ is the identity matrix. But computing using the 3-arrays $T$, gives us $\mathbf{M}(A,B)=\{I_2\}$. Hence, Part (b) of \cite[Proposition 5]{V98} fails. Similarly, if we take $A'=I_2\prec A$ and $B'=B$, then $A\circ B=s=A'\circ B'$. Thus, Part (d) also fails.
    
    The statement in \cite[Proposition 10]{V98} remains correct, and variants of Claim~ \ref{claim:AB} and Proposition \ref{prop:AB} below (used to prove Theorem \ref{thm:generators}) can be used to fix its proof.
\end{remark}

\subsection{A pushforward formula} \label{subsec:push}

Let 
\[
\mathbf{R} = \mathbb{C}[x_1^{\pm 1},x_2^{\pm 1},\cdots,x_d^{\pm 1}]\simeq K^G(G/B), 
\]
where $B$ is a Borel subgroup of $G$. We shall define a natural action of the Weyl group $W_{\fc}$ on $\mbf R$ below, which by restriction leads to actions of the subgroups $W_{[\mbf v]^{\fc}}$ and $W_{[A]^{\fc}}$ on $\mbf R$. The action of $\sigma \in S_{d}$ on $\mbf R$ is given by
$$\sigma: \mathbf{R} \longrightarrow \mathbf{R}, 
\qquad
f(x_1^{\pm 1},x_2^{\pm 1},\cdots,x_d^{\pm 1}) \mapsto  f(x_{\sigma(1)}^{\pm 1},x_{\sigma(2)}^{\pm 1},\cdots,x_{\sigma(d)}^{\pm 1}).$$
For any $m\in [1,d]$, the generator $\iota_m$ in the m-th copy of $\bbZ_2$ in $\bbZ_2^d$ acts on $\mathbf{R}$ by
\begin{equation*}
  \begin{split}
\iota_m: \mathbf{R} &\longrightarrow \mathbf{R},\\
 f(x_1^{\pm 1},\cdots,x_{m-1}^{\pm 1}, x_m^{\pm 1},x_{m+1}^{\pm 1},\cdots x_d^{\pm 1}) &\mapsto  f(x_1^{\pm 1},\cdots,x_{m-1}^{\pm 1}, x_m^{\mp 1},x_{m+1}^{\pm 1},\cdots x_d^{\pm 1}).
  \end{split}
\end{equation*}
For two subgroups of the Weyl group $W_1\subset W_2\subset W_\fc$, we define a map
 $$W_2/W_1 : \mathbf{R}^{W_1} \longrightarrow \mathbf{R}^{W_2}, \qquad f \mapsto \sum\limits_{\sigma \in W_2/W_1}\sigma(f).$$

To simplify the notations, we shall denote $\mathbf{R}^{W_{[\bfv]^{\fc}}}$ by $\mathbf{R}^{[\bfv]^{\fc}}$, 
and $\mathbf{R}^{W_{[A]^{\fc}}}$ by $\mathbf{R}^{[A]^{\fc}}$. 

\begin{prop}\label{prop:pushforward}
Let $\bfv, \bfv_1, \bfv_2 \in \Lambda_{\fc,d}$, and $A \in \Xi_d(\bfv_1,\bfv_2)$.

{\rm (a)} There exist  $\mathbb{C}$-algebra isomorphisms $K^{G}(\calF_{\bfv}) \cong \mathbf{R}^{[\bfv]^{\fc}} \text{and} \ K^{G}(\mathcal{O}_A) \cong \mathbf{R}^{[A]^{\fc}}$.

{\rm (b)} The first projection map $p_{1,A} : \mathcal{O}_A \rightarrow \calF_{\bfv_1}$ is a smooth fibration.
Moreover, if $\mathcal{O}_A $ is closed, then the direct image morphism $p_{{1,A} *}$ is given by
$$p_{{1,A} *}(\mathscr{F}) =  W_{[\bfv_1]^{\fc}}/W_{[A]^{\fc}}\bigg(\frac{\mathscr{F}}{\bigwedge(T_{p_{1,A}}^*)}\bigg),$$
where $T_{p_{1,A}}^*$ is the relative cotangent bundle and $\bigwedge(T_{p_{1,A}}^*) = \sum_i(-1)^i\bigwedge^iT^*_{p_{1,A}}$.

\noindent (Statement (b) holds for the localized K-groups if the orbit $\calO_A$ is not closed.)
\end{prop}

\begin{proof}
Let $\{\epsilon_i\mid 1\leq i\leq 2d\}$ be the standard basis of $V$. Let $F$ be the flag such that for $1\leq i\leq n$, $V_i=\Span\{\epsilon_j\mid j\in \sqcup_{k\leq i}[\bfv]_k\}$ and $V_{N-i}=V_i^\perp$. Let $P_F$ be the stabilizer of the flag $F$ inside $G$. Then the Weyl group of the Levi subgroup of $P_F$ is $W_{[\bfv]^{\fc}}$. Thus,
\[
K^{G}(\calF_{\bfv}) = K^G(G/P_F)\cong \mathbf{R}^{[\bfv]^{\fc}},
\]
and the isomorphism is given by 
\[
K^{G}(\calF_{\bfv})\ni \calF \mapsto \calF|_{\pt_F}\in \bfR^{[\bfv]^\fc},
\]
where $\calF|_{\pt_F}$ denotes the restriction of $\calF$ to the torus fixed point $\pt_F\in G/{P_F}$.

Define a bijection of $[1,2d]$ by
\[\varphi(k)=\begin{cases}k, &\textit{ if } 1\leq k\leq d;\\
3d+1-k, & \textit{ if } d+1\leq k\leq 2d.\end{cases}\]
Given $A \in \Xi_d(\bfv_1,\bfv_2)$, define a decomposition $V=\bigoplus_{1\leq i,j\leq N}V_{ij}$ by
\[V_{ij}=\Span\{\epsilon_{\varphi(k)}\mid k\in [A]_{ij}\}.\]
One sees that 
\begin{equation}\label{equ:d}
\epsilon_k\in V_{ij} \textit{ if and only if }\epsilon_{k+d}\in V_{N+1-i,N+1-j}.
\end{equation}
Let $F_\circ =(V_k)_{1\leq k\leq N}$ be the flag with $V_k=\bigoplus\limits_{1\leq i\leq k,1\leq j\leq N}V_{ij}$, and let $F'_\circ =(V'_k)_{1\leq k\leq N}$ be the flag with $V'_k=\bigoplus\limits_{1\leq j\leq k,1\leq i\leq N}V_{ij}$. Then by the observation \eqref{equ:d}, $F_\circ \in \calF_{\ro(A)}$, $F'_\circ \in\calF_{\co(A)}$ and $(F_\circ,F'_\circ)\in \calO_A$. Let $P_{F_\circ}$ (respectively, $P_{F'_\circ}$) be the stablizer of $F_\circ$ (respectively, $F'_\circ$). Then $\calO_A\simeq G/(P_{F_\circ}\cap P_{F'_\circ})$. The reductive part of $P_{F_\circ}\cap P_{F'_\circ}$ is isomorphic to 
\[
GL(V_{11})\times \dotsb\times GL(V_{N1})\times GL(V_{12})\times \dotsb \times GL(V_{Nn}),
\]
with Weyl group $W_{[A]^\fc}$. Thus we have an isomorphism
\[ K^{G}(\mathcal{O}_A) \cong \mathbf{R}^{[A]^{\fc}},\]
given by
$\calF\mapsto \calF|_{\pt_A}$, where we have denoted $\pt_A =(F_\circ,F'_\circ) \in \calO_A$. 
This finishes the proof of Part (a). Part (b) follows directly from the localization theorem.
\end{proof}

\subsection{The Steinberg variety}
\label{steinberg variety of type c}

Let $\calM = T^*\calF $ be the cotangent bundle of the $N$-step partial flag variety $\calF$ in \eqref{eq:F}.
More explicitly, $\calM$ can be written as
\begin{equation*}
  \label{tangent}
  \calM = T^*\calF = \{(F,x) \in \calF \times \mathfrak{sp}_{2d}\mid x(F_i) \subseteq F_{i-1}, ~ \forall i \}
  \subseteq \calF \times \mathfrak{sp}_{2d},
\end{equation*}
where $\mathfrak{sp}_{2d}$ is the Lie algebra of $G$. There is a natural $G$-action on $\calM$ induced by the $G$-action on $\calF$. Define a $G\times \bbC^*$-action on $\calM$ by
$$ (g,z) \cdot  (F,x) = (gF, z^{-2}gxg^{-1}), \quad \forall (g,z) \in G\times \bbC^*.$$ Let $q$ be the equivariant parameter for the $\bbC^*$-action. Then $K_{\bbC^*}(\pt)\simeq \bbC[q,q^{-1}]$.

Let $\calN$ be the nilpotent variety of the Lie algebra $\mathfrak{sp}_{2d}$. By definition, we have $\calM \subseteq \calF \times \calN$. The projection map $\pi: \calM\rightarrow \calN, (F,x) \mapsto x$, is proper and $G\times \bbC^*$-equivariant.
Let 
\begin{align}
    \label{eq:Z}
    \calZ:=\calM\times_{\calN} \calM\subset \calM\times\calM
\end{align}
be the (generalized) Steinberg variety of type C. The group $G\times \bbC^* $ acts on $\calZ$ diagonally, and $(K^{G\times \bbC^*}(\calZ), \star)$ is a convolution $\bbA$-algebra with unit, see \S\ref{subsec:convolution}. By \cite{CG97},
\begin{equation}\label{equ:relative}
    K^{G\times \bbC^*}(\calZ)\simeq K^{G\times \bbC^*}(\calM\times\calM; \calZ),
\end{equation}
where the right-hand side is the Grothendieck group of the derived category of $G\times\bbC^*$-equivariant complexes of vector bundles on $\calM\times\calM$, which are exact outside $\calZ$.

We briefly digress and introduce some notations. Given a $\C$-vector space $U$, via base change we define the $\C[q,q^{-1}]$-module and the $\C(q)$-vector space, respectively:
\begin{align}   \label{eq:Uq}
U[q,q^{-1}] :=\C[q,q^{-1}] \otimes_\C U,
\qquad
U(q) :=\C(q) \otimes_\C U. 
\end{align}
 In this way, we make sense of notations $\mbf R^{[\mbf v]^{\fc}}[q,q^{-1}]$, $\mbf R^{[\mbf v]^{\fc}}(q)$, $\bfR^{[A]^{\fc}}[q,q^{-1}]$, $\bfR^{[A]^{\fc}}(q)$, and so on. 

Via convolution, the algebra $K^{G\times \bbC^*}(\mcal{Z})$ acts on 
\begin{align}  \label{KMR}
    K^{G\times \bbC^*}(\calM)\simeq \bigoplus_{\mbf v\in \Lambda_{\fc,d}} \mbf R^{[\mbf v]^{\fc}}[q,q^{-1}].
\end{align}
Following \cite[Claim 7.6.7]{CG97}, we have the following lemma.

\begin{lemma}
\label{lemma:faithful}
We have a faithful representation of $K^{G\times \bbC^*}(\calZ)$ on $K^{G\times \bbC^*}(\calM)$.
\end{lemma}

For any $A\in \Xi_d(\bfv,\bfw)$, let $\calZ_A:=T^*_{\calO_A}(\calF_{\bfv}\times \calF_{\bfw})$ be the conormal bundle of the diagonal $G$-orbit $\calO_A$. Then all the irreducible components of $\calZ$ are of the form $\overline{\calZ_A}$. Let $\calZ_{\bfv,\bfw}=\bigcup_{A\in \Xi_d(\bfv,\bfw)}\overline{\calZ_A}$. For any $A\in \Xi_d$, let $\calZ_{\preceq A}:=\bigcup_{B\preceq A}\calZ_B$, which is a closed subvariety of $\calZ$. The induced maps $K^{G\times\bbC^*}(\calZ_{\preceq A}) \rightarrow K^{G\times\bbC^*}(\calZ)$ are injective and their images form a filtration of $K^{G\times\bbC^*}(\calZ)$ indexed by $\Xi_d$. Moreover, Proposition \ref{prop:orbit} implies $\calZ_{\preceq A}\circ \calZ_{\preceq B}\subset \calZ_{\preceq A\circ B}$. Thus, $K^{G\times\bbC^*}(\calZ_{\preceq A})\star K^{G\times\bbC^*}(\calZ_{\preceq B})\subset K^{G\times\bbC^*}(\calZ_{\preceq A\circ B})$. By \cite{DR14}, the open immersion $\calZ_A\hookrightarrow \calZ_{\preceq A}$ gives rise to the following short exact sequence
\[0\longrightarrow K^{G\times\bbC^*}(\calZ_{\prec A})\longrightarrow K^{G\times\bbC^*}(\calZ_{\preceq A})\longrightarrow K^{G\times\bbC^*}(\calZ_{A})\longrightarrow 0,\]
where $\calZ_{\prec A}:=\calZ_{\preceq A}\setminus \calZ_A$. Thus, the following maps
\begin{align*}
K^{G\times\bbC^*}(\calZ_{\preceq A})
\twoheadrightarrow K^{G\times\bbC^*}(\calZ_{\preceq A})/K^{G\times\bbC^*}(\calZ_{\prec A}) & \xrightarrow{\sim}K^{G\times\bbC^*}(\calZ_{A})
\\
&\xrightarrow{\sim}K^{G\times\bbC^*}(\calO_{A})\xrightarrow{\sim}\bfR^{[A]^{\fc}}[q,q^{-1}]
\end{align*}
identifies the associated graded of $K^{G\times \bbC^*}(\calZ)$ with $\bigoplus_{A\in \Xi_d}\bfR^{[A]^{\fc}}[q,q^{-1}]$.

\section{A generating set for the convolution algebra}
\label{sec:generators}

In this section, we establish a generating set for the convolution algebra $\underline{K}^{G\times\mathbb{C}^{*}}(\calZ)$.

\subsection{A reduction}

Due to the non-closedness of the orbit $\calO_{E_{n,n+1}^\theta}$ given in \eqref{equ:orbitn}, we need to work with the localized equivariant K-group of the Steinberg variety. 

Note that for any $G$-variety $X$, $K^T(X)\simeq K^G(X)\otimes_{K^G(\pt)}K^T(\pt)$ and $K^G(X)\simeq K^T(X)^W$ (see \cite[Theorem 6.1.22]{CG97}). The $T\times\bbC^*$-weights of the cotangent spaces to the torus fixed points in $\calM\times\calM$ are of the form $e^\alpha$ and $q^{2}e^\alpha$, where $\alpha$ lies in the root system $R$ of $G$. Hence, we are led to the following definition of the localized $G$-equivariant K-group modules.
For any $K^{G\times\bbC^*}(\pt)$-module $U$ (e.g. $K^{G\times \bbC^*}(X)$ for some $G\times\bbC^*$-variety $X$), let
\begin{align*}\label{equ:localizedmodule}
    \underline{U} &:= \bigg(U\otimes_{K^{G\times\bbC^*}(\pt)}K^{T}(\pt)(q) \Big[\frac{1}{1-e^\alpha},\frac{1}{1-q^2e^\alpha}, \forall \alpha \in R \Big]\bigg)^W\\
    &= U\otimes_{\bbC}\bigg(\bbC(q)\Big[\frac{1}{1-e^\alpha},\frac{1}{1-q^2e^\alpha}, \forall \alpha \in R \Big]\bigg)^W
\end{align*}
denote the localized module, where the $W$-action on $U$ is trivial; this is actually a $\C(q)$-vector space and should be viewed as a shorthand for the notation $\underline{U}(q)$ which we choose not to use. We denote 
\begin{align*}
\underline{\bfR}^{[A]^{\fc}} 
    &= \bfR^{[A]^{\fc}}(q)\otimes_{\bbC}\bigg(\bbC \Big[\frac{1}{1-e^\alpha},\frac{1}{1-q^2e^\alpha}, \forall \alpha \in R \Big]\bigg)^W.
\end{align*} 
For any $A\in \Xi_d$, we can consider the image of $\underline{K}^{G\times\bbC^*}(\calZ_A)\simeq \underline{\bfR}^{[A]^{\fc}}$ in $\underline{K}^{G\times\bbC^*}(\calZ)$ via pushforward by localization, even if the orbit $\calO_A\subset \calF\times\calF$ is not closed. Moreover, the convolution is also well defined by localization:
\[
\star:\underline{\bfR}^{[A]^{\fc}}\otimes \underline{\bfR}^{[B]^{\fc}}\rightarrow  \underline{K}^{G\times\bbC^*}(\calZ_{\preceq A\circ B}).
\]

\begin{lemma}\label{lem:Egenerators}
    For $1\leq i\leq N-1$ and $a\in \bbZ_{\geq 1}$, the image of $\underline{K}^{G\times\bbC^*}(\calZ_{E^\theta_{i,i+1}(\bfv,a)})$ in $(\underline{K}^{G\times \bbC^*}(\calZ), \star)$ can be obtained via convolution from images of $\underline{K}^{G\times\bbC^*}(\calZ_{E^\theta_{i,i+1}(\bfv',k)})$ for various $\bfv'$ and $0\leq k\leq a-1$.

    \noindent (The localization is only needed for $i=n$, as the orbit $\calO_{E^\theta_{i,i+1}(\bfv,a)}$ is closed for $i\neq n$.)
\end{lemma}

\begin{proof}
    The case of $i\neq n$ can be proved exactly the same as in \cite[Example, p.280]{V98}. We prove it for $i=n$. It suffices to show the case for $n=1$. Pick a non-negative integer $a$, such that $a+1\leq d$. Put 
    \[A=\bigg(\begin{matrix}
        d-a & a\\
        a & d-a
    \end{matrix} \bigg),\quad 
    B=\bigg(\begin{matrix}
        d-1 & 1\\ 
        1 & d-1
    \end{matrix} \bigg).\]
    By definition, 
    \[\calO_A=\{(V_1,V_1')\mid V_1=V_1^\perp, V'_1=(V'_1)^\perp, \dim (V_1\cap V_1')=d-a\}.\]
    Thus, 
    \begin{align*}
        p_{12}^{-1}(\calO_A)\cap p_{23}^{-1}(\calO_B)&= \{(V_1,V_1',V_1'')\mid V_1=V_1^\perp, V'_1=(V'_1)^\perp, V''_1=(V''_1)^\perp,\\
        & \dim (V_1\cap V_1')=d-a, \dim (V'_1\cap V''_1)=d-1\}.
    \end{align*}
    
    Let us first show that for any $(V_1,V_1',V_1'')\in p_{12}^{-1}(\calO_A)\cap p_{23}^{-1}(\calO_B)$, $\dim(V_1\cap V_1'')\geq d-a-1$. If $V_1\cap V_1'\subset V_1'\cap V_1''$, then $V_1\cap V_1'\cap V_1''=V_1\cap V_1'$, hence $\dim (V_1\cap V_1'\cap V_1'')=d-a$. Now let us assume $V_1\cap V_1'\not\subset V_1'\cap V_1''$. Let $w_1,\ldots, w_{d-1}$ be a basis for the $d-1$ dimensional vector space $V_1'\cap V_1''$. Then there exists a nonzero vector $u\in V_1\cap V_1'$, $u\notin V_1'\cap V_1''$. Therefore, $V_1'=\Span\{w_1,\ldots,w_{d-1},u\}=(V_1'\cap V_1'')\oplus \mathbb{C}u$. Intersecting with $V_1$ and since $u\in V_1$, we get $V_1\cap V_1'=(V_1\cap V_1'\cap V_1'')\oplus \mathbb{C}u$, which implies that $\dim (V_1\cap V_1'\cap V_1'')=d-a-1$. Hence, in both cases, $\dim(V_1\cap V_1'')\geq \dim  (V_1\cap V_1'\cap V_1'')\geq d-a-1$, and it is easy to construct an example where both equalities hold.  
    
    Hence, by the definition of $A\circ B$ in Proposition \ref{prop:orbit} and the Bruhat order in Equation \eqref{equ:order}, 
    \[A\circ B=\bigg(\begin{matrix}
        d-a-1 & a+1\\
        a+1 & d-a-1
    \end{matrix} \bigg).\]
    The lemma follows from the following claim:
    \begin{claim}\label{claim:AB}
        The composition
        \[\star:\underline{\bfR}^{[A]^\fc}\otimes \underline{\bfR}^{[B]^\fc}\stackrel{\star}{\longrightarrow}\underline{K}^{G\times\bbC^*}(\calZ_{\preceq A\circ B})\twoheadrightarrow  \underline{\bfR}^{[A\circ B]^\fc}\]
        is surjective, where the first map is convolution and the second map is the quotient by $\underline{K}^{G\times\bbC^*}(\calZ_{\prec A\circ B})$.
    \end{claim}
    Let $\calO_T:=p_{12}^{-1}(\calO_A)\cap p_{23}^{-1}(\calO_B)\cap p_{13}^{-1}(\calO_{A\circ B})$, and let $p_{13,T}$ be the restriction of $p_{13}$ to $\calO_T$. Let us analyze the fiber of $p_{13,T}$. Fix $(V_1,V_1'')\in \calO_{A\circ B}$, and choose any $(V_1,V_1',V_1'')\in p_{13,T}^{-1}(V_1,V_1'')$. Let us first show that $V_1'$ is determined by $V_1'\cap V_1''$. By the argument before Claim \ref{claim:AB}, $V_1\cap V_1'\cap V_1''=V_1\cap V_1''$ has dimension $d-a-1$, say, with a basis $\{w_1,\ldots, w_{d-a-1}\}$. Extend this to a basis $\{w_1,\ldots,w_{d-1}\}$ for $V_1'\cap V_1''$. We show that these will determine $V_1'$ uniquely. Extend the basis of  $V_1'\cap V_1''$ to a basis $\{w_1,\ldots, w_d\}$ for $V_1''$, and assume $(V_1'\cap V_1'')^\perp=V_1'+V_1''=\Span\{w_1,\ldots, w_{d-1},w_d,w_d'\}$ for some nonzero vector $w_d'$. Assume such $V_1'$ is not unique, and let $V_{1,a}'$ and $V_{1,b}'$ be two such $V_1'$'s. Then $V_{1,a}'=\Span\{w_1,\ldots, w_{d-1},u\}$ and $V_{1,b}'=\Span\{w_1,\ldots, w_{d-1},v\}$ for some nonzero $u\in V_1\cap V_{1,a}'$ and $v\in V_1\cap V_{1,b}'$. Moreover, $u,v\notin V_1'\cap V_1''$. Since $u,v\in (V_1'\cap V_1'')^\perp = V_1'+V_1''=\Span\{w_1,\ldots, w_{d-1},w_d,w_d'\}$, and $u,v\notin V_1'\cap V_1''$, the coefficients of $w_d'$ in both $u$ and $v$ are nonzero. Up to rescaling we can assume $u=\sum_{i=1}^d x_iw_i+w_d'$ and $v=\sum_{i=1}^d y_i w_i+w_d'$. But since $V_{1,a}'=\Span\{w_1,\ldots, w_{d-1},u\}\neq V_{1,b}'=\Span\{w_1,\ldots, w_{d-1},v\}$, we get $x_d\neq y_d$. Since $u,v\in V_1$, we have $u-v\in V_1$. Also, $u-v=\sum_{i=1}^d(x_i-y_i)w_i\in V_1''$. Hence, $u-v\in V_1\cap V_1''=  \Span\{w_1,\ldots, w_{d-a-1}\}$, contradicting with the fact that the coefficient of $w_d$ in $u-v$ is $x_d-y_d\neq 0$. Hence, we conclude that $V_1'$ is uniquely determined by $V_1'\cap V_1''$. Note that $V_1\cap V_1''=V_1\cap V_1'\cap V_1''\subset V_1'\cap V_1''$. Thus, the fiber $p_{13,T}^{-1}(V_1,V_1'')$ is isomorphic to $\Gr(a, \frac{V_1''}{V_1\cap V_1''})\simeq \Gr(a,a+1)$.  
    
    To compute the composition in Claim \ref{claim:AB}, we need to compute the pushforward map $p_{13,T*}$. We do this by the localization theorem, and we will use the notation from the proof of Proposition \ref{prop:pushforward}. Let $\pt_T:=(V_1,V_1',V_1'')$ be the base point in $\calO_T$, where $V_1=\Span(\epsilon_1,\dots,\epsilon_{d-a-1},\epsilon_{2d-a},\dots,\epsilon_{2d})$, $V_1'=\Span(\epsilon_1,\dots,\epsilon_{d-1},\epsilon_{2d})$, and $V_1''=\Span(\epsilon_1,\dots,\epsilon_{d})$. Then the point $(V_1',V_1'')$ is the base point $\pt_B$ in the orbit $\calO_B$. However, we have
    \[
    (V_1,V_1')=(\iota_d)s_{d-a,d}(\pt_A),
    \]
    where $\iota_d$ is the Weyl group element defined in \S\ref{subsec:push} and $s_{d-a,d}\in S_d\subset W_\fc$ is the transposition between $d-a$ and $d$. The torus fixed points in the fiber $p_{13,T}^{-1}(V_1,V_1'')$ are $(V_1,V_{1,j}',V_1'')$, where $V_{1,j}'$ is spanned by $\{\epsilon_1,\dots,\epsilon_{d},e_{d+j}\}\setminus\{e_j\}$, where $d-a\leq j\leq d$. Thus, $(V_{1,j}',V_1'')=s_{j,d}(\pt_B)$, and $(V_1,V_{1,j}')=s_{j,d}(V_1,V_1')$. Moreover, the restriction of the relative tangent bundle of $p_{13,T}$ to the base point $\pt_T$ has weights $x_d/x_j$, for $d-a\leq j\leq d-1$. Therefore, the composition in Claim \ref{claim:AB} is given by (see \cite[Corollary 3]{V98})
    \begin{align*}
        f\star g 
        &=\sum_{d-a\leq k\leq d}s_{k,d}\bigg(g\cdot \iota_d s_{d-a,d}(f)\cdot \prod_{d-a\leq j\leq d-1}\frac{1-q^2\frac{x_d}{x_j}}{1-\frac{x_j}{x_d}}\bigg)\\
        &= (-1)^a\sum_{d-a\leq k\leq d}s_{k,d}\bigg(gx_d^a\cdot \iota_ds_{d-a,d}(f\prod_{d-a+1\leq j\leq d}x_j^{-1})\cdot \prod_{d-a\leq j\leq d-1}\frac{1-q^2\frac{x_d}{x_j}}{1-\frac{x_d}{x_j}}\bigg).
    \end{align*}
    Note that
    \begin{align*}
    \bfR^{[A]^\fc}(q) &\simeq\bbC(q)[x_1^{\pm 1},\ldots, x_{d-a}^{\pm 1}]^{S_{d-a}}\otimes \bbC(q)[x_{d-a+1}^{\pm 1},\ldots, x_{d}^{\pm 1}]^{S_a},
    \\
    \bfR^{[B]^\fc}(q) &\simeq \bbC(q)[x_1^{\pm 1},\ldots, x_{d-1}^{\pm 1}]^{S_{d-1}}\otimes \bbC(q)[x_{d}^{\pm 1}],
    \end{align*}
    and 
    \[
    \bfR^{[A\circ B]^\fc}(q)\simeq\bbC(q)[x_1^{\pm 1},\ldots, x_{d-a-1}^{\pm 1}]^{S_{d-a-1}}\otimes \bbC(q)[x_{d-a}^{\pm 1},\ldots, x_{d}^{\pm 1}]^{S_{a+1}}.
    \]
    For any $h\in \bfR^{[A\circ B]^\fc}(q)$, we have
    \[
    \sum_{d-a\leq k\leq d}s_{k,d}\bigg(h\cdot \prod_{d-a\leq j\leq d-1}\frac{1-q^2 x_d/x_j}{1- x_d /x_j}\bigg) =h(1+q^2+\ldots +q^{2a})=hq^a\frac{q^{a+1}-q^{-a-1}}{q-q^{-1}}.
    \]
    Therefore, to prove Claim \ref{claim:AB}, it is enough to observe that the following map 
    \[
    \bfR^{[A]^\fc}(q)\otimes \bfR^{[B]^\fc}(q)\longrightarrow\bfR^{[A\circ B]^\fc}(q)
    \]
    \[f\otimes g\mapsto g\cdot\iota_ds_{d-a,d}(f), 
    \]
    is surjective. This holds since for any $m\in \bbZ$,
    \[
    \sum_{1\leq i\leq d-a-1}x_i^m=\iota_d s_{d-a,d}\bigg(\sum_{1\leq i\leq d-a}x_i^m\bigg)\cdot 1-1\cdot x_d^{-m},
    \]
    and 
    \[\sum_{d-a\leq i\leq d}x_i^m=\iota_d s_{d-a,d}\bigg(\sum_{d-a+1\leq i\leq d}x_i^m\bigg)\cdot 1+1\cdot x_d^m.\] 
    This finishes the proof of the lemma.
\end{proof}

\subsection{Convolution with generating elements}

Let us compute the convolution products with some distinguished elements. 

\begin{prop}\label{prop:diagconv}
	Let $A=\diag(\bfv)$, for $\bfv\in \Lambda_\fc$. Then $A\circ B=B$ for any $B\in \Xi_d(\bfv,\bfw)$. Moreover, we have $f\star g=fg$ for any $f\in \underline{\bfR}^{[A]^{\fc}}$ and $g\in \underline{\bfR}^{[B]^{\fc}}$, where we regard $f \in \underline{\bfR}^{[B]^{\fc}}$ by the natural inclusion $\underline{\bfR}^{[A]^{\fc}} \subset \underline{\bfR}^{[B]^{\fc}}$.
\end{prop}
\begin{proof}
	Since $A$ is a diagonal matrix, the orbit $\calO_A$ is the diagonal copy of $\calF_{\ro(A)}$. Thus, the restriction of $p_{13}$ to $p_{12}^{-1}(\calO_A)\cap p_{23}^{-1}(\calO_{B})$ is an isomorphism. The proposition follows from this fact and  \cite[Corollary 3]{V98}.
\end{proof}

Let $C\in \Xi_d(\bfv,\mathbf{w}) $. Put \[(h,l):=\min\{(i,j)\mid 1\leq j< i\leq N, c_{ij}\neq 0\},\]
with respect to the right-lexicographic order \eqref{equ:rightorder}. Let 
\[B=C-c_{hl}(E_{h,l}^\theta-E_{h-1,l}^\theta),\quad \text{and}\quad A=E^\theta_{h,h-1}(\bfv-c_{hl}\bfe_h-c_{hl}\bfe_{N+1-h},c_{hl}).\]
Then $\co(A)=\ro(B)$ since for any $j$,
\begin{align*}
	\sum_{i}a_{ij} &= a_{jj}+\delta_{j,h-1}c_{h,l}+\delta_{j,N+2-h}c_{hl}\\   
	&= v_j-\delta_{j,h}c_{hl}-\delta_{j,N+1-h}c_{hl}+\delta_{j,h-1}c_{h,l}+\delta_{j,N+2-h}c_{hl}\\
	&= \sum_i c_{ji}-\delta_{j,h}c_{hl}-\delta_{j,N+1-h}c_{hl}+\delta_{j,h-1}c_{h,l}+\delta_{j,N+2-h}c_{hl}
	=\sum_i b_{ji}.
\end{align*}
Let $\mathcal{O}:=p_{12}^{-1}(\mathcal{O}_A)\cap p_{23}^{-1}(\mathcal{O}_B)\cap p_{13}^{-1}(\mathcal{O}_C)$. Then
we have the following proposition.
\begin{prop}\label{prop:AB}
	Let $A,B,C$ be as above. Then $A\circ B=C$, and $p_{13}|_{\mathcal{O}}:\calO\rightarrow \calO_{C}$ is an isomorphism.
\end{prop}
\begin{proof}
    The case of $h\neq n+1$ can be proved as in \cite[Proposition 8]{V98}. Let us show the case $h=n+1$. To simplify notations, we let $a$ denote $c_{n+1,l}$. Pick $(F,F',F'')\in p_{12}^{-1}(\mathcal{O}_A)\cap p_{23}^{-1}(\mathcal{O}_B)$, and assume $(F,F'')\in \calO_{C'}$ for some matrix $C'$. Recall 
	\begin{align*} 
		\mathcal{O}_{A}=& \bigg\{(F,F')\mid 
		F=(V_k)_{0\le k \le N}\in \mcal{F}_{\bfv}, F'=(V'_k)_{0\le k \le N}\in \mcal{F}_{\bfv},\\
		&
		V_n\cap V'_n \stackrel{a}{\subset}V_n, V_k=V'_k \textit{ if } k\neq n\bigg\}.
	\end{align*}
	Hence, $c'_{ij}=b_{ij}$ if $i\neq n,n+1$. Let us assume that, for $1\leq j\leq n$,
	\[\dim (V_n'\cap V_j'')- \dim (V_n\cap V_j'')=s_1+\ldots+s_j\]
	for some integers $s_1,\ldots,s_n$. By Equation \eqref{equ:a}, for $1\leq j\leq n$,
	\begin{align*}
		c'_{n+1,j}-b_{n+1,j}=&\dim \frac{V_{n+1}\cap V_j''}{V_n\cap V_j''+V_{n+1}\cap V_{j-1}''}-\dim \frac{V_{n+1}\cap V_j''}{V'_n\cap V_j''+V_{n+1}\cap V_{j-1}''}\\
		=&\dim (V_n'\cap V_j'')- \dim (V_n\cap V_j'')-\dim (V_n'\cap V_{j-1}'')+ \dim (V_n\cap V_{j-1}'')\\
		=&s_j,
	\end{align*}
	and 
	\begin{align*}
		c'_{n,j}-b_{n,j}=&\dim \frac{V_{n}\cap V_j''}{V_{n-1}\cap V_j''+V_{n}\cap V_{j-1}''}-\dim \frac{V'_{n}\cap V_j''}{V_{n-1}\cap V_j''+V'_{n}\cap V_{j-1}''}
		=-s_j.
	\end{align*}
	Hence, 
	\[C'=B+\sum_{j=1}^n s_j(E_{n+1,j}^\theta-E_{n,j}^\theta).\]
	Since the entries of $C'$ must be nonnegative, and by the definition of $(n+1,l)$, we see that $s_1=\ldots=s_{l-1}=0$. By Lemma \ref{lem:Lagint} below,  $s_1+\ldots+s_j\leq a$ for any $1\leq j\leq n$. Hence, by the partial order defined in \eqref{equ:order}, 
	\[
    C'\preceq B+a(E_{n+1,l}^\theta-E_{n,l}^\theta)=C,
    \]
    while $C$ corresponds to the case $s_l=a$ and $s_j=0$ for $j\neq l$. Hence, $A\circ B=C$.
	
	Now let us prove the second statement. Let $(F,F',F'')\in \mathcal{O}$, then it is enough to show that $F'$ is uniquely determined by $(F,F'')$. Since $(F,F')\in \calO_A$, we only need to show that $V_n'$ is uniquely determined by $(F,F'')$. Recall the matrix $C$ corresponds to the choice $s_l=a$ and $s_j=0$ for $j\neq l$, we get 
    \[\dim (V_n'\cap V_l'')- \dim (V_n\cap V_l'')=a.\]
    Since $ \dim (V_n\cap V_n')=d-a$, Lemma \ref{lem:Lagint} below implies that $V_n'\subset V_n\cap V_n'+V_l''$. Intersecting with $V_n'$, we get $V_n'=V_n\cap V_n+V_n'\cap V_l''$.
    Since $b_{n+1,j}=0$ for $1\leq j\leq l$, we get by Equation \eqref{equ:a} that $\dim V_n'\cap V_l''=\dim V_{n+1}'\cap V_l''$. Hence, $V_n'\cap V_l''= V_{n+1}'\cap V_l''=V_{n+1}\cap V_l''$. Therefore,
    \begin{equation}\label{equ:Vn'}
        V_n'=V_n\cap V_n'+V_{n+1}\cap V_l''.
    \end{equation}
    Since $c_{ij}=c_{N+1-i,N+1-j}$ and by the definition of $(h,l)$, we get $c_{ij}=0$ if $1\leq i\leq n$ and $j>N+1-l$. Hence, $\dim V_n\cap V_{N+1-l}''=\dim V_n$, which implies $V_n=V_n\cap V_{N+1-l}''$. On the other hand, $b_{n+1,l}=0$ implies that $b_{n,N+1-l}=0$. Together with Equation \eqref{equ:a}, we get
    \[V_n'\cap V_{N+1-l}''=V_{n-1}\cap  V_{N+1-l}''+V_n'\cap V_{N-l}''.\]
    Intersecting with $V_n$, we get
    \begin{equation}\label{equ:VnVn'}
        V_n\cap V_n'=V_n\cap V_n'\cap V_{N+1-l}''=V_{n-1}\cap  V_{N+1-l}''+V_n\cap V_n'\cap V_{N-l}'',
    \end{equation}
    where the first equaltiy follows from $V_n=V_n\cap V_{N+1-l}''$. Finally, by Equation \eqref{equ:a} again,
    \[\dim V_n'\cap V_{N-l}''-\dim V_n\cap V_{N-l}''=a,\]
    which implies $V_n\cap V_n'\cap V_{N-l}''=V_n\cap V_{N-l}''$ by Lemma \ref{lem:Lagint}. Plugging this into Equations \eqref{equ:Vn'} and \eqref{equ:VnVn'}, we get
    \[V_n'=V_{n-1}\cap  V_{N+1-l}''+V_n\cap V_{N-l}''+V_{n+1}\cap V_l''.\]
    Thus, $V_n'$ is uniquely determined by $(F,F'')$. This finishes the proof.
\end{proof}    

The following lemma is used in the proof of Proposition \ref{prop:AB} above.

\begin{lemma}\label{lem:Lagint}
	Let $U_1,U_2$ be two Lagrangian subspaces of $V=\mathbb{C}^{2d}$ such that $\dim (U_1\cap U_2)=d-a$. Then for any subspace $U_3\subset V$, we have $\dim (U_1\cap U_3)-\dim (U_2\cap U_3)\leq a$, and the equality holds if and only if $U_1\subset U_1\cap U_2+U_3$ and $U_1\cap U_2\cap U_3=U_2\cap U_3$.
\end{lemma}

\begin{proof}
    First of all, $\dim (U_2\cap U_3)\geq \dim (U_1\cap U_2\cap U_3)$, and the equality holds if and only if $U_1\cap U_2\cap U_3=U_2\cap U_3$. On the other hand,
    \[\dim(U_1+U_3)\geq \dim(U_1\cap U_2+U_3)=d-a+\dim U_3-\dim (U_1\cap U_2\cap U_3),\]
    and the equality holds if and only if $U_1\subset U_1\cap U_2+U_3$. Hence, $\dim(U_1\cap U_3)=d+\dim U_3-\dim(U_1+U_3)\leq a+\dim (U_1\cap U_2\cap U_3)$. Therefore, $\dim (U_1\cap U_3)-\dim (U_2\cap U_3)\leq a$, and the equality holds if and only if $U_1\subset U_1\cap U_2+U_3$ and $U_1\cap U_2\cap U_3=U_2\cap U_3$.
\end{proof}

\subsection{A generating set}

We now describe a generating set for the convolution algebra $\underline{K}^{G\times\bbC^*}(\calZ)$. This is a generalization of \cite[Proposition 10]{V98}.

\begin{theorem}\label{thm:generators}
    The convolution algebra $\underline{K}^{G\times\bbC^*}(\calZ)$ is generated by $\underline{K}^{G\times\bbC^*}(\calZ_{\diag(\bfv)})$ for $\bfv\in \Lambda_{\fc,d}$, and $\underline{K}^{G\times\bbC^*}(\calZ_{E^\theta_{i,i+1}(\bfv',1)})$ for $\bfv'\in \Lambda_{\fc,d-1}$ and $1\leq i\leq N-1$.
\end{theorem}
\begin{proof}
    It suffices to prove the statement for the associated graded of $\underline{K}^{G\times\bbC^*}(\calZ)$. For a matrix $C\in \Xi_d(\bfv,\mathbf{w})$, let
    \[\ell(C):=\sum_{i>j}\begin{pmatrix}
        i-j+1\\
        2
    \end{pmatrix}c_{ij}.\]
    We will prove by induction on $\ell(C)$ that, modulo the lower graded piece $\underline{K}^{G\times\bbC^*}(\calZ_{\prec C})$, $\underline{K}^{G\times\bbC^*}(\calZ_{\preceq C})$ can be obtained via convolution from classes in $\underline{K}^{G\times\bbC^*}(\calZ_A)$ such that $A\in \Xi_d$ is a diagonal matrix or a matrix of type $E^\theta_{i,i+1}(\bfv',1)$ for $1\leq i\leq N-1$. 

    If $\ell(C)=0$, then $C$ is a diagonal matrix. If $\ell(C)=1$, then $C=E^\theta_{i,i+1}(\bfv',1)$ for some $i$. In both cases, there is nothing to show. Suppose $\ell(C)>1$. Put 
    \[(h,l):=\min\{(i,j)\mid 1\leq j< i\leq N, c_{ij}\neq 0\},\]
    with respect to the right-lexicographic order \eqref{equ:rightorder}, and let 
    \[
    B=C-c_{hl}(E_{h,l}^\theta-E_{h-1,l}^\theta),\quad \text{and}\quad A=E^\theta_{h,h-1}(\bfv-c_{hl}\bfe_h-c_{hl}\bfe_{N+1-h},c_{hl}),
    \]
    be as in the setting of Proposition \ref{prop:AB} above. Hence, $A\circ B=C$. Combining with \cite[Corollary 3]{V98}, we get that for any $f\in \underline{\bfR}^{[A]^\fc}$ and $g\in \underline{\bfR}^{[B]^\fc}$ with lifts $\tilde{f}\in \underline{K}^{G\times\bbC^*}(\calZ_{\preceq A})$ and $\tilde{g}\in \underline{K}^{G\times\bbC^*}(\calZ_{\preceq B})$, $\tilde{f}\star \tilde{g}= fg\in \underline{\bfR}^{[C]^\fc}$ modulo $\underline{K}^{G\times\bbC^*}(\calZ_{\prec C})$. 

    We show that the map $\star:\underline{K}^{G\times\bbC^*}(\calZ_{\preceq A})\otimes \underline{K}^{G\times\bbC^*}(\calZ_{\preceq B})\rightarrow \underline{K}^{G\times\bbC^*}(\calZ_{\preceq C})$ is surjective. The case $h\neq n+1$ can be shown exactly the same as in \cite[Proposition 10]{V98}. For the case of $h=n+1$, the corresponding surjectivity follows from the surjectivity of the map
    \[\underline{\bfR}^{S_{[A]_{n+1,n}}}\otimes \underline{\bfR}^{S_{[B]_{n,l}}}\longrightarrow \underline{\bfR}^{S_{[C]_{n+1,l}}}\otimes \underline{\bfR}^{S_{[C]_{n,l}}}, 
    \quad
    f\otimes g\mapsto i(f)j(g),\]
    where $i$ is the isomorphism $\underline{\bfR}^{S_{[A]_{n+1,n}}}\simeq \underline{\bfR}^{S_{[C]_{n+1,l}}}$ and $j: \underline{\bfR}^{S_{[B]_{n,l}}}\hookrightarrow \underline{\bfR}^{S_{[C]_{n+1,l}}}\otimes \underline{\bfR}^{S_{[C]_{n,l}}}$ is the inclusion. (Note that $a_{n+1,n}=c_{n+1,l}$ and $b_{n,l}=c_{n+1,l}+c_{n,l}$.)
    Notice that $\ell(B)<\ell(C)$. Finally, we use Lemma \ref{lem:Egenerators} to conclude the proof.
\end{proof}

\section{A polynomial representation of the affine iquantum group $\tUi$}
\label{sec:iquantum}	

In this section, we formulate a representation of the iquantum group associated with the quasi-split Satake diagram of affine type $\AIII_{2n-1}^{(\tau)}$; the proofs will be completed in Section~\ref{sec:proof} and Appendix~\ref{App:A}. To that end, it is essential for us to use a variant of the Drinfeld presentation of this iquantum group obtained in \cite{LWZ24}.

\subsection{The Drinfeld presentation}
	
 Let $(c_{ij})_{i,j\in \I}$ be the Cartan matrix of affine type $A_{2n-1}$, where 
 \[
 \I = \I_0 \cup \{0\}, \qquad
  \I_0 =\{1, 2,\ldots, 2n-1\},
 \]
 with the affine node $0$. Let $\tau$ be the diagram involution such that $\tau (0)=0,\tau(i)=2n-i$, for $i\in \I_0$. Then $(\I,\tau)$ represents a quasi-split Satake diagram below:
	\begin{center}
		\begin{tikzpicture}[baseline=0, scale=5]
  \put(-200,-1.4){${\rm AIII}_{1}^{(\tau)}$}
			\node at (-0.5,-0.015) {$\circ$};
			\node at (-0.14,-0.015) {$\circ$};

			\draw[-](-0.48, 0) to (-0.16, 0);
			\draw[-](-0.48, -0.025) to (-0.16, -0.025);
			\node at (-0.14, -.06) { \tiny $1$};
			\node at (-0.5,-.06) {\tiny $0$};
		\end{tikzpicture}
	\end{center}

\begin{center}
\centering
\begin{picture}(70,35)(0,5)
                   \put(-150,20){${\rm AIII}_{2n-1}^{(\tau)}\,(n\ge 2)$}
   
				\put(0,9){$\circ$}
				\put(0,29){$\circ$}
				\put(47,9){$\circ$}
				\put(47,29){$\circ$}
				\put(72,9){$\circ$}
				\put(72,29){$\circ$}
				\put(94,19){$\circ$}
				\put(-11,2.8){\tiny ${2n-1}$}
				\put(-2,36){\tiny ${1}$}
				\put(64,2.8){\tiny $n+1$}
				\put(69,36){\tiny $n-1$}
				\put(94,14){\tiny $n$}
				
				\put(5,11.5){\line(1,0){13}}
				\put(5,31.5){\line(1,0){13}}
				\put(19,9){$\cdots$}
				\put(19,29){$\cdots$}
				\put(33.5,11.5){\line(1,0){13}}
				\put(33.5,31.5){\line(1,0){13}}
				\put(53,11.5){\line(1,0){18.5}}
				\put(53,31.5){\line(1,0){18.5}}
				\put(-17.5,22.5){\line(2,1){17}}
				\put(-17.5,20.5){\line(2,-1){17}}
				\put(78,12){\line(2,1){17}}
				\put(78,31){\line(2,-1){17}}
				\put(-23,19){$\circ$}
				\put(-23,11){\tiny $0$}
				
				\color{red}
				\qbezier(0,13.5)(-4,21.5)(0,29.5)
				\put(-0.25,14){\vector(1,-2){0.5}}
				\put(-0.25,29){\vector(1,2){0.5}}
				
				\qbezier(47,13.5)(43,21.5)(47,29.5)
				\put(46.75,14){\vector(1,-2){0.5}}
				\put(46.75,29){\vector(1,2){0.5}}

				\qbezier(72,13.5)(68,21.5)(72,29.5)
				\put(71.75,14){\vector(1,-2){0.5}}
				\put(71.75,29){\vector(1,2){0.5}}
			\end{picture}
\end{center}

Denote by $[k]$ and $\qbinom{k}{r}$ the $q$-integers and $q$-binomials, for $k,r\in \mbb N$. 

The (universal) quasi-split affine iquantum group of type AIII$_{2n-1}^{(\tau)}$ is the $\C(q)$-algebra $\tUi =\tUi(\widehat{\fg})$ generated by $B_i$, $\K_i^{\pm 1}$ $(i\in \I)$, subject to the following relations \eqref{eq:SK}--\eqref{eq:S3}, for $i, j\in \I$:
	\begin{align}
		\K_i\K_i^{-1} =\K_i^{-1}\K_i=1,  \quad \K_i\K_\ell &=\K_\ell \K_i, \quad
		\K_\ell B_i=q^{c_{\tau \ell,i} -c_{\ell i}} B_i \K_\ell,
		\label{eq:SK} \\
		B_iB_j -B_j B_i&=0, \qquad\qquad\qquad \text{ if } c_{ij}=0 \text{ and }\tau i\neq j,
		\label{eq:S1} \\
		\sum_{s=0}^{1-c_{ij}} (-1)^s \qbinom{1-c_{ij}}{s} & B_i^{s}B_jB_i^{1-c_{ij}-s} =0, \quad \text{ if } \tau i\neq i \text{ and } j \not\in \{i,\tau i\},
		\label{eq:S6}
		\\
		B_{\tau i}B_i -B_i B_{\tau i}& =   \frac{\K_i -\K_{\tau i}}{q-q^{-1}},
		\quad \text{ if }  c_{i,\tau i}=0,
		\label{relation5}	\\
		B_i^2 B_j -[2] B_i B_j B_i +B_j B_i^2 &= - q^{-1}  B_j \K_i,  \quad \text{if }c_{ij}=-1 \text{ and }c_{i,\tau i}=2,
		\label{eq:S2}
  \\
		\sum_{r=0}^3 (-1)^r \qbinom{3}{r} B_i^{3-r} B_j B_i^{r} &= -q^{-1} [2]^2 (B_iB_j-B_jB_i) \K_i, 
		\quad \text{if } -c_{ij}= c_{i,\tau i}=2.  \label{eq:S3}
	\end{align}
The above formulation of universal iquantum groups was due to Ming Lu and the second author and originated from iHall algebras, where parameters in Letzter-Kolb's iquantum groups are replaced by additional Cartan generators $\K_i$'s (see \cite{LW21, LWZ24} for references therein).

The Serre relation \eqref{eq:S3} arises only in the iquantum group of type AIII$_1^{(\tau)}$, which is also known as the $q$-Onsager algebra.

We introduce the following generating functions for an indeterminate $z$: 
\begin{align}
\label{GF}
	\begin{split}
		\bB_{i}(z) & =\sum_{k\in\bbZ} B_{i,k}z^{k}, 
		\\
		\bTh_{i}(z) &=1+ \sum_{k\ge 1} (q-q^{-1})\Theta_{i,k}z^{k} =\sum_{k\in\bbZ}(q-q^{-1})\Theta_{i,k}z^{k}=\exp\big((q-q^{-1})\bH_i(z) \big),
		\\
		\bH_i(z) &=\sum_{k\ge 1} H_{i,k} z^{k},\\
		\bDel(z) & =\sum_{k\in\bbZ} C^k z^k,
		\quad
		\bDel^+(z)=\sum_{k\geq0} C^k z^k.
	\end{split}
\end{align}  
Here $\Theta_{i,0}=\frac{1}{q-q^{-1}}$, $\Theta_{i,k}=0$ for $k<0$.
Define 
\begin{align*}
	&\mathbb{S}_{i,j}(w_1,w_2| z)\\
	&:= \Sym_{w_1,w_2}\big(\bB_i(w_1)\bB_i(w_2)\bB_j(z)-[2]\bB_i(w_1)\bB_j(z)\bB_i(w_2)+\bB_j(z)\bB_i(w_1)\bB_i(w_2)\big),
\end{align*}
where it is understood that $\Sym_{w_1,w_2} f(w_1,w_2) := f(w_1,w_2) +f(w_2,w_1).$ We shall denote by $[x,y]_v=xy-vyx$ the $v$-commutator below. 

The Drinfeld type current presentations for $\tUi$ have been obtained in \cite{LW21} for $n=1$ and in \cite{LWZ24} for general $n\ge 2$. 

\begin{theorem} \cite[Theorem 4.6]{LWZ24}  \label{thm:drinfeld}
	The quasi-split affine iquantum group $\tUi$ is isomorphic to the
	$\C(q)$-algebra generated by the elements $B_{i,\ell}$, $\Theta_{i,m}$, $\K_i^{\pm 1}$, and $C^{\pm 1}$, where $i\in\I_0$, $\ell\in\Z$ and $m\ge 1$, subject
	to the following relations \eqref{R-comm}--\eqref{serre}, for $i,j\in \I_0$:
	\begin{align}
		C \textit{ is central}, & \qquad [\K_i,\K_j]=[\K_i, \bTh_j(w)] = 0,\label{R-comm}
		\\
		\bTh_i(z)\bTh_j(w)&= \bTh_j(w)\bTh_i(z), 
		\\
		\K_i\bB_{j}(w) &= q^{c_{\tau i,j}-c_{ij}} \bB_{j}(w) \K_i, 
		\\
		\bB_j(w)  \bTh_i(z)
		&= \frac{(1 -q^{c_{ij}}zw^{-1}) (1 -q^{-c_{\tau i,j}}zw C)}{(1 -q^{-c_{ij}}zw^{-1})(1 -q^{c_{\tau i,j}} zw C)}
		\bTh_i(z) \bB_j(w), \label{R-BTh}
		\\
		[\bB_i(z), \bB_{\tau i}(w)] 
		&= \frac{\bDel (zw)}{q-q^{-1}} (\K_{\tau i} \bTh_i(z) -\K_{i} \bTh_{\tau i}(w)),  \qquad \text{ if } c_{i,\tau i}=0, 
		\\
		(q^{c_{ij}}z -w) \bB_i(z) & \bB_j(w) +(q^{c_{ij}}w-z) \bB_j(w) \bB_i(z)=0, \quad \text{ if } j \neq  \tau i,  
		\\
		(q^2z-w) \bB_i(z) & \bB_i(w) +(q^{2}w-z) \bB_i(w) \bB_i(z)\label{R-BiBi}
		\\ 
		= \frac{\bDel(zw)\K_{i}}{q-q^{-1}}  & \big( (z -q^{-2}w)\bTh_i(w) +(w -q^{-2}z)\bTh_i(z) \big), \quad \text{ if } i =\tau i, \notag
	\end{align}
	and the Serre relations:
 \begin{align}
		&\mathbb{S}_{i,j}(w_1,w_2| z)=0, \textit{\quad if } c_{ij}=-1, j\neq \tau i\neq i, \\
		& \label{serre}
  \mathbb{S}_{i,j}(w_1,w_2| z) = -\K_i\frac{\bDel(w_1w_2)}{q-q^{-1}}\Sym_{w_1,w_2}\bigg( \frac{[2] z w_1^{-1} }{1-q^{2}w_2w_1^{-1}}  [\bTh_i(w_2),\bB_{j}(z)]_{q^{-2}} \\
		&	+\frac{1 +w_2w_1^{-1}}{1 -q^{2}w_2w_1^{-1}}[\bB_{j}(z),\bTh_i(w_2)]_{q^{-2}}  \bigg),\textit{ if } c_{ij}=-1,i=\tau i.\notag
	\end{align}
\end{theorem}
If one formally sets $C=0$ and $\Delta =0$ in the relations above, one sees the Drinfeld presentation of half a quantum loop algebra \cite{Dr87, Beck, Da12}. 

\subsection{Some equivalent relations}

For our geometric application, it turns out some variant of the above Drinfeld presentation of $\tUi$ will be more appropriate. 

\subsubsection{Serre relations}

First of all, the Serre relation \eqref{serre} can be simplified as follows. Recall from \cite{Z22} that the Serre relation \eqref{serre} can be derived from other relations \eqref{R-comm}--\eqref{R-BiBi} together with finite type Serre relations. In the process the Serre relation \eqref{serre} is derived from the following two relations (see \cite[(5.1)--(5.2)]{Z22}):
\begin{align}\label{serre2}
	&(w_1^{-1}+w_2^{-1}-[2]z^{-1})\bbS_{i,j}(w_1,w_2|z)\\
	&=\Sym_{w_1,w_2}\frac{\bDel(w_1w_2)}{q-q^{-1}}(q^{-2}w_1^{-1}-w_2^{-1}) [\bTh_i(w_2),\bB_{j}(z)]_{q^{-2}}\K_i,\notag
\end{align}
and 
\begin{align}\label{serre3}
	&(w_1+w_2-[2]z)\bbS_{i,j}(w_1,w_2|z)\\
	&=\Sym_{w_1,w_2}\frac{\bDel(w_1w_2)}{q-q^{-1}}(q^{-2}w_2-w_1) [\bB_{j}(z),\bTh_i(w_2)]_{q^{-2}}\K_i.\notag
\end{align}
Calculating $\eqref{serre2}\times [2]z+\eqref{serre3}\times (w_1^{-1}+w_2^{-1})$, we obtain \eqref{serre}. On the other hand, we can also obtain a Serre relation in a different form.

\begin{lemma}\label{lem:serre2}
	For any $i,j$ such that	$c_{ij}=-1,i=\tau i$, we have
	\begin{align}  \label{R-Serre2}
		(z-qw_1)(z-qw_2)\mathbb{S}_{i,j}(w_1,w_2| z)=\Sym_{w_1,w_2} \K_i\bDel(w_1w_2)(zw_1-q^{-2}zw_2)\bTh_i(w_2)\bB_{j}(z).
	\end{align}
\end{lemma}

\begin{proof}
	Calculating $\eqref{serre2}\times q^2zw_1w_2+\eqref{serre3}\times z$ gives us the identity
	\begin{align}
     \label{eq:newcomb}
		- [2](z-qw_1)(z-qw_2)\mathbb{S}_{i,j}(w_1,w_2|z) &= \K_i\frac{\bDel(w_1w_2)}{q-q^{-1}} \times 
  \\
		\Sym_{w_1,w_2}\Big((zw_2-q^2zw_1) [\bTh_i(w_2),\bB_{j}(z)]_{q^{-2}} &	+(q^{-2}zw_2-zw_1)[\bB_{j}(z),\bTh_i(w_2)]_{q^{-2}}  \Big). \notag
	\end{align}
	Using the relation \eqref{R-BTh} with $z, w$ switched,
	\[ \bB_j(z)  \bTh_i(w)
	=  \frac{(1 -q^{-1}z^{-1}w) (1 -qzw C)}{(1 -qz^{-1}w)(1 -q^{-1} zw C)}
	\bTh_i(w) \bB_j(z),\]
	we compute the main component of the RHS of \eqref{eq:newcomb} as
	\begin{align*}
		&(zw_2-q^2zw_1) [\bTh_i(w_2),\bB_{j}(z)]_{q^{-2}} +(q^{-2}zw_2-zw_1)[\bB_{j}(z),\bTh_i(w_2)]_{q^{-2}}\\
		&=\bigg((zw_2-q^2zw_1)\bigg(1-q^{-2} \frac{(1 -q^{-1}z^{-1}w_2) (1 -qzw_2 C)}{(1 -qz^{-1}w_2)(1 -q^{-1} zw_2 C)}\bigg)\\
		&+(q^{-2}zw_2-zw_1)\bigg(\frac{(1 -q^{-1}z^{-1}w_2) (1 -qzw_2 C)}{(1 -qz^{-1}w_2)(1 -q^{-1} zw_2 C)}-q^{-2}\bigg)\bigg)\bTh_i(w_2)\bB_{j}(z)\\
		&=(q^{-2}zw_2-zw_1)(q^2-q^{-2})\bTh_i(w_2)\bB_{j}(z)\\
		&=[2](q^{-2}zw_2-zw_1)(q-q^{-1})\bTh_i(w_2)\bB_{j}(z).
	\end{align*}
	Plugging this back into \eqref{eq:newcomb} finishes the proof.
\end{proof}

\subsubsection{A new variant of $\bTh_n(z)$}
We define $\check{\Theta}_{n,m}$, for $m\geq 1$, by the recursive formula
\[
\Theta_{n,m}=\check{\Theta}_{n,m}-\sum_{a=1}^{\lfloor \frac{m-1}{2}\rfloor}(q^{2}-1)\check{\Theta}_{n,m-2a}C^a-\delta_{m,ev}qC^{\frac{m}{2}},
\]
where $\delta_{m,ev}=\begin{cases}
		1, &\textit{ if m is even},\\
		0, &\textit{otherwise}
	\end{cases}$.
Set $\check{\Theta}_{n,m}=0$ for $m<0$, and $\check{\Theta}_{n,0}=\frac{1}{q-q^{-1}}$. We further form the generating function
$\check\bTh_{n}(z) =1+ \sum_{k\ge 1} (q-q^{-1}) \check\Theta_{n,k}z^{k}.$

\begin{lemma}
We have
\begin{align}
    \label{eq:Th2}
     \bTh_n(z)=\frac{1-q^2Cz^2}{1-Cz^2} \check{\bTh}_n(z).
\end{align}
\end{lemma}

\begin{proof}
It follows by definitions that 
\begin{align*}
	\bTh_n(z)&= 1+\sum_{m\geq 1}(q-q^{-1})\Theta_{n,m}z^m\\
			&= \check{\bTh}_n(z)-\sum_{m\geq 1}(q-q^{-1})\sum_{a=1}^{\lfloor \frac{m-1}{2}\rfloor}(q^{2}-1)\check{\Theta}_{n,m-2a}C^az^m-(q-q^{-1})\frac{qCz^2}{1-Cz^2}\\
			&= \check{\bTh}_n(z)-\sum_{m\geq 1}(q-q^{-1})\check{\Theta}_{n,m}z^m\sum_{a\geq 1}(q^{2}-1)C^az^{2a}-(q-q^{-1})\frac{qCz^2}{1-Cz^2}\\
			&= \check{\bTh}_n(z)-(\check{\bTh}_n(z)-1)\frac{(q^{2}-1)Cz^2}{1-Cz^2}-(q^2-1)\frac{Cz^2}{1-Cz^2}\\
			&= \frac{1-q^2Cz^2}{1-Cz^2}\check{\bTh}_n(z).
		\end{align*}
  The lemma is proved. 
\end{proof}

\subsubsection{A variant of \eqref{R-BiBi} and \eqref{R-Serre2}}

\begin{lemma}
The relation \eqref{R-BiBi} is equivalent to  the relation 
\begin{align}  \label{R-BnBn}
&(q^2z-w) \bB_n(z) \bB_n(w) +(q^{2}w-z) \bB_n(w) \bB_n(z)\\
&= \frac{\bDel(zw)  \K_{i}q^{-2} }{q-q^{-1}} \frac{(z-q^2w)(w-q^2z)}{z-w} ( \check{\bTh}_n(z) -\check{\bTh}_n(w)).\notag
\end{align}
\end{lemma}
	
\begin{proof}
 By definition \eqref{GF} of $\Delta$, we have $\Delta(zw) Cw = \Delta(zw) z^{-1}$. Using this identity in the equation~ (*) below, we compute that
	\begin{align*}
		&(q^2z-w) \bB_n(z) \bB_n(w) +(q^{2}w-z) \bB_n(w) \bB_n(z)
			\\
		&\overset{\eqref{R-BiBi}}{=} \frac{\bDel(zw)  \K_{i} }{q-q^{-1}}  \bigg( (z -q^{-2}w)\frac{1-q^2Cw^2}{1-Cw^2}\check{\bTh}_n(w) +(w -q^{-2}z)\frac{1-q^2Cz^2}{1-Cz^2}\check{\bTh}_n(z) \bigg)
           \\
		& \overset{(*)}{= }\frac{\bDel(zw)  \K_{i} }{q-q^{-1}}  \bigg( (z -q^{-2}w)\frac{z-q^2w}{z-w}\check{\bTh}_n(w) +(w -q^{-2}z)\frac{w-q^2z}{w-z}\check{\bTh}_n(z) \bigg)
           \\
		&= \frac{\bDel(zw)  \K_{i}q^{-2} }{q-q^{-1}} \frac{(z-q^2w)(w-q^2z)}{z-w} ( \check{\bTh}_n(z) -\check{\bTh}_n(w)).
		\end{align*}
The lemma is proved. 
	\end{proof}

\begin{lemma}
    The relation \eqref{R-Serre2} is equivalent to the relation: for any $i,j$ such that	$c_{ij}=-1,i=\tau i$, we have
	\begin{align}  \label{R-Serre2var}
		&(z-qw_1)(z-qw_2)\mathbb{S}_{i,j}(w_1,w_2| z)\\
        &=\K_i\bDel(w_1w_2)\frac{z(qw_1-q^{-1}w_2)(q^{-1}w_1-qw_2)}{w_1-w_2}(\check{\bTh}_i(w_2)\bB_{j}(z)-\check{\bTh}_i(w_1)\bB_{j}(z)).\notag
	\end{align}
\end{lemma}
\begin{proof}
    By the assumption, $i=n$. Hence, 
    \begin{align*}  
		&(z-qw_1)(z-qw_2)\mathbb{S}_{i,j}(w_1,w_2| z)\\
        &\overset{\eqref{R-Serre2}}{=}\Sym_{w_1,w_2} \K_i\bDel(w_1w_2)(zw_1-q^{-2}zw_2)\bTh_i(w_2)\bB_{j}(z)\\
        &=\K_i\bDel(w_1w_2)(zw_1-q^{-2}zw_2)\frac{1-q^2Cw_2^2}{1-Cw_2^2} \check{\bTh}_i(w_2)\bB_{j}(z)\\
        &+\K_i\bDel(w_1w_2)(zw_2-q^{-2}zw_1)\frac{1-q^2Cw_1^2}{1-Cw_1^2} \check{\bTh}_i(w_1)\bB_{j}(z)\\
        &=\K_i\bDel(w_1w_2)(zw_1-q^{-2}zw_2)\frac{w_1-q^2w_2}{w_1-w_2} \check{\bTh}_i(w_2)\bB_{j}(z)\\
        &+\K_i\bDel(w_1w_2)(zw_2-q^{-2}zw_1)\frac{w_2-q^2w_1}{w_2-w_1} \check{\bTh}_i(w_1)\bB_{j}(z)\\
        &=\K_i\bDel(w_1w_2)\frac{z(qw_1-q^{-1}w_2)(q^{-1}w_1-qw_2)}{w_1-w_2}(\check{\bTh}_i(w_2)\bB_{j}(z)-\check{\bTh}_i(w_1)\bB_{j}(z)).
	\end{align*}
\end{proof}

\subsection{A variant of Drinfeld presentation} 

Combining all the new variants of relations \eqref{eq:Th2}, \eqref{R-BnBn}, and \eqref{R-Serre2var}, we obtain a new Drinfeld presentation of $\tUi$ below from the original one in Theorem~\ref{thm:drinfeld}. For the sake of simplicity of notations, we shall use $\bTh_n(z)$ to denote the $\check{\bTh}_n(z)$ above.

\begin{theorem} \label{thm:drinfeld'} 
The affine iquantum group $\tUi$ is isomorphic to the
	$\C(q)$-algebra generated by the elements $B_{i,l}$, $\Theta_{i,m}$, $\K_i^{\pm 1}$, and $C^{\pm 1}$, where $i\in\I_0$, $l\in\Z$ and $m\ge 1$, subject
	to the following relations \eqref{iDR KK}--\eqref{iDR Serre1}, for $i,j\in \I_0$:
	\begin{align}
		C \textit{ is central}, & \qquad  [\K_i,\K_j]=[\K_i, \bTh_j(w)] = 0,
		\label{iDR KK} \\
		\bTh_i(z)\bTh_j(w) &= \bTh_j(w)\bTh_i(z),
		\label{iDR ThTh} \\
		\K_i\bB_{j}(w) &= q^{c_{\tau i,j}-c_{ij}} \bB_{j}(w) \K_i,
		\label{iDR KB} \\
	  \bB_j(w)  \bTh_i(z)
		&=  \frac{(1 -q^{c_{ij}}zw^{-1}) (1 -q^{-c_{\tau i,j}}zw C)}{(1 -q^{-c_{ij}}zw^{-1})(1 -q^{c_{\tau i,j}} zw C)}\bTh_i(z) \bB_j(w), \label{iDR BTh}
		\\
		[\bB_i(z), \bB_{\tau i}(w)] 
		&= \frac{\bDel (zw)}{q-q^{-1}} (\K_{\tau i} \bTh_i(z) -\K_{i} \bTh_{\tau i}(w)),  \qquad \text{ if }c_{i,\tau i}=0\label{iDR Btaui},
		\\
		(q^{c_{ij}}z -w) \bB_i(z) & \bB_j(w) +(q^{c_{ij}}w-z) \bB_j(w) \bB_i(z)=0, \quad \text{ if } j \neq  \tau i, \label{iDR BB0}
		\\
		(q^2z-w) \bB_i(z) & \bB_i(w) +(q^{2}w-z) \bB_i(w) \bB_i(z)
  		\label{iDR BBii} \\
		= \frac{\bDel(zw)\K_{i}q^{-2}}{q-q^{-1}} & \frac{(z-q^2w)(w-q^2z)}{z-w} ( \bTh_i(z) -\bTh_i(w))\notag, \quad \text{ if } i =\tau i,
		\end{align}
		and the Serre relations:
  \begin{align}
		&\mathbb{S}_{i,j}(w_1,w_2| z)=0, \textit{\quad if } c_{ij}=-1, j\neq \tau i\neq i,
  		\label{iDR Serre0} \\
		& (z-qw_1)(z-qw_2)\mathbb{S}_{i,j}(w_1,w_2| z)=\K_i\bDel(w_1w_2)\frac{z(qw_1-q^{-1}w_2)(q^{-1}w_1-qw_2)}{w_1-w_2}
  		\label{iDR Serre1} \\
        &(\bTh_i(w_2)\bB_{j}(z)-\bTh_i(w_1)\bB_{j}(z)),\textit{ if } c_{ij}=-1,i=\tau i.\notag
	\end{align}
\end{theorem}	
Note that, in our setting, $c_{i,\tau i}=0 \Leftrightarrow i\neq n$ and $i=\tau i \Leftrightarrow i=n$. 

\subsection{A polynomial representation}
 \label{subsec:poly}

For $m\in \Z$, we denote
\begin{align}  \label{eq:th}
\theta_m(z) =\frac{q^mz-1}{z-q^m}. 
\end{align}
Then $\theta_m(z)^{-1}=\theta_m(z^{-1})$. For any subset $I\subset [d]$, we denote 
\begin{align}  \label{PhiI}
\Phi_I(z):=\prod_{t\in I}\theta_1(z/x_t). 
\end{align}

For any set partition $I=(I_1,I_2,\ldots, I_l)$ of $\{1,2,\dots,d\}$, let $(x_I)$ denote the corresponding ordered set of variables 
\[
(x_{i_{1,1}},\dots, x_{i_{1,j_1}},x_{i_{2,1}},\dots, x_{i_{2,j_2}},\dots, x_{i_{l,j_l}},x^{-1}_{i_{1,1}},\dots, x^{- 1}_{i_{1,j_1}},x^{- 1}_{i_{2,1}},\dots, x^{- 1}_{i_{2,j_2}},\dots, x^{- 1}_{i_{l,j_l}}), 
\]
where $I_k=\{i_{k,1}<i_{k,2}<\dots<i_{k,j_k}\}$. For any $r\in I_s$, let $\tau_r^+I$ (respectively, $\tau_r^-I$) be the partition of $\{1,2,\dots,d\}$ with $r$ shifted from $I_s$ to $I_{s+1}$ (respectively, $I_{s-1}$). Let $(x_{\iota_jI})$ (resp. $(x_{\iota_j\iota_kI})$) denote the ordered set of variables $(x_I)$ with $x_j^{\pm 1}$ (resp. $x_j^{\pm 1}$ and $x_k^{\pm 1}$) switched to $x_j^{\mp 1}$ (resp. $x_j^{\mp 1}$ and $x_k^{\mp 1}$). For example, let $I=(\{1,2\},\{3,4\})$ be a set partition of $\{1,2,3,4\}$. Then \[f(x_{\tau_3^-\tau_1^+I})=f(x_2,x_3,x_1,x_4,x_2^{-1},x_3^{-1},x_1^{-1},x_4^{-1})=f(x_{\tau_1^+\tau_3^-I}),\] and 
\[
f(x_{\iota_1\tau_1^+I})=f(x_2,x_1^{-1},x_3,x_4,x_2^{-1},x_1,x_3^{-1},x_4^{-1}).
\]
Moreover, we define $f(x_{\tau_r^+\iota_sI}):=f(x_{\iota_s\tau_r^+I})$, $f(x_{\tau_r^+\iota_s\iota_tI}):=f(x_{\iota_s\iota_t\tau_r^+I})$, and $f(x_{\iota_s\tau_r^+\iota_tI}):=f(x_{\iota_s\iota_t\tau_r^+I})$.

Denote
\begin{align}  \label{eq:P}
    \bfP:= \bigoplus_{\mbf v\in \Lambda_{\fc,d}} \bfR^{[\bfv]^\fc} [q,q^{-1}]. 
    \end{align}
Define operators $\hat{B}_{n,k}$ and $\hat{E}_{i,k}$,$\hat{F}_{i,k}$ (for $i\in \I_0 \setminus \{n\}$, i.e., $1\leq i\leq n-1, k\in \bbZ$)  on 
\[
\underline{\bfP} := \bigoplus_{\mbf v\in \Lambda_{\fc,d}} \underline{\bfR}^{[\bfv]^\fc},
\]
where we recall that 
\[\underline{\bfR}^{[\bfv]^\fc}:=\bfR^{[\bfv]^\fc}(q)\otimes_{\bbC}\bigg(\bbC\Big[\frac{1}{1-e^\alpha},\frac{1}{1-q^2e^\alpha}, \forall \alpha \in R \Big]\bigg)^W.\]
Recall also $\underline{R}(G)(q)=R(G)(q)\otimes_{\bbC}\Big(\bbC\big[\frac{1}{1-e^\alpha},\frac{1}{1-q^2e^\alpha}, \forall \alpha \in R \big]\Big)^W$.

Let $1\leq i\leq n-1$. For $\mbf v\in \Lambda_{\fc,d}$  \eqref{eq:Lamdacd}, denote
\begin{align*}
\mathbf{v'} &=\bfv-\bfe_i+\bfe_{i+1}+\bfe_{2n-i}-\bfe_{2n+1-i},
\\
\mathbf{v''} &=\bfv+\bfe_i-\bfe_{i+1}-\bfe_{2n-i}+\bfe_{2n+1-i}.
\end{align*}
We define 
\[
\hat{E}_{i,k}\in \bigoplus_{\mbf v\in \Lambda_{\fc,d}} \Hom_{\underline{R}(G)(q)}(\underline{\bfR}^{[\mbf v']^{\fc}},\underline{\bfR}^{[\bfv]^{\fc}})
\]
by letting
\begin{equation}\label{equ:Ei}
    (\hat{E}_{i,k}f)(x_{[\bfv]^\fc}):=\sum_{j\in [\bfv]_{i}}x_j^k\Phi_{[\bfv]_i\setminus \{j\}}(qx_j)f(x_{\tau^+_j[\bfv]^\fc})\in \underline{\bfR}^{[\bfv]^\fc},
    \quad \text{for } f\in \underline{\bfR}^{[\bfv']^\fc}.
\end{equation}
We further define 
\[\hat{F}_{i,k}\in \bigoplus_{\mbf v\in \Lambda_{\fc,d}} \Hom_{\underline{R}(G)(q)}(\underline{\bfR}^{[\mbf v'']^{\fc}},\underline{\bfR}^{[\bfv]^{\fc}})\] 
by letting
\begin{equation}\label{equ:F}
    (\hat{F}_{i,k}f)(x_{[\bfv]^\fc}) :=\sum_{j\in [\bfv]_{i+1}}x_j^k\Phi_{[\bfv]_{i+1}\setminus \{j\}}(q^{-1}x_j)^{-1}f(x_{\tau^-_j[\bfv]^\fc})\in \underline{\bfR}^{[\bfv]^{\fc}},
    \quad \text{for }f\in \underline{\bfR}^{[\bfv'']^\fc}.
\end{equation}
Finally, we define 
\[\hat{B}_{n,k}\in \bigoplus_{\mbf v\in \Lambda_{\fc,d}} \Hom_{\underline{R}(G)(q)}(\underline{\bfR}^{[\mbf v]^{\fc}},\underline{\bfR}^{[\bfv]^{\fc}})\]
by letting
\begin{equation}\label{equ:En}
    (\hat{B}_{n,k}f)(x_{[\bfv]^\fc}):=\sum_{j\in [\bfv]_n}x_{j}^k\cdot\theta_1(qx_j^2)\cdot\Phi_{[\bfv]_n\setminus\{j\}}(qx_j) \cdot f(x_{\iota_j[\bfv]^\fc})\in \underline{\bfR}^{[\bfv]^\fc}.
\end{equation}

\begin{remark}
    By Propositions \ref{prop:actions1} and \ref{prop:actions2} below, the operators $\hat{E}_{i,k}$ and $\hat{F}_{i,k}$ defined in \eqref{equ:Ei}--\eqref{equ:F} preserve the (non-localized) vector space $\bfP$ while Proposition~ \ref{prop:actions3} shows that $\hat{B}_{n,k}$ acts on the localized vector space $\underline{\bfP}$.
\end{remark}

For $1\leq i\leq n-1$, recalling $N=2n$, we denote  
\begin{align*}
	\hat{E}_i(z) &=\sum_{k\in\bbZ}\hat{E}_{i,k}q^{ki}z^{k},
 \qquad \hat{F}_{i}(z)=\sum_{k\in\bbZ}\hat{F}_{i,k}q^{-k(N-i)}z^{-k},
 \\
 \hat{B}_n(z) &=\sum_{k\in\bbZ}\hat{B}_{n,k}q^{kn}z^{k}.
\end{align*}

We define linear operators on $\underline{\bfR}^{[\bfv]^\fc}$, for each $\mbf v\in \Lambda_{\fc,d}$ and $1\leq i\leq n-1$:
\begin{align}  \label{eq:3K}
\begin{cases}
    &\hat{\K}_i = \text{ the scalar multiplication by }q^{-v_i+v_{i+1}},\\
    &\hat{\K}_{\tau i} = \hat{\K}_i^{-1}, \\
    &\hat{\K}_n = \text{ the scalar multiplication by } -q.
\end{cases}
\end{align}
This gives rise to linear maps
\[
\hat{\K}_i  \in \bigoplus_{\mbf v\in \Lambda_{\fc,d}} \Hom_{\underline{R}(G)(q)}(\underline{\bfR}^{[\mbf v]^{\fc}},\underline{\bfR}^{[\bfv]^{\fc}}),
\qquad \text{ for } i\in \I_0 =\{1, \ldots, 2n-1\}.
\]

For $1\leq i\leq n$, recalling $[\bfv]_i$ from \eqref{eq:vi} and $\Phi_I(z)$ from \eqref{PhiI}, we define the functions
\begin{equation}\label{equ:Phi}
	\Phi_{i,\bfv}(z):=\begin{dcases}
 \Phi_{[\bfv]_i}(q^{1-i}z)\Phi_{[\bfv]_{i+1}}(q^{-1-i}z)^{-1},\quad &\textit{ if } 1\leq i\leq n-1;\\
		\Phi_{[\bfv]_n}(q^{1-n}z)\Phi_{[\bfv]_n}(q^{1+n}z^{-1}), \quad &\textit{ if } i=n. 
  \end{dcases}
\end{equation}
Note that $\Phi_{n,\bfv}(z^{-1})=\Phi_{n,\bfv}(zq^N)$. For $1\leq i\leq n$, we define the operator $\hat{\bTh}_i(z)$ (respectively, $\hat{\bTh}_{\tau i}(z)$), which acts on $\underline{\bfR}^{[\bfv]^\fc}$ by multiplication with the rational function $q^{v_{i+1}-v_i}\Phi_{i,\bfv}(z^{-1})$ (respectively, $q^{v_i-v_{i+1}}\Phi_{i,\bfv}(zq^N)$). All these functions expand as a formal power series in $z$ with leading term 1. 

\begin{theorem}\label{thm:polyrepeven}
	The assignment 
	\[
 \Psi: C\mapsto q^N, \quad\K_i \mapsto \hat{\K}_i,\quad \bTh_i(z)\mapsto \hat{\bTh}_i(z),\]
	\[
 \bB_{i}(z)\mapsto \begin{dcases}\hat{E}_{i}(z) &\textit{ if } 1\leq i\leq n-1,\\
		\hat{B}_n(z) &\textit{ if } i=n,\\
		\hat{F}_{\tau(i)}(z) &\textit{ if } 1+n\leq i\leq 2n-1,
	\end{dcases}
 \]
	defines a representation of $\tUi$ on $\underline{\bfP}$; here the presentation of $\tUi$ in Theorem \ref{thm:drinfeld'} is used.
\end{theorem}

The proof of the theorem will occupy Section~\ref{sec:proof} and Appendix~\ref{App:A}.

\section{Proof of Theorem \ref{thm:polyrepeven}}
\label{sec:proof}

We shall prove Theorem \ref{thm:polyrepeven} by checking that the corresponding operators under $\Psi$ therein satisfy all the relations \eqref{iDR KK}--\eqref{iDR Serre1} for $\tUi$ in Theorem \ref{thm:drinfeld'}. The relations \eqref{iDR KK} and \eqref{iDR ThTh} are clear. The verification of the Serre relations \eqref{iDR Serre0}--\eqref{iDR Serre1} is long and will be deferred to  Appendix~\ref{App:A}. In this section we shall verify the remaining relations \eqref{iDR KB} and \eqref{iDR BTh}--\eqref{iDR BBii} one-by-one. 

\subsection{Relation \eqref{iDR KB}}  

Recall Relation \eqref{iDR KB} states that $\K_i \bB_j(w) =q^{c_{\tau i,j}-c_{ij}} \bB_{j}(w) \K_i$.

This relation clearly holds for $i=n$ since $\hat{\K}_n=-q \text{Id}$. 

Let us assume $i\neq n$.
Since $\hat{\K}_{\tau i}=\hat{\K}^{-1}_i$, it suffices to verify the relation for $1\leq i\leq n-1$. If $1\leq j\leq n-1$, for any $f\in\underline{\bfR}^{[\bfv]^\fc}$,
\begin{align*}
	\hat{\K}_i\hat{E}_{j}(w)\hat{\K}_i^{-1}(f)&=q^{-(\bfe_j-e_{j+1}-\bfe_{2n+1-j}+\bfe_{2n+2-j})_i+(\bfe_j-\bfe_{j+1}-\bfe_{2n+1-j}+\bfe_{2n+2-j})_{i+1}}\hat{E}_{j}(z)(f)\\	&=q^{-2\delta_{i,j}+\delta_{i,j+1}+\delta_{i,j-1}+2\delta_{i,\tau j}-\delta_{i,\tau j+1}-\delta_{i,\tau j-1}}\hat{E}_{j}(w)(f)\\
	&=q^{c_{\tau i,j}-c_{i,j}}\hat{E}_{j}(w)(f),
\end{align*}
and
\begin{align*}
	\hat{\K}_i\hat{F}_{j}(w)\hat{\K}_i^{-1}(f)&=q^{(\bfe_j-\bfe_{j+1}-\bfe_{2n+1-j}+\bfe_{2n+2-j})_i-(\bfe_j-\bfe_{j+1}-\bfe_{2n+1-j}+\bfe_{2n+2-j})_{i+1}}\hat{F}_{j}(w)(f)\\
	&=q^{c_{i,j}-c_{\tau i,j}}\hat{F}_{j}(w)(f)\\
	&=q^{c_{\tau i,\tau j}-c_{i,\tau j}}\hat{F}_{j}(w)(f).
\end{align*}
If $j=n$, it follows by \eqref{eq:3K} that  $\hat{\K}_i\hat{B}_{n}(w) = \hat{B}_{n}(w)\hat{\K}_i = q^{c_{\tau i,n} -c_{i,n}}\hat{B}_{n}(w)\hat{\K}_i$, since $\hat{B}_n(w)$ maps each $\underline{\bfR}^{[\bfv]^\fc}$ to itself and $c_{\tau i,n}=c_{i,n}$.

\subsection{Relation \eqref{iDR BTh}} 

Introduce the delta function 
\[
\delta(z)=\sum_{k\in\Z} z^k.
\]
Recall Relation \eqref{iDR BTh} states that 
\[\bB_j(w)\bTh_i(z)= \frac{(1 -q^{c_{ij}}zw^{-1}) (1 -q^{-c_{\tau i,j}}zw C)}{(1 -q^{-c_{ij}}zw^{-1})(1 -q^{c_{\tau i,j}} zw C)}\bTh_i(z)\bB_j(w),
\]
which is equivalent to 
\begin{align} \label{eq:BKTh}
\bB_j(w)  \K_{\tau i}\bTh_i(z)= q^{c_{\tau i,j}-c_{ij}}\frac{(1 -q^{c_{ij}}zw^{-1}) (1 -q^{-c_{\tau i,j}}zw C)}{(1 -q^{-c_{ij}}zw^{-1})(1 -q^{c_{\tau i,j}} zw C)}\K_{\tau i}\bTh_i(z) \bB_j(w).
\end{align}
Since $C\mapsto q^{2n}$, $\hat{\K}_{ i}\hat{\bTh}_{\tau i}(z)= \big(\hat{\K}_{\tau i}\hat{\bTh}_i(z) \big)|_{z\mapsto z^{-1}q^{-2n}}$, and $q^{c_{\tau i,j}-c_{ij}}\frac{(1 -q^{c_{ij}}zw^{-1}) (1 -q^{-c_{\tau i,j}}zw C)}{(1 -q^{-c_{ij}}zw^{-1})(1 -q^{c_{\tau i,j}} zw C)}$ is invariant under the change $i\leftrightarrow \tau i$, $z\leftrightarrow z^{-1}q^{-2n}$, we only need to check the relation \eqref{eq:BKTh}, for $1\leq i\leq n$. We separate these into 2 cases.

\underline{Case (1): $1\leq i\leq n-1$}. 

We prove it when $1\leq j\leq n$, as the case of $n+1\leq j\leq 2n-1$ follows in the same way. The case for $|j-i|\geq 2$ is trivial. So we only need check the cases for $j\in\{i-1,i,i+1\}$. If $j=i-1$, then we have 
\begin{align*}
	&(\hat{E}_{i-1}(w)\hat{\bTh}_{i}(z)f)(x_{[\bfv]^\fc})\\
	&= \sum_{r\in [\bfv]_{i-1}}\delta(q^{i-1}x_rw)\Phi_{[\bfv]_{i-1}\setminus\{r\}}(qx_r)\cdot (\hat{\bTh}_{i}(z)f)(x_{\tau_r^+[\bfv]^\fc})\\
	&= q^{v_{i+1}-v_i-1}\Phi_{i,\bfv}(z^{-1})\sum_{r\in [\bfv]_{i-1}}\delta(q^{i-1}x_rw)\theta_1(q^{1-i}z^{-1}x_r^{-1})\Phi_{[\bfv]_{i-1}\setminus\{r\}}(qx_r)f(x_{\tau_r^+[\bfv]^\fc})\\
	&= q^{-1}\theta_1(wz^{-1})(\hat{\bTh}_{i}(z)\hat{E}_{i-1}(w)f)(x_{[\bfv]^\fc})\\
	&= \frac{1-q^{-1}zw^{-1}}{1-qzw^{-1}}(\hat{\bTh}_{i}(z)\hat{E}_{i-1}(w)f)(x_{[\bfv]^\fc}),
\end{align*}
where the third equality follows from the fact that 
\[
\delta(q^{i-1}x_rw)\theta_1(q^{1-i}z^{-1}x_r^{-1})=\delta(q^{i-1}x_rw)\theta_1(wz^{-1}).
\]
If $j=i$, then we have 
\begin{align*}
	&(\hat{E}_{i}(w)\hat{\bTh}_{i}(z)f)(x_{[\bfv]^\fc})\\
	&= \sum_{r\in [\bfv]_{i}}\delta(q^{i}x_rw)\Phi_{[\bfv]_{i}\setminus\{r\}}(qx_r)\cdot (\hat{\bTh}_{i}(z)f)(x_{\tau_r^+[\bfv]^\fc})\\
	&= q^{v_{i+1}-v_i+2}\Phi_{i,\bfv}(z^{-1})\sum_{r\in [\bfv]_{i}}\delta(q^{i}x_rw)\theta_1(q^{-1-i}z^{-1}x_r^{-1})^{-1}\theta_1(q^{1-i}z^{-1}x_r^{-1})^{-1}\Phi_{[\bfv]_{i}\setminus\{r\}}(qx_r)f(x_{\tau_r^+[\bfv]^\fc})\\
	&= q^2\frac{\theta_1(qzw^{-1})}{\theta_1(qwz^{-1})}(\hat{\bTh}_{i}(z)\hat{E}_{i}(w)f)(x_{[\bfv]^\fc})\\
	&= \frac{1-q^{2}zw^{-1}}{1-q^{-2}zw^{-1}}(\hat{\bTh}_{i}(z)\hat{E}_{i}(w)f)(x_{[\bfv]^\fc}).
\end{align*}
If $j=i+1\leq n-1$, then we have 
\begin{align*}
	&(\hat{E}_{i+1}(w)\hat{\bTh}_{i}(z)f)(x_{[\bfv]^\fc})\\
	&= \sum_{r\in [\bfv]_{i+1}}\delta(q^{i+1}x_rw)\Phi_{[\bfv]_{i+1}\setminus\{r\}}(qx_r)\cdot (\hat{\bTh}_{i}(z)f)(x_{\tau_r^+[\bfv]^\fc})\\
	&= q^{v_{i+1}-v_i-1}\Phi_{i,\bfv}(z^{-1})\sum_{r\in [\bfv]_{i+1}}\delta(q^{i+1}x_rw)\theta_1(q^{-1-i}z^{-1}x_r^{-1})\Phi_{[\bfv]_{i+1}\setminus\{r\}}(qx_r)f(x_{\tau_r^+[\bfv]^\fc})\\
	&= q^{-1}\theta_1(wz^{-1})(\hat{\bTh}_{i}(z)\hat{E}_{i+1}(w)f)(x_{[\bfv]^\fc})\\
	&= \frac{1-q^{-1}zw^{-1}}{1-qzw^{-1}}(\hat{\bTh}_{i}(z)\hat{E}_{i+1}(w)f)(x_{[\bfv]^\fc}).
\end{align*}
If $j=i+1=n$, then we have 
\begin{align*}
	&(\hat{B}_n(w)\hat{\bTh}_{n-1}(z)f)(x_{[\bfv]^\fc})\\
	&= \sum_{r\in [\bfv]_n}\delta(q^nx_rw)\theta_1(qx_r^2)\Phi_{[\bfv]_n\setminus\{r\}}(qx_r)\cdot (\hat{\bTh}_{n-1}(z)f)(x_{\iota_r[\bfv]^\fc})\\
	&= q^{v_{n}-v_{n-1}}\Phi_{n-1,\bfv}(z^{-1})\sum_{r\in [\bfv]_n}\delta(q^nx_rw)\theta_1(q^{-n}z^{-1}x_r^{-1}) \times \\
 &\qquad\qquad\qquad\theta_1(q^{-n}z^{-1}x_r)^{-1}\theta_1(qx_r^2)\Phi_{[\bfv]_n\setminus\{r\}}(qx_r)f(x_{\iota_r[\bfv]^\fc})\\
	&= \theta_1(wz^{-1})\theta_1(q^{2n}zw)(\hat{\bTh}_{n-1}(z)\hat{B}_n(w)f)(x_{[\bfv]^\fc})\\
	&= \frac{(1 -q^{-1}zw^{-1}) (1 -qzw C)}{(1 - qzw^{-1})(1 -q^{-1}zw C)}(\hat{\bTh}_{n-1}(z)\hat{B}_n(w)f)(x_{[\bfv]^\fc}).
\end{align*}

\underline{Case (2): $i=n$}. 

For $j\notin \{n-1,n,n+1\}$, the relation \eqref{eq:BKTh} is clear. \\
If $j=n-1$, then we have 
\begin{align*}
	&(\hat{E}_{n-1}(w)\hat{\bTh}_{n}(z)f)(x_{[\bfv]^\fc})\\
	=&\sum_{r\in [\bfv]_{n-1}}\delta(q^{n-1}x_rw)\Phi_{[\bfv]_{n-1}\setminus\{r\}}(qx_r)\cdot (\hat{\bTh}_{n}(z)f)(x_{\tau_r^+[\bfv]^\fc})\\
	=&\Phi_{n,\bfv}(z^{-1})\sum_{r\in [\bfv]_{n-1}}\delta(q^{n-1}x_rw)\theta_1(q^{1-n}z^{-1}x_r^{-1})\theta_1(q^{1+n}zx_r^{-1})\Phi_{[\bfv]_{n-1}\setminus\{r\}}(qx_r)f(x_{\tau_r^+[\bfv]^\fc})\\
	=&\theta_1(wz^{-1})\theta_1(zwC)(\hat{\bTh}_n(z)\hat{E}_{n-1}(w)f)(x_{[\bfv]^\fc})\\
	=&\frac{1-q^{-1}zw^{-1}}{1-qzw^{-1}}\frac{1-qzwC}{1-q^{-1}zwC}(\hat{\bTh}_{i}(z)\hat{E}_{i-1}(w)f)(x_{[\bfv]^\fc}).
\end{align*}
If $j=n$, then we have 
\begin{align*}
	&(\hat{B}_n(w)\hat{\bTh}_{n}(z)f)(x_{[\bfv]^\fc})\\
	=&\sum_{r\in [\bfv]_n}\delta(q^nx_rw)\theta_1(qx_r^2)\Phi_{[\bfv]_n\setminus\{r\}}(qx_r)\cdot (\hat{\bTh}_{n}(z)f)(x_{\iota_r[\bfv]^\fc})\\
	=&\Phi_{n,\bfv}(z^{-1})\sum_{r\in [\bfv]_n}\delta(q^nx_rw)\theta_1(q^{1-n}z^{-1}x_r)\theta_1(q^{1+n}zx_r)\theta_1(q^{1-n}z^{-1}x_r^{-1})^{-1}\theta_1(q^{1+n}zx_r^{-1})^{-1}\\
	&\theta_1(qx_r^2)\Phi_{[\bfv]_n\setminus\{r\}}(qx_r)\cdot f(x_{\iota_r[\bfv]^\fc})\\	=&\theta_1(qC^{-1}w^{-1}z^{-1})\theta_1(qzw^{-1})\theta_1(qz^{-1}w)^{-1}\theta_1(qCzw)^{-1}(\hat{\bTh}_n(z)\hat{B}_n(w)f)(x_{[\bfv]^\fc})\\
	=&\frac{1-q^{2}zw^{-1}}{1-q^{-2}zw^{-1}}\frac{1-q^{-2}zwC}{1-q^2zwC}(\hat{\bTh}_{n}(z)\hat{B}_n(w)f)(x_{[\bfv]^\fc}).
\end{align*}
If $j=n+1$, then we have 
\begin{align*}
	&(\hat{F}_{n-1}(w)\hat{\bTh}_{n}(z)f)(x_{[\bfv]^\fc})\\
	=&\sum_{r\in [\bfv]_{n}}\delta(q^{n+1}x_r^{-1}w)\Phi_{[\bfv]_{n}\setminus\{r\}}(q^{-1}x_r)^{-1}\cdot (\hat{\bTh}_{n}(z)f)(x_{\tau_r^-[\bfv]^\fc})\\
	=&\Phi_{n,\bfv}(z^{-1})\sum_{r\in [\bfv]_{n}}\delta(q^{n+1}x_r^{-1}w)\theta_1(q^{1-n}z^{-1}x_r^{-1})^{-1}\theta_1(q^{1+n}zx_r^{-1})^{-1}\Phi_{[\bfv]_{n}\setminus\{r\}}(q^{-1}x_r)^{-1}f(x_{\tau_r^-[\bfv]^\fc})\\
	=&\theta_1(wz^{-1})\theta_1(zwC)(\hat{\bTh}_n(z)\hat{F}_{n-1}(w)f)(x_{[\bfv]^\fc})\\
	=&\frac{1-q^{-1}zw^{-1}}{1-qzw^{-1}}\frac{1-qzwC}{1-q^{-1}zwC}(\hat{\bTh}_{i}(z)\hat{E}_{i-1}(w)f)(x_{[\bfv]^\fc}).
\end{align*}
This completes the proof of Relation \eqref{eq:BKTh}, which is equivalent to \eqref{iDR BTh}, under $\Psi$.

\subsection{Relation \eqref{iDR Btaui} } 

Relation \eqref{iDR Btaui} states that $[\bB_i(z), \bB_{\tau i}(w)] = \frac{\bDel (zw)}{q-q^{-1}} (\K_{\tau i} \bTh_i(z) -\K_{i} \bTh_{\tau i}(w)), \text{ if }c_{i,\tau i}=0$, i.e., $i\neq n$. By symmetry, we can assume $1\leq i\leq n-1$. 

By definition, we have 
\begin{align*}
	&(\hat{E}_i(z)\hat{F}_i(w)f)(x_{[\bfv]^\fc})\\
	=&\sum_{r\in[\bfv]_i}\sum_{s\in [\bfv]_{i+1}\cup\{r\}}\delta(q^ix_rz)\delta(q^ix_sw^{-1}C^{-1})\Phi_{[\bfv]_i\setminus\{r\}}(qx_r)\Phi_{[\bfv]_{i+1}\cup\{r\}\setminus\{s\}}(q^{-1}x_s)^{-1}f(x_{\tau_s^-\tau_r^+[\bfv]^\fc}),
\end{align*}
and
\begin{align*}
	&(\hat{F}_i(w)\hat{E}_i(z)f)(x_{[\bfv]^\fc})\\
	=&\sum_{s\in [\bfv]_{i+1}}\sum_{r\in[\bfv]_i\cup\{s\}}\delta(q^ix_rz)\delta(q^{i}x_sw^{-1}C^{-1})\Phi_{[\bfv]_i\cup\{s\}\setminus\{r\}}(qx_r)\Phi_{[\bfv]_{i+1}\setminus\{s\}}(q^{-1}x_s)^{-1}f(x_{\tau_s^-\tau_r^+[\bfv]^\fc}).
\end{align*}
Thus, we have 
\begin{align*}
	&([\hat{E}_i(z),\hat{F}_i(w)]f)(x_{[\bfv]^\fc})\\
	=&f(x_{[\bfv]^\fc})\bigg(\sum_{r\in[\bfv]_i}\delta(q^ix_rz)\delta(q^{i}x_rw^{-1}C^{-1})\Phi_{[\bfv]_i\setminus\{r\}}(qx_r)\Phi_{[\bfv]_{i+1}}(q^{-1}x_r)^{-1}\\
	&-\sum_{r\in [\bfv]_{i+1}}\delta(q^ix_rz)\delta(q^{i}x_rw^{-1}C^{-1})\Phi_{[\bfv]_i}(qx_r)\Phi_{[\bfv]_{i+1}\setminus\{r\}}(q^{-1}x_r)^{-1}\bigg)\\
	=&\frac{f(x_{[\bfv]^\fc})}{q-q^{-1}}\sum_{r\in[\bfv]_i\cup [\bfv]_{i+1}}\delta(q^ix_rz)\delta(q^{i}x_rw^{-1}C^{-1})\frac{B(x_r)}{x_rA'(x_r)},
\end{align*}
where 
\begin{align*}
    A(x)&:= \prod_{r\in[\bfv]_i\cup [\bfv]_{i+1}}(x-x_r) 
    \\
    B(x)&:= \prod_{r\in[\bfv]_i}(qx-q^{-1}x_r)\prod_{r\in[\bfv]_{i+1}}(q^{-1}x-qx_r).
\end{align*}
Applying the residue theorem to $\delta(q^ixz)\delta(q^{i}xw^{-1}C^{-1})\frac{B(x)}{xA(x)}$, we obtain that 
\begin{align*}
	&([\hat{E}_i(z),\hat{F}_i(w)]f)(x_{[\bfv]^\fc})\nonumber\\
	&= \frac{f(x_{[\bfv]^\fc})\delta(zwC)}{q-q^{-1}}\left(\left(\frac{B(x)}{A(x)}\right)^+|_{x=q^{-i}z^{-1}}-\left(\frac{B(x)}{A(x)}\right)^-|_{x=q^{-i}wC}\right)\\
	&= \left(\frac{\bDel (zw)}{q-q^{-1}} (\hat{\K}_{\tau i} \hat{\bTh}_i(z) -\hat{\K}_{i} \hat{\bTh}_{\tau i}(w))f\right)(x_{[\bfv]^\fc}),
\end{align*}
where $\left(\frac{B(x)}{A(x)}\right)^\pm\in \bbC[x^{\mp 1}]$ denotes the Laurent expansion at $x=\infty$ and $x=0$, respectively, $|_{x=q^{-i}wC}$ denotes the substitution of $x$ by $q^{-i}wC$, and the last equality follows from $\frac{B(x)}{A(x)}=\Phi_{[\bfv]_i}(qx)\Phi_{[\bfv]_{i+1}}(q^{-1}x)^{-1}$.

\subsection{Relation \eqref{iDR BB0} }

Recall Relation \eqref{iDR BB0} states that $(q^{c_{ij}}z -w) \bB_i(z) \bB_j(w) +(q^{c_{ij}}w-z) \bB_j(w) \bB_i(z)=0,  \text{ if } j \neq \tau i$.

By the symmetry, we can assume $i\leq j$. There are three cases to consider depending on the values of $c_{ij}$.

\underline{Case (a): $c_{ij}=0$}. The only nontrivial case is that $j=\tau(i-1)$ and $i\leq n-1$. It follows by definition that 
	\begin{align*}
		&(\hat{E}_{i}(z)\hat{F}_{i-1}(w)f)(x_{[\bfv]^\fc})\\
		=&\sum_{r\in [\bfv]_i}\delta(q^ix_rz)\Phi_{[\bfv]_i\setminus \{r\}}(qx_r)\cdot (\hat{F}_{i-1}(w)f)(x_{\tau_r^+[\bfv]^\fc})\\
		=&\sum_{r,s\in [\bfv]_i,r\neq s}\delta(q^ix_rz)\delta(q^{i-1}x_sw^{-1}C^{-1})\Phi_{[\bfv]_i\setminus \{r\}}(qx_r)\Phi_{[\bfv]_i\setminus \{r,s\}}(q^{-1}x_s)^{-1}f(x_{\tau_s^-\tau_r^+[\bfv]^\fc}),
	\end{align*}
	and
	\begin{align*}
		&(\hat{F}_{i-1}(w)\hat{E}_{i}(z)f)(x_{[\bfv]^\fc})\\
		=&\sum_{r,s\in [\bfv]_i,r\neq s}\delta(q^ix_rz)\delta(q^{i-1}x_sw^{-1}C^{-1})\Phi_{[\bfv]_i\setminus \{s,r\}}(qx_r)\Phi_{[\bfv]_i\setminus \{s\}}(q^{-1}x_s)^{-1}f(x_{\tau_r^+\tau_s^-[\bfv]^\fc}).
	\end{align*}
Since $\theta_1(qx_r/x_s)=\theta_1(q^{-1}x_s/x_r)^{-1}$, we have $\hat{E}_{i}(z)\hat{F}_{i-1}(w)=\hat{F}_{i-1}(w)\hat{E}_{i}(z)$.
	
\underline{Case (b): $c_{ij}=-1$}. By the assumption $i\leq j$, we have $j=i+1$. Let us first consider the case $i<n-1$. 
	By definition, we have
	\begin{align*}
		&(\hat{E}_{i+1}(w)\hat{E}_i(z)f)(x_{[\bfv]^\fc})\\
		&= \sum_{s\in [\bfv]_{i+1}}\delta(q^{i+1}x_sw)\Phi_{[\bfv]_{i+1}\setminus\{s\}}(qx_s)\cdot (\hat{E}_{i}(z)f)(x_{\tau_s^+[\bfv]^\fc})\\
		&= \sum_{s\in [\bfv]_{i+1}}\sum_{r\in [\bfv]_i}\delta(q^ix_rz)\delta(q^{i+1}x_sw)\Phi_{[\bfv]_{i}\setminus\{r\}}(qx_r)\Phi_{[\bfv]_{i+1}\setminus\{s\}}(qx_s)\cdot f(x_{\tau_r^+\tau_s^+[\bfv]^\fc}).
	\end{align*}
	On the other hand, we have
	\begin{align*}
		&(\hat{E}_i(z)\hat{E}_{i+1}(w)f)(x_{[\bfv]^\fc})\\
		&= \sum_{r\in [\bfv]_i}\delta(q^ix_rz)\Phi_{[\bfv]_i\setminus\{r\}}(qx_r)\cdot (\hat{E}_{i+1}(w)f)(x_{\tau_r^+[\bfv]^\fc})\\
		&= \sum_{r\in [\bfv]_i}\sum_{s\in [\bfv]_{i+1}\cup\{r\}}\delta(q^ix_rz)\delta(q^{i+1}x_sw)\Phi_{[\bfv]_i\setminus\{r\}}(qx_r)\Phi_{[\bfv]_{i+1}\cup\{r\}\setminus\{s\}}(qx_s)\cdot f(x_{\tau_s^+\tau_r^+[\bfv]^\fc})\\
		&= \theta_1(z/w)(\hat{E}_{i+1}(w)\hat{E}_i(z)f)(x_{[\bfv]^\fc})\\
		&+\sum_{r\in [\bfv]_i}\delta(q^ix_rz)\delta(q^{i+1}x_rw)\Phi_{[\bfv]_i\setminus\{r\}}(qx_r)\Phi_{[\bfv]_{i+1}}(qx_r)\cdot f(x_{\tau_r^+\tau_r^+[\bfv]^\fc})
	\end{align*}
	Using $(q^{-1}z-w)\delta(q^ix_rz)\delta(q^{i+1}x_rw)=0$, we get 
	\[(q^{-1}z-w)\hat{E}_i(z)\hat{E}_{i+1}(w)=(z-q^{-1}w)\hat{E}_{i+1}(w)\hat{E}_i(z).\]
	The case of $i\geq n+1$ follows in the same way. 
 
 Finally, let us prove the case $i=n-1$ (the case of $i=n$ will follow similarly). By definition, we have 
	\begin{align*}
		&(\hat{B}_n(w)\hat{E}_{n-1}(z)f)(x_{[\bfv]^\fc})\\
		&= \sum_{s\in [\bfv]_n}\delta(q^nx_sw)\theta_1(qx_s^2)\Phi_{[\bfv]_n\setminus\{s\}}(qx_s)\cdot (\hat{E}_{n-1}(z)f)(x_{\iota_s[\bfv]^\fc})\\
		&= \sum_{s\in [\bfv]_n}\sum_{r\in [\bfv]_{n-1}}\delta(q^{n-1}x_rz)\delta(q^nx_sw)\Phi_{[\bfv]_{n-1}\setminus\{r\}}(qx_r)\theta_1(qx_s^2)\Phi_{[\bfv]_n\setminus\{s\}}(qx_s)\cdot f(x_{\tau_r^+\iota_s[\bfv]^\fc}).
	\end{align*}
	On the other hand, we have
	\begin{align*}
		&(\hat{E}_{n-1}(z)\hat{B}_n(w)f)(x_{[\bfv]^\fc})\\
		&= \sum_{r\in [\bfv]_{n-1}}\delta(q^{n-1}x_rz)\Phi_{[\bfv]_{n-1}\setminus\{r\}}(qx_r)\cdot (\hat{B}_n(w)f)(x_{\tau_r^+[\bfv]^\fc})\\
		&= \sum_{r\in [\bfv]_{n-1}}\sum_{s\in [\bfv]_n\cup\{r\}}\delta(q^{n-1}x_rz)\delta(q^nx_sw) \times \\
  &\qquad\qquad\qquad \Phi_{[\bfv]_{n-1}\setminus\{r\}}(qx_r)\theta_1(qx_s^2)\Phi_{[\bfv]_n\cup\{r\}\setminus\{s\}}(qx_s)\cdot f(x_{\iota_s\tau_r^+[\bfv]^\fc})
  \\
		&= \theta_1(z/w)(\hat{B}_n(w)\hat{E}_{n-1}(z)f)(x_{[\bfv]^\fc})\\
		&+\sum_{r\in [\bfv]_{n-1}}\delta(q^{n-1}x_rz)\delta(q^nx_rw)\Phi_{[\bfv]_{n-1}\setminus\{r\}}(qx_r)\theta_1(qx_r^2)\Phi_{[\bfv]_n}(qx_r)\cdot f(x_{\iota_r\tau_r^+[\bfv]^\fc}).
	\end{align*} 
	Since $(q^{-1}z-w)\delta(q^{n-1}x_rz)\delta(q^nx_rw)=0$, we have 
	\[(q^{-1}z-w)\hat{E}_{n-1}(z)\hat{B}_n(w)=(z-q^{-1}w)\hat{B}_n(w)\hat{E}_{n-1}(z).\]
	
\underline{Case (c): $c_{ij}=2$}. In this case, $i\neq n$.
	Let us first consider the case for $1\leq i\leq n-1$. By definition,
	\begin{align*}
	&(\hat{E}_i(z)\hat{E}_i(w)f)(x_{[\bfv]^\fc})\\
	&= \sum_{r, s\in [\bfv]_i, r\neq s}\delta(q^ix_rz)\delta(q^ix_sw)\Phi_{[\bfv]_i\setminus \{r\}}(qx_r)\Phi_{[\bfv]_i\setminus \{r,s\}}(qx_s)f(x_{\tau_s^+\tau_r^+[\bfv]^c})\\
	&= \theta_1(qw/z)\sum_{r, s\in [\bfv]_i, r\neq s}\delta(q^ix_rz)\delta(q^ix_sw)\Phi_{[\bfv]_i\setminus \{r,s\}}(qx_r)\Phi_{[\bfv]_i\setminus \{r,s\}}(qx_s)f(x_{\tau_s^+\tau_r^+[\bfv]^c}).
    \end{align*}
	Hence,
	\[(q^2z-w)\hat{E}_i(z)\hat{E}_i(w)=(z-q^2w)\hat{E}_i(w)\hat{E}_i(z).\] 
	The case for $n+1\leq i$ is entirely similar.

\subsection{Relation \eqref{iDR BBii}} Recall Relation \eqref{iDR BBii} states that $(q^2z-w) \bB_n(z) \bB_n(w) +(q^2w-z) \bB_n(w) \bB_n(z)= \frac{\bDel(zw)  \K_{n} q^{-2}}{q-q^{-1}} \frac{(z-q^2w)(w-q^2z)}{z-w} (\bTh_n(z) -\bTh_n(w))$.

It follows by definition that
\begin{align*}
	&\hat{B}_n(z)\hat{B}_n(w)f\\
	&= \sum_{j\in [\bfv]_n}\delta(zx_jq^n)\theta_1(qx_j^2)\Phi_{[\bfv]_n\setminus\{j\}}(qx_j)\cdot \iota_j(\hat{B}_n(w)f)\\
	&= \sum_{i\neq j\in [\bfv]_n}\delta(zx_jq^n)\delta(wx_iq^n)\theta_1(qx_j^2)\Phi_{[\bfv]_n\setminus\{j,i\}}(qx_j) \times \\
&\qquad\qquad \theta_1(qx_j/x_i)\theta_1(qx_i^2)\Phi_{[\bfv]_n\setminus\{i,j\}}(qx_i)\theta_1(qx_ix_j)\cdot \iota_j\iota_i f\\
	&\quad +\sum_{j\in [\bfv]_n}\delta(zx_jq^n)\delta(wx_j^{-1}q^n)\theta_1(qx_j^2)\theta_1(qx_j^{-2})\Phi_{[\bfv]_n\setminus\{j\}}(qx_j)\Phi_{[\bfv]_n\setminus\{j\}}(qx_j^{-1}) f.
\end{align*}
Using the identities 
\begin{align*}
\delta(zx_jq^n)\delta(wx_iq^n)\theta_1(qx_j/x_i) &=\delta(zx_jq^n)\delta(wx_iq^n)\theta_1(qw/z),
\\
\delta(zx_jq^n)\delta(wx_j^{-1}q^n)\theta_1(qx_j^2) &=\delta(zx_jq^n)\delta(wx_j^{-1}q^n)\theta_1(qw/z), 
\end{align*}
we have
\begin{align*}
	&(q^2z-w)\hat{B}_n(z)\hat{B}_n(w)f+(q^2w-z)\hat{B}_n(w)\hat{B}_n(z)f\\
	&= \frac{(q^2z-w)(q^2w-z)}{qw-qz}\bigg(\sum_{j\in [\bfv]_n}\delta(zx_jq^n)\delta(wx_j^{-1}q^n)\Phi_{[\bfv]_n\setminus\{j\}}(qx_j)\Phi_{[\bfv]_n}(qx_j^{-1}) \\
	&-\sum_{j\in [\bfv]_n}\delta(wx_jq^n)\delta(zx_j^{-1}q^n)\Phi_{[\bfv]_n\setminus\{j\}}(qx_j)\Phi_{[\bfv]_n}(qx_j^{-1}) \bigg)f\\
	&= \frac{f}{q-q^{-1}}\frac{(q^2z-w)(q^2w-z)}{qw-qz}\Res'\bigg(\frac{\delta(zxq^n)\delta(wx^{-1}q^n)\Phi_{[\bfv]_n}(qx)\Phi_{[\bfv]_n}(qx^{-1})}{x}\bigg)\\
	&= \frac{f\delta(q^Nzw)}{q-q^{-1}}\frac{(q^2z-w)(q^2w-z)}{qw-qz} \times \\
  &\qquad \qquad \bigg(\big(\Phi_{[\bfv]_n}(qx)\Phi_{[\bfv]_n}(qx^{-1})\big)^+|_{x=q^{-n}z^{-1}}-\big(\Phi_{[\bfv]_n}(qx)\Phi_{[\bfv]_n}(qx^{-1})\big)^-|_{x=q^{n}w}\bigg)
 \\
	&= \frac{f\bDel(zw)}{q-q^{-1}}\frac{(q^2z-w)(q^2w-z)}{qw-qz} \times \\
 &\qquad\qquad \bigg(\Phi_{[\bfv]_n}(q^{1-n}z^{-1})\Phi_{[\bfv]_n}(q^{1+n}z)-\Phi_{[\bfv]_n}(q^{1+n}w)\Phi_{[\bfv]_n}(q^{1-n}w^{-1})\bigg)\\
	&= \frac{\bDel(zw)  \hat{\K}_{\alpha_n} q^{-2}}{q-q^{-1}} \frac{(z-q^2w)(w-q^2z)}{z-w} (\hat{\bTh}_n(z) -\hat{\bTh}_n(w))f,
\end{align*}
where $\Res'$ denotes the sum of the residues of $\frac{\delta(zxq^n)\delta(wx^{-1}q^n)\Phi_{[\bfv]_n}(qx)\Phi_{[\bfv]_n}(q^{-1})}{x}$ at the singular points $x_j$ and $x_j^{-1}$ for $j\in [\bfv]_n$, the third equality follows from the residue theorem, and $(\Phi_{[\bfv]_n}(qx)\Phi_{[\bfv]_n}(qx^{-1}))^+$ (respectively, $(\Phi_{[\bfv]_n}(qx)\Phi_{[\bfv]_n}(qx^{-1}))^-$) denotes the expansion at $x=\infty$ (respectively, $x=0$).

\section{A K-theoretic realization of the affine iquantum group}
 \label{sec:K}

In this section, we first introduce certain $G\times\bbC^*$-equivariant sheaves on the Steinberg variety $\calZ$.
Through the convolution action $\underline{K}^{G\times \bbC^*}(\calZ) \curvearrowright \underline{K}^{G\times \bbC^*}(\calM)$, we show that the class of these sheaves act on 
\begin{equation}\label{equ:KM}
\underline{K}^{G\times \bbC^*}(\calM)\simeq \underline{K}^{G\times \bbC^*}(\calF)\simeq \underline{\bfP}
\end{equation}
via the operators on $\underline{\bfP}$ defined in \S\ref{subsec:poly}. Then
we construct an algebra homomorphism from the iquantum group $\tUi$ to the convolution algebra  $\underline{K}^{G\times \bbC^*}(\calZ)$. 

\subsection{The $\bTh$ operator}\label{subsec:Theta}

Recall that $E_{ij}^{\theta}, E_{ij}^{\theta}(\bfv,a)$ are defined in \eqref{eq:Eijv}, and recall the standard basis $\bfe_i$ $(1\le i\le N)$ for $\C^{N}$. 

Recall the functions $\Phi_{i,\bfv}$ introduced in Equation \eqref{equ:Phi}. For $1\leq i\leq n-1$ (and hence $n+1 \le \tau i \le 2n-1$) and $k\ge 1$, we denote by $\hat{\bTh}_{i,\bfv,k}$ (respectively, $\hat{\bTh}_{\tau i,\bfv, k}$) the coefficient of $z^k$ in the series expansion of 
\[
(q-q^{-1})^{-1} q^{v_{i+1}-v_i} \cdot \Phi_{[\bfv]_i}(q^{1-i}z^{-1})\Phi_{[\bfv]_{i+1}}(q^{-1-i}z^{-1})^{-1}
\] 
\big(respectively, $(q-q^{-1})^{-1} q^{v_{i}-v_{i+1}} \cdot \Phi_{[\bfv]_i}(q^{1-i+N}z)\Phi_{[\bfv]_{i+1}}(q^{-1-i+N}z)^{-1}$\big) at $z=0$. 
Denote by $\hat{\bTh}_{n,\bfv,k}$ the coefficient of $z^k$ in the expansion of $(q-q^{-1})^{-1} \cdot 
\Phi_{[\bfv]_n}(q^{1-n}z^{-1})\Phi_{[\bfv]_n}(q^{1+n}z)$ at $z=0$, for $k\ge 1$.

By construction, we have $\hat{\bTh}_{i,\bfv, k}\in \mbf R_\bbA^{[\mbf v]^{\fc}}$, for $1\leq i\leq 2n-1$ and $k\ge 1$. Denote
\[
\hat{\bTh}_{i,k}:=\sum_{\bfv}\hat{\bTh}_{i,\bfv,k}.
\]
The cotangent bundle $\calM$ embeds diagonally into the Steinberg variety $\calZ$, and by the isomorphism \eqref{equ:KM}, $\hat{\bTh}_{i,k}$ can also be regarded as equivariant K-theory class on $\calZ$, i.e., $\hat{\bTh}_{i,k} \in K^{G\times \bbC^*}(\calZ)$.  

\subsection{The $E$ operators}
  \label{subsec:Eoperators}

Let $1\leq i\leq n-1$. One sees that $E_{i,i+1}^{\theta}(\bfv,1)$ is minimal under the order on $\Xi_d$ defined in Equation \eqref{equ:order}. Therefore, the orbit $\mathcal{O}_{E_{i,i+1}^{\theta}(\bfv,1)}$ is closed by Proposition \ref{prop:orbit}.

By definition, we have
\begin{align*}
\mathcal{O}_{E_{i,i+1}^{\theta}(\bfv,1)}&= \bigg\{(F,F')\mid \substack{F=(V_k)_k\in \mcal{F}_{\bfv + \mathbf{e}_i  + \mathbf{e}_{N + 1-i} }\\ F'=(V'_k)_k\in \mcal{F}_{\bfv + \mathbf{e}_{i+1}  + \mathbf{e}_{N -i}}}, V'_i \stackrel{1}{\subset} V_i, V_k=V'_k \textit{ if } k\neq i,N-i \bigg\}.
\end{align*}
Let $\calL$ be the line bundle on $\mathcal{O}_{E_{i,i+1}^{\theta}(\bfv,1)}$ whose fiber at $(F,F')$ is $V_i/V'_i$. Under the isomorphism $K^{G}(\mathcal{O}_{E_{i,i+1}^{\theta}(\bfv,1)})\cong \mathbf{R}^{[E_{i,i+1}^{\theta}(\bfv,1)]^{\fc}}$ in Proposition \ref{prop:pushforward}(a), this line bundle $\calL$ corresponds to $x_{1+\bar{v}_{i}}\in \mathbf{R}^{[E_{i,i+1}^{\theta}(\bfv,1)]^{\fc}}$. Notice that $[E_{i,i+1}^{\theta}(\bfv,1)]_{i,i+1}=[1+\bar{v}_{i},1+\bar{v}_{i}]$, 
$x_{1+\bar{v}_{i}}\in \mathbf{R}^{[E_{i,i+1}^{\theta}(\bfv,1)]^{\fc}}$.

Let $p_1: \mathcal{O}_{E_{i,i+ 1}^{\theta}(\bfv,1)} \rightarrow \mcal{F}_{\bfv + \mathbf{e}_i  + \mathbf{e}_{N + 1-i} }$ be the first projection. Let $T^{*}_{p_1}$ be the relative cotangent sheaf along the fibers of $p_1$ and $\mathrm{Det}(T^{*}_{p_1})$ be the determinant line bundle on $\mathcal{O}_{E_{i,i+ 1}^{\theta}(\bfv,1)}$. The fiber of $p_1$ at $F\in \mcal{F}_{\bfv + \mathbf{e}_i  + \mathbf{e}_{N + 1-i} }$ is the Grassmannian $\Gr(v_i, V_i/V_{i-1})$. Therefore, we have 
\[T^*_{p_{1}}=\sum_{ \bar{v}_{i-1} < t \leq  \bar{v}_{i}} x_t/x_{\bar{v}_{i}+1}\in K^{G}(\mathcal{O}_{E_{i,i+1}^{\theta}(\bfv,1)}).\]

Let $\pi: \mcal{Z}_{E_{i,i+1}^{\theta}(\bfv,1)} \rightarrow \mathcal{O}_{E_{i,i+ 1}^{\theta}(\bfv,1)}$ be the projection. For any $1\leq i\leq n-1$ and $k\in \bbZ$, define
\begin{align}\label{equ:Eivk}
    \mathscr{E}_{i,\bfv,k} &=  \pi^{*} \big(\mathrm{Det} \, T^{*}_{p_1} \otimes \calL^{\otimes k} \big)\in K^{G\times \bbC^*}(\calZ_{E_{i,i+1}^{\theta}(\bfv,1)}),
\\
\mathscr{E}_{i,k} &= \sum\limits_{\bfv} (-q)^{-v_i}\mathscr{E}_{i,\bfv,k}\in K^{G\times \bbC^*}(\calZ). \notag
\end{align}

Let us compute the convolution operators on $K^{G\times \bbC^*}(\calM)\simeq \oplus_{\mbf v\in \Lambda_{\fc,d}} \bfR^{[\mbf v]^{\fc}}[q,q^{-1}]$ (see \eqref{KMR}) corresponding to  $\mathscr{E}_{i,k}$. Recall the operators 
\[
\hat{E}_{i,k}\in \bigoplus_{\mbf v\in \Lambda_{\fc,d}} \Hom_{R(G)[q,q^{-1}]} \big(\bfR^{[\mbf v']^{\fc}}[q,q^{-1}],\mathbf{R}^{[\bfv]^{\fc}}[q,q^{-1}] \big)
\]
from Equation \eqref{equ:Ei}, where $\mathbf{v'}=\bfv-\bfe_i+\bfe_{i+1}+\bfe_{2n-i}-\bfe_{2n+1-i}$.

\begin{prop}\label{prop:actions1}
	For any $1 \leq i \leq n-1$ and $k\in\bbZ$, the convolution action of $\mathscr{E}_{i,k}$ on $K^{G\times \bbC^*}(\calM)\simeq \oplus_{\mbf v\in \Lambda_{\fc,d}} \bfR^{[\mbf v]^{\fc}}[q,q^{-1}]$ is given by the operator $\hat{E}_{i,k}$.
\end{prop}
\begin{proof}
	Let $p_2: \mathcal{O}_{E_{i,i+1}^{\theta}(\bfv-\bfe_i-\bfe_{2n+1-i},1)} \rightarrow \mcal{F}_{\bfv'}$ be the second projection. For any $f \in  K^{\mathbb{C}^{*}\times G}(\calF_{\bfv'})\simeq \bfR^{[\bfv']^{\fc}}[q,q^{-1}]$, the convolution action of $(-q)^{-v_i+1}\mathscr{E}_{i,\bfv-\bfe_i-\bfe_{2n+1-i},k}$ is given by 
	\begin{align*}
		&(-q)^{-v_i+1}\mathscr{E}_{i,\bfv-\bfe_i-\bfe_{2n+1-i},k}\star f\\
		&=  (-q)^{-v_i+1}Rp_{1*}\bigg(\bigwedge\nolimits_{q^2}T_{p_1}\otimes p_2^*f\otimes \pi^{*}(\mathrm{Det}\,T^{*}_{p_1} \otimes \calL^{\otimes k}) \bigg)\\
		&= (-q)^{-v_i+1}W_{[\bfv]^{\fc}}/W_{[E_{i,i+1}^{\theta}(\bfv-\bfe_i-\bfe_{2n+1-i},1)]^{\fc}}\bigg(\frac{\bigwedge\nolimits_{q^2}T_{p_1}\otimes p_2^*f\otimes \pi^{*}(\mathrm{Det}\,T^{*}_{p_1}\otimes \calL^{\otimes k})}{\bigwedge (T^*_{p_1})}\bigg)\\
		&= S_{[\bar{v}_{i-1}+1,\bar{v}_i]}/S_{[\bar{v}_{i-1}+1,\bar{v}_i-1]}\bigg(x_{\bar{v}_{i}}^k\prod_{\bar{v}_{i-1} < t \leq \bar{v}_{i}-1} \frac{q-q^{-1}x_{t}/x_{\bar{v}_{i}}}{1-x_t/x_{\bar{v}_{i}}} \cdot f\bigg)\\
		&= \sum_{j\in [\bfv]_i}s_{j,\bar{v}_i}\Big(x_{\bar{v}_{i}}^k\Phi_{[\bfv]_i\setminus\{\bar{v}_i\}}(qx_{\bar{v}_i}) \cdot f\Big)\\
		&= \sum_{j\in [\bfv]_i}x_{j}^k\Phi_{[\bfv]_i\setminus\{j\}}(qx_j) \cdot s_{j,\bar{v}_i}f.
	\end{align*}
	Here the first equality follows from Lemma \ref{lem:convolution}, the second one follows from Proposition \ref{prop:pushforward}(b), and the third one is due to $T^*_{p_{1}}=\sum_{ \bar{v}_{i-1} < t \leq  \bar{v}_{i}-1} x_t/x_{\bar{v}_{i}}$. Since $(s_{j,\bar{v}_i}f)(x_{[\bfv]^\fc})=f(x_{\tau^+_j[\bfv]^\fc})$, we can convert the last formula above to $\hat{E}_{i,k}$ in \eqref{equ:Ei} as desired. 
\end{proof}

\subsection{The $B_n$ operator}
  \label{subsec:Boperator}
  
The matrix $E^\theta_{n,n+1}(\bfv,1)$ is not minimal as $\diag(\bfv+\bfe_n+\bfe_{n+1})\prec E^\theta_{n,n+1}(\bfv,1)$. Thus,
\[\overline{\calO_{E^\theta_{n,n+1}(\bfv,1)}}=\Delta\calF_{\bfv+\bfe_n+\bfe_{n+1}}\sqcup\calO_{E^\theta_{n,n+1}(\bfv,1)},\]
where $\Delta\calF_{\bfv+\bfe_n+\bfe_{n+1}}$ denotes the diagonal copy of $\calF_{\bfv+\bfe_n+\bfe_{n+1}}$ inside $\calF_{\bfv+\bfe_n+\bfe_{n+1}}\times \calF_{\bfv+\bfe_n+\bfe_{n+1}}$. More explicitly, we have
\begin{align*}
	\mathcal{O}_{E_{n,n+1}^{\theta}(\bfv,1)}&= \bigg\{\substack{F=( 0=V_0 \subset V_1 \subset \cdots \subset V_{N}=V)\in \mcal{F}_{\bfv + \mathbf{e}_n  + \mathbf{e}_{n+1} }\\ F'=( 0=V'_0 \subset V'_1 \subset \cdots \subset V'_{N}=V)\in \mcal{F}_{\bfv + \mathbf{e}_{n}  + \mathbf{e}_{n+1}}}\ \mid \\
	& \dim V_n/V_n\cap V'_n=1, V_k=V'_k \textit{ if } k\neq n\bigg\}.
\end{align*}
Let $p_1$ and $p_2$ denote its projections to $\calF_{\bfv+\bfe_n+\bfe_{n+1}}$. Then for any $F=( 0=V_0 \subset V_1 \subset \cdots \subset V_{N}=V)\in \mcal{F}_{\bfv + \mathbf{e}_n  + \mathbf{e}_{n+1} }$, we have
\[
p_1^{-1}(F)\simeq\{V'_n\subset V \mid V'_n=(V'_n)^{\perp}, \dim V_n/V_n\cap V'_n=1 \}.
\]
To fix a point $V_n'$ on the right hand side, we can first choose a $v_n$-dimensional subspace of $V_n/V_{n-1}$, which gives $V_n'\cap V_n$. Then we can choose a line in $(V_n\cap V'_n)^\perp/(V_n\cap V'_n)$. Since $\dim(V_n\cap V'_n)^\perp/(V_n\cap V'_n)=2$ and $V_n'\neq V_n$, the fiber $p_1^{-1}(F)$ is isomorphic to an $\bbA^1$-bundle over the Grassmannian $\Gr(v_n, V_n/V_{n-1})$. Thus, the class of the  relative bundle is
\[T_{p_1}=\sum_{t\in [\bfv]_n}\frac{x_{1+\bar{v}_n}}{x_t}+x^2_{1+\bar{v}_n}=\sum_{t\in [\bfv]_n}\frac{x_d}{x_t}+x^2_d\in K^{G\times \bbC^*}(\mathcal{O}_{E_{n,n+1}^{\theta}(\bfv,1)}).\]
 
Let $\calL$ be the line bundle over $\mathcal{O}_{E_{n,n+1}^{\theta}(\bfv,1)}$ whose fiber over $(F,F')$ is $V_n/V_n\cap V'_n$. Then 
\[\calL=x_{1+\bar{v}_n}=x_d\in K^{G\times \bbC^*}(\mathcal{O}_{E_{n,n+1}^{\theta}(\bfv,1)}).\]
Let $\pi:\calZ_{E_{n,n+1}^{\theta}(\bfv,1)}\rightarrow \mathcal{O}_{E_{n,n+1}^{\theta}(\bfv,1)}$ denote the projection. Define
\[\mathscr{B}_{n,\bfv,k}:=\pi^*(\mathrm{Det} T_{p_1}^*\otimes \calL^{\otimes k}),\quad \textit{and \quad} \mathscr{B}_{n,k}=\sum_{\bfv}(-q)^{-v_n-1}\calE_{n,\bfv,k}.\]

Although $p_1$ is not proper, the pushforward $p_{1,*}$ can still be defined using localization, and it is given by the same formula in Proposition \ref{prop:pushforward}(b). Recall the operator 
\[
\hat{B}_{n,k}\in \bigoplus_{\mbf v\in \Lambda_{\fc,d}} \Hom_{\underline{R}(G)(q)} \big(
\underline{\bfR}^{[\mbf v]^{\fc}},\underline{\bfR}^{[\bfv]^{\fc}} \big)
\]
defined in Equation \eqref{equ:En}.

\begin{prop}\label{prop:actions3}
	The convolution action of $\mathscr{B}_{n,k}$ on $\underline{\bfP}$ is given by the operator $\hat{B}_{n,k}$.
\end{prop}
\begin{proof}
	Let $p_1,p_2: \mathcal{O}_{E_{n,n+1}^{\theta}(\bfv-\bfe_n-\bfe_{n+1},1)} \rightarrow \mcal{F}_{\bfv}$ be the two projections. We will use the notations in the proof of Proposition \ref{prop:pushforward} for the flags. The base point $(F,F')\in \mathcal{O}_{E_{n,n+1}^{\theta}(\bfv-\bfe_n-\bfe_{n+1},1)}$ with $F=(V_i)$ and $F'=(V_i')$ satisfies that $V_i=V'_i$ for $i\neq n$, $V_n/V_{n-1}=\Span\{\epsilon_i\mid 1+\bar{v}_{n-1}\leq i\leq d\}$, and $V'_n/V'_{n-1}=\Span\{\epsilon_{2d},\epsilon_i\mid 1+\bar{v}_{n-1}\leq i\leq d-1\}$. Thus, we have
	\[F'=\iota_d(F).\]
	Recall $\iota_d\in W_\fc=\bbZ_2^d\rtimes S_d$ is the non-trivial element in the d-th copy of $\bbZ_2$.
	
	Given any $\mathscr{F} \in  K^{\mathbb{C}^{*}\times G}(\mathscr{F}_{\bfv})$, let $f=\mathscr{F}|_F\in\mathbf{R}^{[\bfv]^{\fc}}[q,q^{-1}]$. Then 
	\[p_2^*(\mathscr{F})|_{(F,F')}=\mathscr{F}|_{F'}=\iota_d(f).\]
	Thus, $p_2^*(\mathscr{F})$ corresponds to $\iota_d(f)$ under the isomorphism in Proposition \ref{prop:pushforward}(b).
	
	The convolution action of $(-q)^{-v_n}\mathscr{B}_{n,\bfv-\bfe_n-\bfe_{n+1},k}$ is given by 
	\begin{align*}
		&(-q)^{-v_n}\mathscr{B}_{n,\bfv-\bfe_n-\bfe_{n+1},k}\star \calF\\
		&= (-q)^{-v_n}Rp_{1*}\Big(\bigwedge\nolimits_{q^2}T_{p_1}\otimes p_2^*\calF\otimes \pi^{*} (\mathrm{Det}\,T^{*}_{p_1}\otimes \calL^{\otimes k}) \Big)\\
		&= (-q)^{-v_n}W_{[\bfv]^{\fc}}/W_{[E_{n,n+1}^{\theta}(\bfv-\bfe_n-\bfe_{n+1},1)]^{\fc}}\bigg(\frac{\bigwedge\nolimits_{q^2}T_{p_1}\otimes p_2^*\calF\otimes \pi^{*} (\mathrm{Det}\,T^{*}_{p_1}\otimes \calL^{\otimes k})}{\bigwedge (T^*_{p_1})}\bigg)\\
		&= S_{[\bar{v}_{n-1}+1,d]}/S_{[\bar{v}_{n-1}+1,d-1]}\bigg(x_d^k\frac{q-q^{-1}x_d^{-2}}{1-x_d^{-2}}\prod_{\bar{v}_{n-1} < t \leq d-1}\frac{q-q^{-1}x_{t}/x_d}{1-x_t/x_d} \cdot \iota_d(f)\bigg)\\
		&= \sum_{j\in [\bfv]_n}s_{j,d}\Big(x_d^k\cdot\theta_1(qx_d^2)\cdot\Phi_{[\bfv]_n\setminus\{d\}}(qx_d) \cdot \iota_d(f)\Big)\\
		&= \sum_{j\in [\bfv]_n}x_{j}^k\cdot\theta_1(qx_j^2)\cdot\Phi_{[\bfv]_n\setminus\{j\}}(qx_j) \cdot \iota_j(f).
	\end{align*}
	Here the first equation follows from Lemma \ref{lem:convolution}, the second one follows from Proposition \ref{prop:pushforward}(b), the third one is due to $T_{p_{1}}=\sum_{ \bar{v}_{n-1} < t \leq  \bar{v}_{n}-1}x_d/x_t+x_d^2$, and the last one follows from the fact that $f\in\mathbf{R}^{[\bfv]^{\fc}}[q,q^{-1}]$. This finishes the proof.
\end{proof}

\subsection{The $F$ operators} \label{subsec:F}
This section is parallel to \S\ref{subsec:Eoperators}. 
For any $1\leq i\leq n-1$, the orbit $\mathcal{O}_{E_{i+1,i}^{\theta}(\bfv,1)}$ is closed, and we have
\begin{align*}
\mathcal{O}_{E_{i+1,i}^{\theta}(\bfv,1)}&=\bigg\{\substack{F=( 0=V_0 \subset V_1 \subset \cdots \subset V_{N}=V)\in \mcal{F}_{\bfv + \mathbf{e}_{i+1}  + \mathbf{e}_{N-i} }\\ F'=( 0=V'_0 \subset V'_1 \subset \cdots \subset V'_{N}=V)\in \mcal{F}_{\bfv + \mathbf{e}_i  + \mathbf{e}_{N+1-i}}}\ \mid \\\
& V_i\subset V'_i, \dim V'_i/V_i=1, V_k=V'_k \textit{ if } k\neq i,N-i\bigg\}.
\end{align*}

Let $q_1: \mathcal{O}_{E_{i+1,i }^{\theta}(\bfv,1)} \rightarrow \mcal{F}_{\bfv + \mathbf{e}_{i+1}  + \mathbf{e}_{N-i} }$ be the first projection. Let $T^{*}_{q_1}$ be the relative cotangent sheaf along the fibers of $q_1$, and $\mathrm{Det}(T^{*}_{q_1})$ be the determinant line bundle on $\mathcal{O}_{E_{i+1,i }^{\theta}(\bfv,1)}$. The fiber of $q_1$ over $F$ is the Grassmannian $\Gr(1, V_{i+1}/V_i)$. Thus 
\[T^*_{q_1}=\sum_{\bar{v}_i+1<t\leq \bar{v}_{i+1}+1}x_{\bar{v}_i+1}/x_t\in K^{G\times \bbC^*}(\mathcal{O}_{E_{i+1,i }^{\theta}(\bfv,1)}).\]

Define a line bundle $\calL'$ on $\mathcal{O}_{E_{i+1,i }^{\theta}(\bfv,1)}$, whose fiber at $(F,F')$ is $V'_i/V_i$. Under the isomorphism $K^{G}(\mathcal{O}_{E_{i+1,i}^{\theta}(\bfv,1)})\cong \mathbf{R}^{[E_{i+1,i}^{\theta}(\bfv,1)]^{\fc}}$ in Proposition \ref{prop:pushforward}(a), the line bundle $\calL'$ corresponds to $x_{1+\bar{v}_{i}}\in \mathbf{R}^{[E_{i+1,i}^{\theta}(\bfv,1)]^{\fc}}$.

Let $\pi':  \mcal{Z}_{E_{i+1,i}^{\theta}(\bfv,1)} \rightarrow  \mathcal{O}_{E_{i+1,i }^{\theta}(\bfv,1)}$ be the projection. Let
\[\mathscr{F}_{i,\bfv,k} =  \pi^{'*}(\mathrm{Det}(T^{*}_{p'_1})\otimes \calL'^{\otimes k})\in K^{G\times \bbC^*}(\calZ_{E_{i+1,i}^{\theta}(\bfv,1)}),\]
and 
\[\mathscr{F}_{i,k}= \sum\limits_{\bfv} (-q)^{-v_{i+1}}\mathscr{F}_{i,\bfv,k}\in K^{G\times \bbC^*}(\calZ).\]

Recall the operators 
\[
\hat{F}_{i,k}\in \bigoplus_{\mbf v\in \Lambda_{\fc,d}} \Hom_{R(G)[q,q^{-1}]}(\bfR^{[\mbf v'']^{\fc}}[q,q^{-1}],\mathbf{R}^{[\bfv]^{\fc}}[q,q^{-1}])
\] 
from Equation \eqref{equ:F}, where $\mathbf{v''}=\bfv+\bfe_i-\bfe_{i+1}-\bfe_{2n-i}+\bfe_{2n+1-i}$.

\begin{prop}\label{prop:actions2}
	For any $1 \leq i \leq n-1$ and $k\in\bbZ$, the convolution action of $\mathscr{F}_{i,k}$ on $K^{G\times\mathbb{C}^{*} }(\calM)\simeq \oplus_{\mbf v\in \Lambda_{\fc,d}} \bfR^{[\mbf v]^{\fc}}[q,q^{-1}]$ is given by the operator $\hat{F}_{i,k}$.
\end{prop}
\begin{proof}
	Let $q_2: \mathcal{O}_{E_{i+1,i}^{\theta}(\bfv-\bfe_{i+1}-\bfe_{2n-i},1)} \rightarrow \mcal{F}_{\bfv''}$ be the second projection. Same argument as in the proof of Proposition \ref{prop:actions1} shows that, for any $f \in  K^{\mathbb{C}^{*}\times G}(\calF_{\bfv''})\simeq \bfR^{[\mbf v'']^{\fc}}[q,q^{-1}]$, 
	\begin{align*}
	&	(-q)^{1-v_{i+1}} \mathscr{F}_{i,\bfv-\bfe_{i+1}-\bfe_{2n-i},k}\star f\\
		&= (-q)^{1-v_{i+1}}W_{[\bfv]^{\fc}}/W_{[E_{i+1,i}^{\theta}(\bfv-\bfe_{i+1}-\bfe_{2n-i},1)]^{\fc}}\bigg(\frac{\bigwedge\nolimits_{q^2}T_{q_1}\otimes q_2^*f\otimes \pi^{*}(\mathrm{Det}(T^{*}_{q_1})\otimes \calL'^{\otimes k}}{\bigwedge (T^*_{q_1})}\bigg)\\
	&= S_{[\bar{v}_i+1,\bar{v}_{i+1}]}/S_{[\bar{v}_i+2,\bar{v}_{i+1}]}\bigg(x_{\bar{v}_{i}+1}^k\prod_{\bar{v}_i+1<t\leq \bar{v}_{i+1}} \frac{q-q^{-1}x_{\bar{v}_{i}+1}/x_{t}}{1-x_{\bar{v}_{i}+1}/x_t} \cdot f\bigg)\\
	&= \sum_{j\in [\bfv]_{i+1}}s_{\bar{v}_i+1,j}\Big(x_{\bar{v}_{i}+1}^k\Phi_{[\bfv]_{i+1}\setminus\{1+\bar{v}_i\}}(q^{-1}x_{1+\bar{v}_i})^{-1}\cdot f\Big)\\
	&= \sum_{j\in [\bfv]_{i+1}}x_j^k \Phi_{[\bfv]_{i+1}\setminus\{j\}}(q^{-1}x_j)^{-1}\cdot s_{\bar{v}_i+1,j}f.
	\end{align*}
	The second equality follows from $T^*_{q_1}=\sum_{\bar{v}_i+1<t\leq \bar{v}_{i+1}}x_{\bar{v}_i+1}/x_t$. Since
	$(s_{\bar{v}_i+1,j}f)(x_{[\bfv]^\fc})=f(x_{\tau^-_j[\bfv]^\fc})$, we convert the last formula to $\hat{F}_{i,k}$ in \eqref{equ:F} as desired. 
\end{proof}

\subsection{The homomorphism $\Psi$}

We are now ready to connect the iquantum group $\tUi$ to the convolution algebra $\underline{K}^{G\times\mathbb{C}^{*}}(\calZ)$.

\begin{theorem}
 \label{thm:krealization}
	Sending $
 C\mapsto q^N, \quad\K_i \mapsto \hat{\K}_i,\quad \Theta_{i,k}\mapsto \hat{\bTh}_{i,k}$,
 and
	\[
 B_{i,k}\mapsto \begin{dcases}q^{ki}\mathscr{E}_{i,k}&\textit{ if } 1\leq i\leq n-1,\\
		q^{kn}\mathscr{B}_{n,k} &\textit{ if } i=n,\\
		q^{ki}\mathscr{F}_{\tau(i),-k} &\textit{ if } 1+n\leq i\leq 2n-1,
	\end{dcases}
 \]
	defines a $\bbC(q)$-algebra homomorphism
    \begin{align} 
        \label{eq:PsiUi}
        \Psi:\tUi\longrightarrow \underline{K}^{G\times\mathbb{C}^{*}}(\calZ).
    \end{align}
\end{theorem}

\begin{proof}
Consider the following diagram
\begin{align*}
\xymatrix{
\tUi \ar[r]    \ar@{..>}[d]
&\End_{\C(q)} (\underline{\bfP}) 
\ar@{=}[d]
  \\
\underline{K}^{G\times\mathbb{C}^{*}}(\calZ) 
\ar@{^{(}->}[r]
& \End_{\C(q)} (\underline{\bfP})
}
\end{align*}
where the map on the top row is given by Theorem \ref{thm:polyrepeven}, the map on the bottom row is given by the convolution action (see Lemma \ref{lemma:faithful} and \eqref{equ:KM}).
Thanks to the explicit formulae for the actions of the $\mathscr{E}, \mathscr{F}, \mathscr{B}, \hat{\bTh}$ operators in \S\ref{subsec:Theta}--\ref{subsec:F}, we see that the correspondence given in the statement of the theorem indeed defines a homomorphism $\tUi\rightarrow \underline{K}^{G\times\mathbb{C}^{*}}(\calZ)$ which makes the above diagram commutative.
\end{proof}

\section{Finite-dimensional $\tUi$-modules}
 \label{sec:modules}

In this section, specializing the homomorphism $\Psi:\tUi \rightarrow K^{G\times\mathbb{C}^{*}}(\calZ)$ in \eqref{eq:PsiUi} at a non-root of unity, we establish a surjective homomorphism $\Psi_a: \tUi_t\twoheadrightarrow  K^{G\times\mathbb{C}^{*}}(\calZ)_a \cong H^{BM}_*(\calZ^a,\bbC)$ in \eqref{eq:Psia} and \eqref{UtoH}. We further construct a family of finite-dimensional standard modules and irreducible modules of $\tUi_t$ and provide a composition multiplicity formula for these standard modules. 

\subsection{Surjectivity of the homomorphism $\Psi_a$}
 \label{subsec:surj}

In this subsection, we pick $a:=(s,t)\in T\times\bbC^*$ satisfying $1-e^\alpha(s)\neq 0$ and $1-t^2e^\alpha(s)\neq 0$ for all roots $\alpha \in R$.

Let $\tUi_t$ denote the specialization of $\tUi$ at $q=t$. Let $\bbC_a$ denote the one-dimensional module of $K^{G\times\bbC^*}(\pt)$ by evaluation at $a$, and 
\[
K^{G\times\mathbb{C}^{*}}(\calZ)_a:=K^{G\times\mathbb{C}^{*}}(\calZ)\otimes_{K^{G\times\bbC^*}(\pt)}\bbC_a.
\]
By the definition of the localized module and the choice of $a$, we have
\begin{align} \label{eq:Kloc}
\underline{K}^{G\times\mathbb{C}^{*}}(\calZ)\otimes_{\underline{K}^{G\times\bbC^*}(\pt)}\bbC_a=K^{G\times\mathbb{C}^{*}}(\calZ)_a.
\end{align}
The specialization of \eqref{eq:PsiUi} at $q=t$ gives us an algebra homomorphism
\begin{align}  \label{eq:Psia}
\Psi_a:\tUi_t \longrightarrow K^{G\times\mathbb{C}^{*}}(\calZ)_a.
\end{align}

\begin{theorem}  \label{thm:surj}
Suppose that $t$ is not a root of unity. Then 
the homomorphism $\Psi_a$ in \eqref{eq:Psia} is surjective.
\end{theorem}

\begin{remark}
\label{rem:at}
    The specialization at $a =(s,t)$, instead of at $t$, is needed in Theorem~ \ref{thm:surj} since a localization at the target space $\underline{K}^{G\times\bbC^*}(\calZ)$ is used in \eqref{eq:Kloc}.
\end{remark}

The remainder of this subsection is devoted to the proof of Theorem \ref{thm:surj}. To that end, we consider the specialization of $K^{G\times\mathbb{C}^{*}}(\calZ)$ (and its localization) at $q=t$, denoted by $K^{G\times\mathbb{C}^{*}}(\calZ)_t$, and the specialized morphism $\Psi_t$. 

Let us first show that $\mathscr{E}_{i,\bfv,k}$, $\mathscr{B}_{n,k}$, $\mathscr{F}_{i,\bfv,k}$, and $\hat{\bTh}_{i,\bfv,k}$ belong to the image of $\Psi_t$.
For any $\bfv\in \Lambda_{\fc,d}$, let 
\[
\bfI_\bfv:=\prod_{i=1}^{n-1}\prod_{\substack{m=-d\\m\neq v_{i+1}-v_i}}^d\frac{\K_i-t^m}{t^{v_{i+1}-v_i}-t^m}\in \tUi_t.
\] 

\begin{lemma}\label{lem:Iv}
    For any $\bfu,\bfv,\bfw\in \Lambda_{\fc,d}$, $\Psi_t(\bfI_\bfv)\star K^{G\times \bbC^*}(\calZ_{\bfu,\bfw})\neq 0$ if and only if $\bfv=\bfu$.
\end{lemma}

\begin{proof}
    Recall that, for $1\leq i\leq n-1$, $\hat{\K}_i$ acts on $\bfR^{[\bfu]^\fc}[q,q^{-1}]$ as scalar multiplication by $q^{u_{i+1}-u_i}$. Thus, for any $\mathscr{F}\in K^{G\times \bbC^*}(\calZ_{\bfu,\bfw})$,
    \begin{align}  \label{eq:starF}
    \Psi_t(\bfI_\bfv)\star \mathscr{F}= \prod_{i=1}^{n-1}\prod_{\substack{m=-d\\m\neq v_{i+1}-v_i}}^d\frac{t^{u_{i+1}-u_i}-t^m}{t^{v_{i+1}-v_i}-t^m}\mathscr{F}.
    \end{align}
    Hence, if $\bfv=\bfu$, the RHS of \eqref{eq:starF} equals $\mathscr{F}$. If $\Psi_t(\bfI_\bfv)\star \mathscr{F}\neq 0$, then by \eqref{eq:starF} we must have $u_{i+1}-u_i=v_{i+1}-v_i$, for $1\leq i\leq n-1$. Since $\sum_{i=1}^nu_i=\sum_{i=1}^nv_i=d$, we have $\bfu=\bfv$.
\end{proof}

The following lemma deals with the diagonal orbits.

\begin{lemma}
 \label{lem:diagonal}
    Let $\bfv\in \Lambda_{\fc,d}$. Suppose that $t$ is not a root of unity. Then the elements $\hat{\bTh}_{i,\bfv,k}$ $(1\leq i\leq 2n-1)$ and the elements $\sum_{j=1}^d(x_j^k+x_j^{-k})$, for $k\ge 1$, generate the algebra $K^{G\times\bbC^*}(\calZ_{\diag(\bfv)})_t\simeq \bfR^{[\bfv]^\fc}$.
\end{lemma}

\begin{proof}
    Recall from \eqref{GF} that 
    \[
    \bTh_{i}(z)=1+ \sum_{k\geq 1} (q-q^{-1})\Theta_{i,k}z^{k} =\exp\Big((q-q^{-1})\sum_{k\geq 1}H_{i,k}z^k\Big).
    \]
    For $1\leq i\leq n-1$, the operator $\hat{\bTh}_{i}(z)$ acts on $\bfR^{[\bfv]^\fc}[q,q^{-1}]$ as the scalar multiplication by $q^{v_{i+1}-v_i}\Phi_{[\bfv]_i}(q^{1-i}z^{-1})\Phi_{[\bfv]_{i+1}}(q^{-1-i}z^{-1})^{-1}$, and $\hat{\bTh}_{\tau i}(z)$ acts on $\bfR^{[\bfv]^\fc}[q,q^{-1}]$ as the scalar multiplication by $q^{v_i-v_{i+1}}\Phi_{[\bfv]_i}(q^{1+N-i})\Phi_{[\bfv]_{i+1}}(q^{-1+N-i}z)^{-1}$. Also, $\hat{\bTh}_n(z)$ acts on $\bfR^{[\bfv]^\fc}[q,q^{-1}]$ as the scalar multiplication by $\Phi_{[\bfv]_i}(q^{1-n}z^{-1})\Phi_{[\bfv]_{n}}(q^{1+n}z)$. A direct computation shows that the operators $\hat{H}_{i,k}$ acts on $K^{G\times\bbC^*}(\calZ_{\diag(\bfv)})_t\simeq \bfR^{[\bfv]^\fc}$ by the following scalar multiplications
    \[
    \hat{H}_{i,k}=\begin{dcases}
        \frac{1}{k}[k]_tt^{(i-1)k}\Big(\sum_{j\in [\bfv]_i}x_j^k-t^{2k}\sum_{j\in [\bfv]_{i+1}}x_j^k\Big), & \quad\text{ for }  1\leq i\leq n-1;
        \\
        \frac{1}{k}[k]_tt^{(n-1)k}\Big(\sum_{j\in [\bfv]_n}x_j^k-t^{2k}\sum_{j\in [\bfv]_n}x_j^{-k}\Big), & \quad\text{ for }  i=n.\\
    \end{dcases}
    \]
    Moreover, we have 
    \[
    \hat{H}_{\tau i,k}=
        \frac{1}{k}[k]_tt^{(N-i-1)k}\Big(\sum_{j\in [\bfv]_{i+1}}x_j^{-k}-t^{2k}\sum_{j\in [\bfv]_i}x_j^{-k}\Big),
        \qquad \text{ for } 1\leq i\leq n-1.
    \]
    
    Therefore, the column vector 
    \[
    \bigg(\frac{k}{[k]_t}\hat{H}_{1,k},\ldots, \frac{k}{[k]_tt^{(n-1)k}}\hat{H}_{n,k}, \frac{k}{[k]_tt^{nk}}\hat{H}_{\tau (n-1),k},\ldots,\frac{k}{[k]_tt^{(N-2)k}}\hat{H}_{\tau 1,k},\sum_{j=1}^d(x_j^k+x_j^{-k})\bigg)^T
    \]
    is related to the vector 
    \[
    \bigg(\sum_{j\in [\bfv]_1}x_j^k,\ldots, \sum_{j\in [\bfv]_n}x_j^k, \sum_{j\in [\bfv]_n}x_j^{-k},\ldots, \sum_{j\in [\bfv]_1}x_j^{-k}\bigg)^T
    \]
    by the matrix
    \[A=\begin{pmatrix}
        1 & c & & & \\
        & 1 & c &&\\
        & & 1 & c &\\
        &&&\ddots &\ddots\\
        1&1&1&\cdots &1
    \end{pmatrix},\]
    where $c=-t^{2k}$. Then $\det(A)=\frac{1-c^{2n}}{1-c}$ is nonzero (since $t$ is not a root of 1) and $A$ is invertible. Since $\big\{\sum_{j\in [\bfv]_i}x_j^{\pm k} \mid 1\le i\le n,k\ge 1 \big\}$ generate the algebra $\bfR^{[\bfv]^\fc}$, we conclude that the algebra $\bfR^{[\bfv]^\fc}$ is generated by $\{\hat{H}_{i,k} \mid 1\le i \le 2n-1, k\ge 1\}$ and the elements $\sum_{j=1}^d(x_j^k+x_j^{-k})$, for $k\ge 1$. The lemma follows by converting $\hat{H}_{i,k}$'s to $\hat{\bTh}_{i,k}$'s. 
\end{proof}

We now deal with the non-diagonal orbits required in Theorem~\ref{thm:generators}.

\begin{lemma}
  \label{lem:E_ijorbit}
    Let $\bfv\in \Lambda_{\fc,d-1}$.
    \begin{enumerate}
        \item 
        For $1\leq i\leq n-1$, $K^{G\times\bbC^*}(\calZ_{E^\theta_{i,i+1}(\bfv,1)})$ (respectively, $K^{G\times\bbC^*}(\calZ_{E^\theta_{i+1,i}(\bfv,1)})$) is contained in the algebra generated by $\mathscr{E}_{i,\bfv,k}$ (respectively, $\mathscr{F}_{i,\bfv,k}$), for $k\in \bbZ$, and the classes of sheaves supported on the diagonal orbits.
        \item
        $K^{G\times\bbC^*}(\calZ_{E^\theta_{n,n+1}(\bfv,1)})$ is contained in the algebra generated by $\mathscr{B}_{n,\bfv,k}$, for $k\in \bbZ$, and the classes of sheaves supported on the diagonal orbits.
    \end{enumerate}
\end{lemma}
\begin{remark}
    In Lemma \ref{lem:E_ijorbit}(2), we consider everything inside $K^{G\times\bbC^*}(\calZ_{E^\theta_{n,n+1}(\bfv,1)})$. By Proposition \ref{prop:diagconv}, this space is stable under the convolution with the classes supported on the diagonal orbits. However, if we want to consider $K^{G\times\bbC^*}(\calZ_{E^\theta_{n,n+1}(\bfv,1)})$ as elements in $K^{G\times\bbC^*}(\calZ)$ via pushforward, we need to use localization.
\end{remark}
\begin{proof}[Proof of Lemma~\ref{lem:E_ijorbit}]
    (1) It suffices to show the statement for the $\mathscr{E}$-operators, as the other case can be proved similarly. First of all, the orbit $\calO_{E^\theta_{i,i+1}(\bfv,1)}$ is closed, and by Proposition \ref{prop:pushforward}, we have
    \[
    K^{G\times\bbC^*}(\calZ_{E^\theta_{i,i+1}(\bfv,1)})\simeq \bfR^{[E^\theta_{i,i+1}(\bfv,1)]^\fc}[q,q^{-1}],
    \]
    where 
    \[
    W_{[E^\theta_{i,i+1}(\bfv,1)]^\fc}=S_{[1,\bar{v}_1]}\times \cdots \times S_{[1+\bar{v}_{i-1},\bar{v}_i]}\times S_{[1+\bar{v}_{i},1+\bar{v}_i]}\times S_{[2+\bar{v}_i,1+\bar{v}_{i+1}]}\times\cdots \times S_{[2+\bar{v}_{n-1},d]}.
    \]
    By the assumption of $\bfv$, $\bfv+\bfe_i+\bfe_{N+1-i}\in \Lambda_{\fc,d}$. Moreover, we have 
    \[K^{G\times\bbC^*}(\calZ_{\diag(\bfv+\bfe_i+\bfe_{N+1-i})})\simeq \bfR^{S_{[1,\bar{v}_1]}\times \cdots \times S_{[1+\bar{v}_{i-1},1+\bar{v}_i]}\times S_{[2+\bar{v}_i,\bar{v}_{i+1}]}\times\cdots \times S_{[2+\bar{v}_{n-1},d]}}[q,q^{-1}],\]
    By the definition of $\mathscr{E}_{i,\bfv,k}$ in \eqref{equ:Eivk}, we have 
    \[
    \mathscr{E}_{i,\bfv,k}=\prod_{1+\bar{v}_{i-1}\leq j\leq \bar{v}_i}\frac{x_j}{x_{\bar{v}_i+1}}\cdot x_{\bar{v}_i+1}^k\in K^{G\times\bbC^*}(\calZ_{E^\theta_{i,i+1}(\bfv,1)}).
    \]
    Therefore, by Proposition \ref{prop:diagconv},
    \[
    \prod_{1+\bar{v}_{i-1}\leq j\leq 1+\bar{v}_i}x_j^{-1}\star\mathscr{E}_{i,\bfv,k}=x_{\bar{v}_i+1}^{k-v_i-1}\in K^{G\times\bbC^*}(\calZ_{E^\theta_{i,i+1}(\bfv,1)}).
    \]
    Since $K^{G\times\bbC^*}(\calZ_{E^\theta_{i,i+1}(\bfv,1)})$ is generated by $K^{G\times\bbC^*}(\calZ_{\diag(\bfv+\bfe_i+\bfe_{N+1-i})})$ and $x_{\bar{v}_i+1}^k$ $(k\in\bbZ)$, Part~ (1) follows.
    
    Part (2) can be proved in the same way.
\end{proof}

Now we can finish the proof of Theorem \ref{thm:surj}.

\begin{proof}[Proof of Theorem \ref{thm:surj}]
Upon specialization at $s\in G$, the elements $\sum_{j=1}^d(x_j^k+x_j^{-k})$ in Lemma~\ref{lem:diagonal} specialize to scalars. Now Theorem~\ref{thm:surj} follows from Theorem \ref{thm:generators}, Lemma~ \ref{lem:Iv}, Lemma \ref{lem:diagonal}, and Lemma~ \ref{lem:E_ijorbit}.
\end{proof}

\subsection{Reduction to the homology case}\label{sec:redu}
Recall $a:=(s,t)\in T\times\bbC^*$ as in \S\ref{subsec:surj} and $t$ is not a root of unity. Let $A$ be the subgroup of $G\times\bbC^*$ generated by $a$. Then $\calZ^A=\calZ^a$, where $\calZ^A$ (respectively, $\calZ^a$) denotes the fixed loci of $\calZ$ under the action of $A$ (respectively, $a$). We have the following chain of algebra isomorphisms
\begin{align*}
K^{G\times\bbC^*}(\calZ)_a &:= K^{G\times\bbC^*}(\calZ)\otimes_{K^{G\times\bbC^*}(\pt)}\bbC_a
\\
 & \simeq  K^{A}(\calZ)\otimes_{R(A)}\bbC_a
	\stackrel{r_a}{\xrightarrow{\sim}} K_\bbC(\calZ^a)
	\stackrel{RR}{\xrightarrow{\sim}} H^{BM}_*(\calZ^a,\bbC).
\end{align*}
Here $H^{BM}_*(\calZ^a,\bbC)$ denotes the Borel--Moore homology of $\calZ^a$, which also has a convolution algebra structure, see \cite[Chapter 2]{CG97}. The first isomorphism follows from \cite[Theorem 6.2.10]{CG97}. The map $r_a$ (respectively, $RR$) is the bivariant localization map from Theorem 5.11.10 (respectively, bivariant Riemann--Roch map from Theorem 5.11.11) \textit{loc. cit.}. All these maps respect the convolution algebra structures. Composing with the surjective algebra homomorphism $\Psi_a$ from Theorem \ref{thm:surj}, we get a surjective algebra homomorphism 
\begin{align}  \label{UtoH}
\tUi_t\twoheadrightarrow  H^{BM}_*(\calZ^a,\bbC).
\end{align}
Therefore, every representation of the convolution algebra $H^{BM}_*(\calZ^a,\bbC)$ pulls back to a representation of $\tUi_t$. Since the homomorphism in \eqref{UtoH} is surjective, the pullbacks of irreducible representations will remain irreducible.

\subsection{Representations of a convolution algebra}
We recall Ginzburg's construction of irreducible representations for convolution algebras, see \cite[\S8.6]{CG97}. Given two graded vector spaces $V$ and $W$, we write $V\stackrel{.}{=}W$ if there is a linear isomorphism that does not necessarily preserve the gradings. We also use the same notation to denote that two objects are quasi-isomorphic up to a shift in the derived category.

Let $\mu:M\rightarrow N$ be a proper map and $M$ is smooth (though possibly disconnected), and let $Z:=M\times_N M$. Then $H^{BM}_*(Z,\bbC)$ has a convolution algebra structure. 

Let $\calD^b(N)$ be the bounded derived category of complexes of sheaves with constructible cohomology sheaves. Then $\Ext^*_{\calD^b(N)}(\mu_*\underline{\bbC}_M,\mu_*\underline{\bbC}_M)$
has an algebra structure via the Yoneda product,
and we have the following algebra isomorphism 
\[H^{BM}_*(Z,\bbC)\stackrel{.}{=}\Ext^*_{\calD^b(N)}(\mu_*\underline{\bbC}_M,\mu_*\underline{\bbC}_M).\]

By the BBDG decomposition theorem, we have the following isomorphism
\[\mu_*\underline{\bbC}_M\simeq \bigoplus_{k\in \bbZ, \phi}L_\phi(k)\otimes \IC_\phi[k]\in \calD^b(N),\]
where $L_{\phi}(k)$ is a vector space, and $\IC_\phi$ is some simple perverse sheaf on $N$ such that some shift of it is a direct summand of $\mu_*\underline{\bbC}_M$. Let $L_\phi:=\oplus_k L_\phi(k)$. Applying this decomposition to the above isomorphism and using some property of the IC sheaves, we obtain that 
\[
H^{BM}_*(Z,\bbC)\stackrel{.}{=}\bigg(\bigoplus_\phi \End L_\phi\bigg)\bigoplus \bigg(\bigoplus_{\phi,\psi,k>0}\Hom_\bbC(L_\phi,L_\psi)\otimes \Ext^k_{\calD^b(N)}(\IC_\phi,\IC_\psi)\bigg).\]
The first sum is a direct sum of matrix algebras, hence semisimple. The second sum is concentrated in degrees $k>0$ and is the radical of the algebra $H^{BM}_*(Z,\bbC)$. Therefore, $\{L_\phi \mid L_\phi \neq 0\}$ forms a complete set of the isomorphism classes of simple $H^{BM}_*(Z,\bbC)$-modules.

There is also an equivariant version of this. Suppose a linear algebra group $H$ acts on $M$ and $N$ such that $\mu: M\rightarrow N$ is $H$-equivariant. Further, assume that there are only finitely many $H$-orbits on $N$. Then the data $\phi$ in the decomposition theorem is $\phi=(\calO_\phi,\chi_\phi)$, where $\calO_\phi\subset N$ is an $H$-orbit, while $\chi_\phi$ is an irreducible $H$-equivariant local system on $\calO_\phi$. Recall that $\chi_\phi$ corresponds to some irreducible representation of the component group $H(x)/H(x)^\circ$ of the stabilizer subgroup $H(x)$ of a point $x$ in the orbit $\calO_\phi$. Let $M_x$ denote the fiber $\mu^{-1}(x)$. Then the homology $H_*(M_x)$ has a commuting action of $H(x)/H(x)^\circ$ and $H^{BM}_*(Z,\bbC)$. We let $H_*(M_x)_\phi$ denote the $\chi_\phi$-isotypical component of $H_*(M_x)$. 

For any two parameters $\phi=(\calO_\phi,\chi_\phi)$ and $\psi=(\calO_\psi,\chi_\psi)$, choose a point $x\in \calO_\phi$ and let $i_x:\{x\}\hookrightarrow N$ denote the inclusion. 

\begin{prop} \cite[Theorem 8.6.23]{CG97}\label{prop:mult}
    The multiplicity of the simple $H^{BM}_*(Z,\bbC)$-module $L_\psi$ in the composition series of $H_*(M_x)_\phi$ is given by
    \[[H_*(M_x)_\phi: L_\psi]=\sum_k\dim H^k(i_x^!\IC_\psi)_\phi,\]
    where $H^k(i_x^!\IC_\psi)_\phi$ denotes the $\chi_\phi$-component of $H^k(i_x^!\IC_\psi)$.
\end{prop}

\subsection{Standard modules and irreducible modules}

We apply the above constructions to our case. Recall $a=(s,t)\in T\times \bbC^*$. Let $G(s)\subset G$ be the centralizer of $s$. By definition, 
\[\calN^a=\{x\in \calN\mid sxs^{-1}=t^{-2}x\}.\]
Let $\calM^a:=\bigsqcup_\bfv \calM_\bfv^a$ be the fixed loci. Then the map $\pi:\calM^a\rightarrow \calN^a$ is $G(s)$-equivariant. The equivariant version of the decomposition theorem gives
\[\pi_*\underline{\bbC}_{\calM^a}=\bigoplus_{k\in \bbZ, \phi=(\calO_\phi\subset \calN^a,\chi_\phi)}L_\phi(k)\otimes \IC_\phi[k].\]
Let $L_\phi=\oplus_k L_\phi(k)$. For any $x\in \calO_\phi$, the pullback via \eqref{UtoH} of the $H^{BM}_*(\calZ^a,\bbC)$-module $H_*(\calM_x)_\phi$ is called a {\em standard module} of $\tUi_t$. We also view $L_\phi$ as a $\tUi_t$-module this way. Then the above results give the following proposition.

\begin{theorem}
\label{thm:simplestandard}
    Assume that $t$ is not a root of unity.
    \begin{enumerate}
        \item The nonzero module $L_\phi$ is a simple $\tUi_t$-module.
        \item For any $\phi=(\calO_\phi,\chi_\phi)$ and $\psi=(\calO_\psi,\chi_\psi)$ and $x\in \calO_\phi$, 
        \[[H_*(\calM^a_x)_\phi: L_\psi]=\sum_k\dim H^k(i_x^!\IC_\psi)_\phi.\]
    \end{enumerate}
\end{theorem}
\begin{proof}
    The simplicity of $L_\phi$ follows from Theorem \ref{thm:surj}. The rest follows from Proposition \ref{prop:mult}.
\end{proof}

\appendix

\section{Verification of Serre relations}
\label{App:A}

In this appendix, we verify the Serre relations \eqref{iDR Serre0}--\eqref{iDR Serre1} for the corresponding operators in  Theorem \ref{thm:polyrepeven}, completing the proof of this theorem in Section~\ref{sec:proof}.

\subsection{Serre relations \eqref{iDR Serre0}}

The Serre relations \eqref{iDR Serre0} states that
\begin{align} \label{eq:S}
\mathbb{S}_{i,j}(w_1,w_2|z)=0, \textit{\quad for } c_{ij}=-1, j\neq \tau i\neq i. 
\end{align}
The relation \eqref{eq:S} under the additional assumption that $(i,j)\neq (n-1,n)$ can be verified just as in \cite{V98}. 

The remaining case of \eqref{eq:S} when $(i,j) = (n-1,n)$ can be checked in the same way as the Serre relation \eqref{iDR Serre1} treated in the next subsection.

\subsection{Serre relations \eqref{iDR Serre1}}

We shall verify the following Serre relation:
\begin{align}\label{equ:checkserre}
	&(z-qw_1)(z-qw_2)\mathbb{S}_{n,j}(w_1,w_2|z)=\K_n\bDel(w_1w_2)\\
	&\times\frac{z(qw_1-q^{-1}w_2)(q^{-1}w_1-qw_2)}{w_1-w_2} \big(\bTh_n(w_2)\bB_{j}(z)-\bTh_n(w_1)\bB_{j}(z)\big),\textit{ if } c_{nj}=-1.\notag
\end{align}
The assumption $c_{nj}=-1$ means that $j=n\pm 1$. We check below the case for $j=n-1$, skipping the completely analogous case for $j=n+1$.

We first note that
\begin{align}  \label{eq:delta2}
\begin{split}
\delta(w_1x_j^{-1}q^n)\delta(w_2x_jq^n)\theta_1(qx_j^2) 
&= \delta(w_1x_j^{-1}q^n)\delta(w_2x_jq^n)\theta_1(qw_2/w_1)
\\ 
\delta(w_1x_jq^n)\delta(w_2x_j^{-1}q^n)\theta_1(qx_j^2)
&= \delta(w_1x_jq^n)\delta(w_2x_j^{-1}q^n)\theta_1(qw_1/w_2),
\end{split}
\end{align} 
and 
\begin{align} \label{eq:theta}
&\theta_1(qw_2/w_1)-[2]\theta_1(qw_2/w_1)\theta_1(z/w_2)+\theta_1(qw_2/w_1)\theta_1(z/w_1)\theta_1(z/w_2)\\
&= -\theta_1(qw_1/w_2)+[2]\theta_1(qw_1/w_2)\theta_1(z/w_1)-\theta_1(qw_1/w_2)\theta_1(z/w_1)\theta_1(z/w_2) \notag \\
&= \frac{z(q^2-1)(q^{-1}w_1-qw_2)(qw_1-q^{-1}w_2)}{(w_1-w_2)(z-qw_1)(z-qw_2)}. \notag
\end{align}

By definition, we have
\begin{align*}
&(\hat{B}_n(w_1)\hat{B}_n(w_2)\hat{E}_{n-1}(z)f)(x_{[\bfv]^\fc})\\
&= \sum_{j\in [\bfv]_n}\delta(w_1x_jq^n)\theta_1(qx_j^2)\Phi_{[\bfv]_n\setminus\{j\}}(qx_j)\cdot (\hat{B}_n(w_2)\hat{E}_{n-1}(z)f)(x_{\iota_j[\bfv]^\fc})\\
&= \sum_{j\in [\bfv]_n}\sum_{k\in [\bfv]_n\setminus\{j\}}\delta(w_1x_jq^n)\theta_1(qx_j^2)\Phi_{[\bfv]_n\setminus\{j\}}(qx_j)\delta(w_2x_kq^n)\theta_1(qx_k^2)\Phi_{[\bfv]_n\setminus\{j,k\}}(qx_k)\theta_1(qx_kx_j)\\
&\qquad \cdot (\hat{E}_{n-1}(z)f)(x_{\iota_k\iota_j[\bfv]^\fc})\\
&\qquad +\sum_{j\in [\bfv]_n}\delta(w_1x_jq^n)\theta_1(qx_j^2)\Phi_{[\bfv]_n\setminus\{j\}}(qx_j)\delta(w_2x_j^{-1}q^n)\theta_1(qx_j^{-2})\Phi_{[\bfv]_n\setminus\{j\}}(qx_j^{-1})\cdot (\hat{E}_{n-1}(z)f)(x_{[\bfv]^\fc})
  \\
&= \sum_{j\neq k\in [\bfv]_n}\sum_{l\in [\bfv]_{n-1}}\delta(w_1x_jq^n)\theta_1(qx_j^2)\Phi_{[\bfv]_n\setminus\{j\}}(qx_j)\delta(w_2x_kq^n)\theta_1(qx_k^2)\Phi_{[\bfv]_n\setminus\{j,k\}}(qx_k)\theta_1(qx_kx_j)\\
&\qquad \cdot \delta(zx_lq^{n-1})\Phi_{[\bfv]_{n-1}\setminus\{l\}}(qx_l)f(x_{\tau_l^+\iota_k\iota_j[\bfv]^\fc})\\
&\qquad +\sum_{j\in [\bfv]_n}\sum_{l\in [\bfv]_{n-1}}\delta(w_1x_jq^n)\theta_1(qx_j^2)\Phi_{[\bfv]_n\setminus\{j\}}(qx_j)\delta(w_2x_j^{-1}q^n)\theta_1(qx_j^{-2})\Phi_{[\bfv]_n\setminus\{j\}}(qx_j^{-1})\\
&\qquad \cdot \delta(zx_lq^{n-1})\Phi_{[\bfv]_{n-1}\setminus\{l\}}(qx_l)f(x_{\tau_l^+[\bfv]^\fc}).
\end{align*}
We also compute that
\begin{align*}
&(\hat{B}_n(w_1)\hat{E}_{n-1}(z)\hat{B}_n(w_2)f)(x_{[\bfv]^\fc})\\
&= \sum_{j\in [\bfv]_n}\delta(w_1x_jq^n)\theta_1(qx_j^2)\Phi_{[\bfv]_n\setminus\{j\}}(qx_j)\cdot (\hat{E}_{n-1}(z)\hat{B}_n(w_2)f)(x_{\iota_j[\bfv]^\fc})\\
&= \sum_{j\in [\bfv]_n}\sum_{l\in [\bfv]_{n-1}}\delta(w_1x_jq^n)\theta_1(qx_j^2)\Phi_{[\bfv]_n\setminus\{j\}}(qx_j)\delta(zx_lq^{n-1})\Phi_{[\bfv]_{n-1}\setminus\{l\}}(qx_l)(\hat{B}_n(w_2)f)(x_{\tau_l^+\iota_j[\bfv]^\fc})\\
&= \sum_{j\in [\bfv]_n}\sum_{l\in [\bfv]_{n-1}}\sum_{k\in [\bfv]_n\cup\{l\}\setminus\{j\}}\delta(w_1x_jq^n)\theta_1(qx_j^2)\Phi_{[\bfv]_n\setminus\{j\}}(qx_j)\delta(zx_lq^{n-1})\Phi_{[\bfv]_{n-1}\setminus\{l\}}(qx_l)\\
&\qquad \cdot \delta(w_2x_kq^n)\theta_1(qx_k^2)\Phi_{[\bfv]_n\cup\{l\}\setminus\{j,k\}}(qx_k)\theta_1(qx_kx_j)f(x_{\iota_k\tau_l^+\iota_j[\bfv]^\fc})\\
&\qquad +\sum_{j\in [\bfv]_n}\sum_{l\in [\bfv]_{n-1}}\delta(w_1x_jq^n)\theta_1(qx_j^2)\Phi_{[\bfv]_n\setminus\{j\}}(qx_j)\delta(zx_lq^{n-1})\Phi_{[\bfv]_{n-1}\setminus\{l\}}(qx_l)\\
&\qquad \cdot \delta(w_2x_j^{-1}q^n)\theta_1(qx_j^{-2})\Phi_{[\bfv]_n\cup\{l\}\setminus\{j\}}(qx_j^{-1}) f(x_{\tau_l^+[\bfv]^\fc}).
\end{align*}
Furthermore, we have that 
\begin{align*}
&(\hat{E}_{n-1}(z)\hat{B}_n(w_1)\hat{B}_n(w_2)f)(x_{[\bfv]^\fc})
\\
&= \sum_{l\in [\bfv]_{n-1}}\delta(zx_lq^{n-1})\Phi_{[\bfv]_{n-1}\setminus\{l\}}(qx_l)(\hat{B}_n(w_1)\hat{B}_n(w_2)f)(x_{\tau_l^+[\bfv]^\fc})
\\
&= \sum_{l\in [\bfv]_{n-1}}\sum_{j\in [\bfv]_n\cup\{l\}}\delta(zx_lq^{n-1})\Phi_{[\bfv]_{n-1}\setminus\{l\}}(qx_l) \\
&\qquad\qquad\qquad\qquad \cdot \delta(w_1x_jq^n)\theta_1(qx_j^2)\Phi_{[\bfv]_n\cup\{l\}\setminus\{j\}}(qx_j)(\hat{B}_n(w_2)f)(x_{\iota_j\tau_l^+[\bfv]^\fc})
\\
&= \sum_{l\in [\bfv]_{n-1}}\sum_{j\in [\bfv]_n\cup\{l\}}\sum_{k\in [\bfv]_n\cup\{l\}\setminus\{j\}}\delta(zx_lq^{n-1})\Phi_{[\bfv]_{n-1}\setminus\{l\}}(qx_l)\delta(w_1x_jq^n)\theta_1(qx_j^2)\Phi_{[\bfv]_n\cup\{l\}\setminus\{j\}}(qx_j)\\
&\qquad\qquad\qquad\qquad \cdot \delta(w_2x_kq^n)\theta_1(qx_k^2)\Phi_{[\bfv]_n\cup\{l\}\setminus\{j,k\}}(qx_k)\theta_1(qx_kx_j)f(x_{\iota_k\iota_j\tau_l^+[\bfv]^\fc})\\
&\quad +\sum_{l\in [\bfv]_{n-1}}\sum_{j\in [\bfv]_n\cup\{l\}}\delta(zx_lq^{n-1})\Phi_{[\bfv]_{n-1}\setminus\{l\}}(qx_l)\delta(w_1x_jq^n)\theta_1(qx_j^2)\Phi_{[\bfv]_n\cup\{l\}\setminus\{j\}}(qx_j)\\
&\qquad\qquad\qquad\qquad \cdot \delta(w_2x_j^{-1}q^n)\theta_1(qx_j^{-2})\Phi_{[\bfv]_n\cup\{l\}\setminus\{j\}}(qx_j^{-1}) f(x_{\tau_l^+[\bfv]^\fc}).
\end{align*}

Each of the formulas above for $\hat{B}_n(w_1)\hat{B}_n(w_2)\hat{E}_{n-1}(z), \hat{B}_n(w_1)\hat{E}_{n-1}(z)\hat{B}_n(w_2),
$ and $\hat{E}_{n-1}(z)\hat{B}_n(w_1)\hat{B}_n(w_2)$ consists of two big summands. The corresponding linear combination of the first summands in  
\begin{align*}
\bigg(\Big( \hat{B}_n(w_1)\hat{B}_n(w_2)\hat{E}_{n-1}(z) &-[2]\hat{B}_n(w_1)\hat{E}_{n-1}(z)\hat{B}_n(w_2)
\\
&  +\hat{E}_{n-1}(z)\hat{B}_n(w_1)\hat{B}_n(w_2) +(w_1\leftrightarrow w_2)\Big)f\bigg)(x_{[\bfv]^\fc})
\end{align*}
is given by
\begin{align*}
&\sum_{j\neq k\in [\bfv]_n}\sum_{l\in [\bfv]_{n-1}}\delta(w_1x_jq^n)\delta(w_2x_kq^n)\delta(zx_lq^{n-1})\theta_1(qx_j^2)\theta_1(qx_k^2)\theta_1(qx_kx_j)  \\
&\qquad\cdot f(x_{\tau_l^+\iota_k\iota_j[\bfv]^\fc}) \Phi_{[\bfv]_{n-1}\setminus\{l\}}(qx_l)\Phi_{[\bfv]_n\setminus\{j\}}(qx_j)\Phi_{[\bfv]_n\setminus\{j,k\}}(qx_k)\\
&+ \sum_{j\neq k\in [\bfv]_n}\sum_{l\in [\bfv]_{n-1}}\delta(w_1x_jq^n)\delta(w_2x_kq^n)\delta(zx_lq^{n-1})\theta_1(qx_j^2)\theta_1(qx_k^2)\theta_1(qx_kx_j) \\
&\qquad\cdot f(x_{\tau_l^+\iota_k\iota_j[\bfv]^\fc}) \Phi_{[\bfv]_{n-1}\setminus\{l\}}(qx_l)\Phi_{[\bfv]_n\setminus\{k\}}(qx_k)\Phi_{[\bfv]_n\setminus\{j,k\}}(qx_j)\\
& -[2]\sum_{j\in [\bfv]_n}\sum_{l\in [\bfv]_{n-1}}\sum_{k\in [\bfv]_n\cup\{l\}\setminus\{j\}}\delta(w_1x_jq^n)\delta(w_2x_kq^n)\delta(zx_lq^{n-1})\theta_1(qx_j^2)\theta_1(qx_k^2)\theta_1(qx_kx_j)  \\
&\qquad\cdot f(x_{\iota_k\tau_l^+\iota_j[\bfv]^\fc}) \Phi_{[\bfv]_{n-1}\setminus\{l\}}(qx_l) \Phi_{[\bfv]_n\setminus\{j\}}(qx_j)\Phi_{[\bfv]_n\cup\{l\}\setminus\{j,k\}}(qx_k)\\
& -[2]\sum_{k\in [\bfv]_n}\sum_{l\in [\bfv]_{n-1}}\sum_{j\in [\bfv]_n\cup\{l\}\setminus\{k\}}\delta(w_1x_jq^n)\delta(w_2x_kq^n)\delta(zx_lq^{n-1})\theta_1(qx_j^2)\theta_1(qx_k^2)\theta_1(qx_kx_j)  \\
&\qquad\cdot f(x_{\iota_j\tau_l^+\iota_k[\bfv]^\fc}) \Phi_{[\bfv]_{n-1}\setminus\{l\}}(qx_l) \Phi_{[\bfv]_n\setminus\{k\}}(qx_k)\Phi_{[\bfv]_n\cup\{l\}\setminus\{j,k\}}(qx_j)\\
&+ \sum_{l\in [\bfv]_{n-1}}\sum_{j\in [\bfv]_n\cup\{l\}}\sum_{k\in [\bfv]_n\cup\{l\}\setminus\{j\}}\delta(w_1x_jq^n)\delta(w_2x_kq^n)\delta(zx_lq^{n-1})\theta_1(qx_j^2)\theta_1(qx_k^2)\theta_1(qx_kx_j)  \\
&\qquad\cdot f(x_{\iota_k\iota_j\tau_l^+[\bfv]^\fc}) \Phi_{[\bfv]_{n-1}\setminus\{l\}}(qx_l)\Phi_{[\bfv]_n\cup\{l\}\setminus\{j\}}(qx_j)\Phi_{[\bfv]_n\cup\{l\}\setminus\{j,k\}}(qx_k)\\
&+ \sum_{l\in [\bfv]_{n-1}}\sum_{k\in [\bfv]_n\cup\{l\}}\sum_{j\in [\bfv]_n\cup\{l\}\setminus\{k\}}\delta(w_1x_jq^n)\delta(w_2x_kq^n)\delta(zx_lq^{n-1})\theta_1(qx_j^2)\theta_1(qx_k^2)\theta_1(qx_kx_j)  \\
&\qquad\cdot f(x_{\iota_j\iota_k\tau_l^+[\bfv]^\fc}) \Phi_{[\bfv]_{n-1}\setminus\{l\}}(qx_l)\Phi_{[\bfv]_n\cup\{l\}\setminus\{k\}}(qx_k)\Phi_{[\bfv]_n\cup\{l\}\setminus\{j,k\}}(qx_j).
\end{align*}
We claim that the above expression is equal to $0$.

Let us prove the claim. To than end, we compute that 
\begin{align*}
& \sum_{j\neq k\in [\bfv]_n}  \delta(w_1x_jq^n)\delta(w_2x_kq^n)\theta_1(qx_j^2)\theta_1(qx_k^2)\theta_1(qx_kx_j)  f(x_{\tau_l^+\iota_k\iota_j[\bfv]^\fc})  \\
&\qquad\cdot 
\Phi_{[\bfv]_n\setminus\{j\}}(qx_j)\Phi_{[\bfv]_n\setminus\{j,k\}}(qx_k)\\
&+\sum_{j\neq k\in [\bfv]_n}\delta(w_1x_jq^n)\delta(w_2x_kq^n)\theta_1(qx_j^2)\theta_1(qx_k^2)\theta_1(qx_kx_j)f(x_{\tau_l^+\iota_k\iota_j[\bfv]^\fc})\\
&\qquad\cdot\Phi_{[\bfv]_n\setminus\{k\}}(qx_k)\Phi_{[\bfv]_n\setminus\{j,k\}}(qx_j)\\
&-[2]\sum_{j\in [\bfv]_n}\sum_{k\in [\bfv]_n\cup\{l\}\setminus\{j\}}\delta(w_1x_jq^n)\delta(w_2x_kq^n)\theta_1(qx_j^2)\theta_1(qx_k^2)\theta_1(qx_kx_j)f(x_{\iota_k\tau_l^+\iota_j[\bfv]^\fc})\\
&\qquad\cdot \Phi_{[\bfv]_n\setminus\{j\}}(qx_j)\Phi_{[\bfv]_n\cup\{l\}\setminus\{j,k\}}(qx_k)\\
&-[2]\sum_{k\in [\bfv]_n}\sum_{j\in [\bfv]_n\cup\{l\}\setminus\{k\}}\delta(w_1x_jq^n)\delta(w_2x_kq^n)\theta_1(qx_j^2)\theta_1(qx_k^2)\theta_1(qx_kx_j)f(x_{\iota_j\tau_l^+\iota_k[\bfv]^\fc})\\
&\qquad\cdot \Phi_{[\bfv]_n\setminus\{k\}}(qx_k)\Phi_{[\bfv]_n\cup\{l\}\setminus\{j,k\}}(qx_j)\\
&+\sum_{j\in [\bfv]_n\cup\{l\}}\sum_{k\in [\bfv]_n\cup\{l\}\setminus\{j\}}\delta(w_1x_jq^n)\delta(w_2x_kq^n)\theta_1(qx_j^2)\theta_1(qx_k^2)\theta_1(qx_kx_j)f(x_{\iota_k\iota_j\tau_l^+[\bfv]^\fc})\\
&\qquad\cdot \Phi_{[\bfv]_n\cup\{l\}\setminus\{j\}}(qx_j)\Phi_{[\bfv]_n\cup\{l\}\setminus\{j,k\}}(qx_k)\\
&+\sum_{k\in [\bfv]_n\cup\{l\}}\sum_{j\in [\bfv]_n\cup\{l\}\setminus\{k\}}\delta(w_1x_jq^n)\delta(w_2x_kq^n)\theta_1(qx_j^2)\theta_1(qx_k^2)\theta_1(qx_kx_j)f(x_{\iota_j\iota_k\tau_l^+[\bfv]^\fc})\\
&\qquad\cdot \Phi_{[\bfv]_n\cup\{l\}\setminus\{k\}}(qx_k)\Phi_{[\bfv]_n\cup\{l\}\setminus\{j,k\}}(qx_j)
 \\
&= \sum_{j\neq k\in [\bfv]_n} \delta(w_1x_jq^n)\delta(w_2x_kq^n)\theta_1(qx_j^2)\theta_1(qx_k^2)\theta_1(qx_kx_j)f(x_{\tau_l^+\iota_k\iota_j[\bfv]^\fc})\\
&\quad \cdot \bigg(\Phi_{[\bfv]_n\setminus\{j\}}(qx_j)\Phi_{[\bfv]_n\setminus\{j,k\}}(qx_k)+\Phi_{[\bfv]_n\setminus\{k\}}(qx_k)\Phi_{[\bfv]_n\setminus\{j,k\}}(qx_j)\\
&\qquad -[2]\Phi_{[\bfv]_n\setminus\{j\}}(qx_j)\Phi_{[\bfv]_n\cup\{l\}\setminus\{j,k\}}(qx_k)-[2]\Phi_{[\bfv]_n\setminus\{k\}}(qx_k)\Phi_{[\bfv]_n\cup\{l\}\setminus\{j,k\}}(qx_j)\\
&\qquad +\Phi_{[\bfv]_n\cup\{l\}\setminus\{j\}}(qx_j)\Phi_{[\bfv]_n\cup\{l\}\setminus\{j,k\}}(qx_k)+\Phi_{[\bfv]_n\cup\{l\}\setminus\{k\}}(qx_k)\Phi_{[\bfv]_n\cup\{l\}\setminus\{j,k\}}(qx_j)\bigg)\\
&\quad -[2]\sum_{j\in [\bfv]_n}\delta(w_1x_jq^n)\delta(w_2x_lq^n)\theta_1(qx_j^2)\theta_1(qx_l^2)\theta_1(qx_lx_j)f(x_{\iota_l\tau_l^+\iota_j[\bfv]^\fc})\\
&\qquad\cdot \Phi_{[\bfv]_n\setminus\{j\}}(qx_j)\Phi_{[\bfv]_n\setminus\{j\}}(qx_l)\\
&\quad -[2]\sum_{k\in [\bfv]_n}\delta(w_1x_lq^n)\delta(w_2x_kq^n)\theta_1(qx_l^2)\theta_1(qx_k^2)\theta_1(qx_kx_l)f(x_{\iota_l\tau_l^+\iota_k[\bfv]^\fc})\\
&\qquad\cdot \Phi_{[\bfv]_n\setminus\{k\}}(qx_k)\Phi_{[\bfv]_n\setminus\{k\}}(qx_l)\\
&\quad +\sum_{j\in [\bfv]_n}\delta(w_1x_jq^n)\delta(w_2x_lq^n)\theta_1(qx_j^2)\theta_1(qx_l^2)\theta_1(qx_lx_j)f(x_{\iota_l\iota_j\tau_l^+[\bfv]^\fc})\\
&\qquad\cdot \Phi_{[\bfv]_n\cup\{l\}\setminus\{j\}}(qx_j)\Phi_{[\bfv]_n\setminus\{j\}}(qx_l)\\
&\quad +\sum_{k\in [\bfv]_n}\delta(w_1x_lq^n)\delta(w_2x_kq^n)\theta_1(qx_l^2)\theta_1(qx_k^2)\theta_1(qx_kx_l)f(x_{\iota_k\iota_l\tau_l^+[\bfv]^\fc})\\
&\qquad\cdot \Phi_{[\bfv]_n}(qx_l)\Phi_{[\bfv]_n\setminus\{k\}}(qx_k)\\
&\quad +\sum_{k\in [\bfv]_n}\delta(w_1x_lq^n)\delta(w_2x_kq^n)\theta_1(qx_l^2)\theta_1(qx_k^2)\theta_1(qx_kx_l)f(x_{\iota_l\iota_k\tau_l^+[\bfv]^\fc})\\
&\qquad\cdot \Phi_{[\bfv]_n\cup\{l\}\setminus\{k\}}(qx_k)\Phi_{[\bfv]_n\setminus\{k\}}(qx_l)\\
&\quad +\sum_{j\in [\bfv]_n}\delta(w_1x_jq^n)\delta(w_2x_lq^n)\theta_1(qx_j^2)\theta_1(qx_l^2)\theta_1(qx_lx_j)f(x_{\iota_j\iota_l\tau_l^+[\bfv]^\fc})
\Phi_{[\bfv]_n}(qx_l)\Phi_{[\bfv]_n\setminus\{j\}}(qx_j),
\end{align*}
which can then be rewritten using the identity $\theta_1(qx)+\theta_1(qx^{-1})=[2]$ as
\begin{align*}
&= \sum_{j\neq k\in [\bfv]_n}\delta(w_1x_jq^n)\delta(w_2x_kq^n)\theta_1(qx_j^2)\theta_1(qx_k^2)\theta_1(qx_kx_j)
\\
&\qquad\quad \cdot f(x_{\tau_l^+\iota_k\iota_j[\bfv]^\fc})\Phi_{[\bfv]_n\setminus\{j,k\}}(qx_j)\Phi_{[\bfv]_n\setminus\{j,k\}}(qx_k)\\
&\qquad\quad \cdot \bigg(\theta_1(qx_j/x_k)+\theta_1(qx_k/x_j)
-[2]\theta_1(qx_j/x_k)\theta_1(qx_k/x_l)-[2]\theta_1(qx_k/x_j)\theta_1(qx_j/x_l)\\
&\qquad\qquad + \theta_1(qx_j/x_l)\theta_1(qx_j/x_k)\theta_1(qx_k/x_l)+\theta_1(qx_k/x_l)\theta_1(qx_k/x_j)\theta_1(qx_j/x_l)\bigg)\\
&= \sum_{j\neq k\in [\bfv]_n}\delta(w_1x_jq^n)\delta(w_2x_kq^n)\theta_1(qx_j^2)\theta_1(qx_k^2)\theta_1(qx_kx_j) \\
&\qquad\quad \cdot f(x_{\tau_l^+\iota_k\iota_j[\bfv]^\fc})\Phi_{[\bfv]_n\setminus\{j,k\}}(qx_j)\Phi_{[\bfv]_n\setminus\{j,k\}}(qx_k)\\
&\qquad\quad \cdot \bigg([2]-\theta_1(qx_l/x_j)\theta_1(qx_j/x_k)\theta_1(qx_k/x_l)-\theta_1(qx_l/x_k)\theta_1(qx_k/x_j)\theta_1(qx_j/x_l)\bigg)\\
&= 0.
\end{align*}
This proves the claim. 

Summarizing, keeping in mind the desired Serre relation \eqref{equ:checkserre}, we continue to compute that
\begin{align*}
&\bigg(\Big(\hat{B}_n(w_1)\hat{B}_n(w_2)\hat{E}_{n-1}(z)-[2]\hat{B}_n(w_1)\hat{E}_{n-1}(z)\hat{B}_n(w_2)+\hat{E}_{n-1}(z)\hat{B}_n(w_1)\hat{B}_n(w_2)\\
&\qquad +(w_1\leftrightarrow w_2)\Big)f\bigg) (x_{[\bfv]^\fc})
  \\
&= 
\sum_{l\in [\bfv]_{n-1}}\delta(zx_lq^{n-1})\Phi_{[\bfv]_{n-1}\setminus\{l\}}(qx_l)f(x_{\tau_l^+[\bfv]^\fc})\\
& \cdot \bigg(\sum_{j\in [\bfv]_n}\delta(w_1x_jq^n)\delta(w_2x_j^{-1}q^n)\theta_1(qx_j^2)\theta_1(qx_j^{-2})\Phi_{[\bfv]_n\setminus\{j\}}(qx_j^{-1}) \Phi_{[\bfv]_n\setminus\{j\}}(qx_j)\\
&\quad +\sum_{j\in [\bfv]_n}\delta(w_1x_j^{-1}q^n)\delta(w_2x_jq^n)\theta_1(qx_j^2)\theta_1(qx_j^{-2})\Phi_{[\bfv]_n\setminus\{j\}}(qx_j)\Phi_{[\bfv]_n\setminus\{j\}}(qx_j^{-1})\\
&\quad -[2]\sum_{j\in [\bfv]_n}\delta(w_1x_jq^n)\delta(w_2x_j^{-1}q^n)\theta_1(qx_j^2)\theta_1(qx_j^{-2})\Phi_{[\bfv]_n\setminus\{j\}}(qx_j)\Phi_{[\bfv]_n\cup\{l\}\setminus\{j\}}(qx_j^{-1})\\
&\quad -[2]\sum_{j\in [\bfv]_n}\delta(w_2x_jq^n)\delta(w_1x_j^{-1}q^n)\theta_1(qx_j^2)\theta_1(qx_j^{-2})\Phi_{[\bfv]_n\setminus\{j\}}(qx_j)\Phi_{[\bfv]_n\cup\{l\}\setminus\{j\}}(qx_j^{-1})\\
&\quad +\sum_{j\in [\bfv]_n\cup\{l\}}\delta(w_1x_jq^n)\delta(w_2x_j^{-1}q^n)\theta_1(qx_j^2)\theta_1(qx_j^{-2})\Phi_{[\bfv]_n\cup\{l\}\setminus\{j\}}(qx_j)\Phi_{[\bfv]_n\cup\{l\}\setminus\{j\}}(qx_j^{-1})\\
&\quad +\sum_{j\in [\bfv]_n\cup\{l\}}\delta(w_1x_j^{-1}q^n)\delta(w_2x_jq^n)\theta_1(qx_j^2)\theta_1(qx_j^{-2})\Phi_{[\bfv]_n\cup\{l\}\setminus\{j\}}(qx_j)\Phi_{[\bfv]_n\cup\{l\}\setminus\{j\}}(qx_j^{-1})\bigg),
\end{align*}
which can be rewritten using the identities \eqref{eq:delta2}
as follows:
\begin{align*}
& =
\sum_{l\in [\bfv]_{n-1}}\delta(zx_lq^{n-1})\Phi_{[\bfv]_{n-1}\setminus\{l\}}(qx_l)f(x_{\tau_l^+[\bfv]^\fc})\\
&\cdot \bigg(\theta_1(qw_2/w_1)\sum_{j\in [\bfv]_n}\delta(w_1x_jq^n)\delta(w_2x_j^{-1}q^n)\Phi_{[\bfv]_n}(qx_j^{-1}) \Phi_{[\bfv]_n\setminus\{j\}}(qx_j)\\
&\quad +\theta_1(qw_1/w_2)\sum_{j\in [\bfv]_n}\delta(w_1x_j^{-1}q^n)\delta(w_2x_jq^n)\Phi_{[\bfv]_n\setminus\{j\}}(qx_j)\Phi_{[\bfv]_n}(qx_j^{-1})\\
&\quad -[2]\theta_1(qw_2/w_1)\sum_{j\in [\bfv]_n}\delta(w_1x_jq^n)\delta(w_2x_j^{-1}q^n)\Phi_{[\bfv]_n\setminus\{j\}}(qx_j)\Phi_{[\bfv]_n\cup\{l\}}(qx_j^{-1})\\
&\quad -[2]\theta_1(qw_1/w_2)\sum_{j\in [\bfv]_n}\delta(w_2x_jq^n)\delta(w_1x_j^{-1}q^n)\Phi_{[\bfv]_n\setminus\{j\}}(qx_j)\Phi_{[\bfv]_n\cup\{l\}}(qx_j^{-1})\\
&\quad +\theta_1(qw_2/w_1)\sum_{j\in [\bfv]_n\cup\{l\}}\delta(w_1x_jq^n)\delta(w_2x_j^{-1}q^n)\Phi_{[\bfv]_n\cup\{l\}\setminus\{j\}}(qx_j)\Phi_{[\bfv]_n\cup\{l\}}(qx_j^{-1})\\
&\quad +\theta_1(qw_1/w_2)\sum_{j\in [\bfv]_n\cup\{l\}}\delta(w_1x_j^{-1}q^n)\delta(w_2x_jq^n)\Phi_{[\bfv]_n\cup\{l\}\setminus\{j\}}(qx_j)\Phi_{[\bfv]_n\cup\{l\}}(qx_j^{-1})\bigg)
  \\
&= \sum_{l\in [\bfv]_{n-1}}\delta(zx_lq^{n-1})\Phi_{[\bfv]_{n-1}\setminus\{l\}}(qx_l)f(x_{\tau_l^+[\bfv]^\fc})\\
&\cdot\bigg(\theta_1(qw_2/w_1)\sum_{j\in [\bfv]_n}\delta(w_1x_jq^n)\delta(w_2x_j^{-1}q^n)\Phi_{[\bfv]_n}(qx_j^{-1}) \Phi_{[\bfv]_n\setminus\{j\}}(qx_j)\\
&\quad +\theta_1(qw_1/w_2)\sum_{j\in [\bfv]_n}\delta(w_1x_j^{-1}q^n)\delta(w_2x_jq^n)\Phi_{[\bfv]_n\setminus\{j\}}(qx_j)\Phi_{[\bfv]_n}(qx_j^{-1})\\
&\quad -[2]\theta_1(qw_2/w_1)\theta_1(z/w_2)\sum_{j\in [\bfv]_n}\delta(w_1x_jq^n)\delta(w_2x_j^{-1}q^n)\Phi_{[\bfv]_n\setminus\{j\}}(qx_j)\Phi_{[\bfv]_n}(qx_j^{-1})\\
&\quad -[2]\theta_1(qw_1/w_2)\theta_1(z/w_1)\sum_{j\in [\bfv]_n}\delta(w_2x_jq^n)\delta(w_1x_j^{-1}q^n)\Phi_{[\bfv]_n\setminus\{j\}}(qx_j)\Phi_{[\bfv]_n}(qx_j^{-1})\\
&\quad +\theta_1(qw_2/w_1)\theta_1(z/w_1)\theta_1(z/w_2)\sum_{j\in [\bfv]_n}\delta(w_1x_jq^n)\delta(w_2x_j^{-1}q^n)\Phi_{[\bfv]_n\setminus\{j\}}(qx_j)\Phi_{[\bfv]_n}(qx_j^{-1})\\
&\quad +\theta_1(qw_1/w_2)\theta_1(z/w_1)\theta_1(z/w_2)\sum_{j\in [\bfv]_n}\delta(w_1x_j^{-1}q^n)\delta(w_2x_jq^n)\Phi_{[\bfv]_n\setminus\{j\}}(qx_j)\Phi_{[\bfv]_n}(qx_j^{-1})\\
&\quad + \theta_1(qw_2/w_1)\delta(w_1x_lq^n)\delta(w_2x_l^{-1}q^n)\Phi_{[\bfv]_n}(qx_l)\Phi_{[\bfv]_n}(qx_l^{-1})\theta_1(qx_l^{-2})\\
&\quad + \theta_1(qw_1/w_2)\delta(w_1x_l^{-1}q^n)\delta(w_2x_lq^n)\Phi_{[\bfv]_n}(qx_l)\Phi_{[\bfv]_n}(qx_l^{-1})\theta_1(qx_l^{-2})\bigg),
\end{align*}
which can be further rewritten using \eqref{eq:theta} as 
\begin{align*}
&= (\hat{E}_{n-1}(z)f)(x_{[\bfv]^\fc})\bigg(\theta_1(qw_2/w_1)-[2]\theta_1(qw_2/w_1)\theta_1(z/w_2)+\theta_1(qw_2/w_1)\theta_1(z/w_1)\theta_1(z/w_2)\bigg)\\
&\bigg(\sum_{j\in [\bfv]_n}\delta(w_1x_jq^n)\delta(w_2x_j^{-1}q^n)\Phi_{[\bfv]_n}(qx_j^{-1}) \Phi_{[\bfv]_n\setminus\{j\}}(qx_j)\\
&-\sum_{j\in [\bfv]_n}\delta(w_1x_j^{-1}q^n)\delta(w_2x_jq^n)\Phi_{[\bfv]_n\setminus\{j\}}(qx_j)\Phi_{[\bfv]_n}(qx_j^{-1})\bigg)\\
&+\sum_{l\in [\bfv]_{n-1}}\delta(zx_lq^{n-1})\Phi_{[\bfv]_{n-1}\setminus\{l\}}(qx_l)f(x_{\tau_l^+[\bfv]^\fc})\\
&\bigg(\theta_1(qw_2/w_1)\delta(w_1x_lq^n)\delta(w_2x_l^{-1}q^n)\Phi_{[\bfv]_n}(qx_l)\Phi_{[\bfv]_n}(qx_l^{-1})\theta_1(qx_l^{-2})\\
&+\theta_1(qw_1/w_2)\delta(w_1x_l^{-1}q^n)\delta(w_2x_lq^n)\Phi_{[\bfv]_n}(qx_l)\Phi_{[\bfv]_n}(qx_l^{-1})\theta_1(qx_l^{-2})\bigg)
  \\
&= 
(\hat{E}_{n-1}(z)f)(x_{[\bfv]^\fc})\bigg(\theta_1(qw_2/w_1)-[2]\theta_1(qw_2/w_1)\theta_1(z/w_2)+\theta_1(qw_2/w_1)\theta_1(z/w_1)\theta_1(z/w_2)\bigg)\\
&\frac{\bDel(w_1w_2)}{q-q^{-1}}\bigg(\Phi_{[\bfv]_n}(q^{1-n}w_1^{-1})\Phi_{[\bfv]_n}(q^{1+n}w_1)-\Phi_{[\bfv]_n}(q^{1+n}w_2)\Phi_{[\bfv]_n}(q^{1-n}w_2^{-1})\bigg)\\
&+ \sum_{l\in [\bfv]_{n-1}}\delta(zx_lq^{n-1})\Phi_{[\bfv]_{n-1}\setminus\{l\}}(qx_l)f(x_{\tau_l^+[\bfv]^\fc})\\
&\bigg(\theta_1(qw_2/w_1)\delta(w_1x_lq^n)\delta(w_2x_l^{-1}q^n)\Phi_{[\bfv]_n}(qx_l)\Phi_{[\bfv]_n}(qx_l^{-1})\theta_1(qx_l^{-2}) \\
&+\theta_1(qw_1/w_2)\delta(w_1x_l^{-1}q^n)\delta(w_2x_lq^n)\Phi_{[\bfv]_n}(qx_l)\Phi_{[\bfv]_n}(qx_l^{-1})\theta_1(qx_l^{-2})\bigg),
\end{align*}
which can be shown using \eqref{eq:theta} again to be equal to
\begin{align*}
&= 
-q\bDel(w_1w_2)\frac{z(q^{-1}w_1-qw_2)(qw_1-q^{-1}w_2)}{(w_1-w_2)(z-qw_1)(z-qw_2)}(\hat{\bTh}_n(w_2)\hat{E}_{n-1}(z)f-\hat{\bTh}_n(w_1)\hat{E}_{n-1}(z)f)(x_{[\bfv]^\fc})\\
&+\sum_{l\in [\bfv]_{n-1}}\delta(zx_lq^{n-1})\Phi_{[\bfv]_{n-1}\setminus\{l\}}(qx_l)f(x_{\tau_l^+[\bfv]^\fc}) \times \\
&\qquad\quad \bigg(\theta_1(qw_2/w_1)\delta(w_1x_lq^n)\delta(w_2x_l^{-1}q^n)\Phi_{[\bfv]_n}(qx_l)\Phi_{[\bfv]_n}(qx_l^{-1})\theta_1(qx_l^{-2})\\
&\qquad\qquad +\theta_1(qw_1/w_2)\delta(w_1x_l^{-1}q^n)\delta(w_2x_lq^n)\Phi_{[\bfv]_n}(qx_l)\Phi_{[\bfv]_n}(qx_l^{-1})\theta_1(qx_l^{-2})\bigg).
\end{align*}

Multiplying both sides by $(z-qw_1)(z-qw_2)$ will kill the extra terms
\begin{align*}
	\sum_{l\in [\bfv]_{n-1}} &\delta(zx_lq^{n-1})\Phi_{[\bfv]_{n-1}\setminus\{l\}}(qx_l)f(x_{\tau_l^+[\bfv]^\fc}) \times \\
	&\bigg(\theta_1(qw_2/w_1)\delta(w_1x_lq^n)\delta(w_2x_l^{-1}q^n)\Phi_{[\bfv]_n}(qx_l)\Phi_{[\bfv]_n}(qx_l^{-1})\theta_1(qx_l^{-2})\\
	&\quad +\theta_1(qw_1/w_2)\delta(w_1x_l^{-1}q^n)\delta(w_2x_lq^n)\Phi_{[\bfv]_n}(qx_l)\Phi_{[\bfv]_n}(qx_l^{-1})\theta_1(qx_l^{-2})\bigg).
\end{align*}
Thus, we have established the Serre relation \eqref{equ:checkserre}.

\bibliographystyle{alpha}
\bibliography{k.bib}

\end{document}